\documentclass[11 pt,a4paper,twoside,reqno]{amsart} 
\usepackage{amsfonts,amssymb,amscd,amsmath,enumerate,verbatim,calc}
\usepackage{amsmath,amssymb,amsfonts, mathrsfs, amsthm}
\usepackage{amscd, amssymb, latexsym, amsmath, amscd}
\usepackage[normalem]{ulem}

\usepackage{amsmath}
\usepackage{amsfonts}
\usepackage{amssymb}
\usepackage{setspace}
\usepackage{amsthm,mathrsfs,textcomp}
\usepackage[utf8]{inputenc}	
\usepackage[T1]{fontenc}	
\usepackage{mathtools}
\usepackage[super]{nth}

\usepackage[all]{xy}
\usepackage{pb-diagram}
\usepackage{enumerate}
\usepackage[hidelinks]{hyperref}
\usepackage[capitalize]{cleveref} 
\usepackage{tikz, tikz-cd} 
\usetikzlibrary{positioning}
\usetikzlibrary{calc, arrows, decorations.markings} \tikzset{>=latex}
\usepackage{parskip, verbatim}
\usepackage{microtype,lmodern} 
\usepackage[utf8]{inputenc} 
\usepackage[T1]{fontenc} 
\usepackage{enumerate,comment,braket,xspace,csquotes} 
\usepackage{multirow}
\usepackage{stmaryrd}

\usepackage{xcolor}
\hypersetup{
	colorlinks,
	linkcolor={red!50!black},
	citecolor={blue!50!black},
	urlcolor={blue!80!black}
}

\input xy
\xyoption{all}
\usepackage{pb-diagram}
\usepackage[all]{xy}
\input xy
\xyoption{all}

\newcommand{\rH}{\mathrm H}

\newcommand{\fm}{\mathfrak m}

\newcommand{\fs}{\mathfrak s}

\newcommand{\cA}{\mathcal A}
\newcommand{\cB}{\mathcal B}
\newcommand{\cC}{\mathcal C}
\newcommand{\cD}{\mathcal D}

\newcommand{\cF}{\mathcal F}
\newcommand{\cH}{\mathcal H}
\newcommand{\cG}{\mathcal G}

\newcommand{\cM}{\mathcal M}

\newcommand{\cO}{\mathcal O}

\newcommand{\cR}{\mathcal R}
\newcommand{\cS}{\mathcal S}

\newcommand{\cX}{\mathcal X}

\newcommand{\cZ}{\mathcal Z}
\newcommand{\Hecke}{\mathcal{H}}

\newcommand{\bZ}{\mathbb{Z}}

\newcommand{\bD}{\mathbb{D}}

\newcommand{\bC}{\mathbb{C}}
\newcommand{\bR}{\mathbb{R}}
\newcommand{\bS}{\mathbb{S}}

\newcommand{\bL}{\mathbb{L}}

\newcommand{\bG}{\mathbb{G}}

\newcommand{\lmod}{\mhyphen \mathbf{mod}}
\newcommand{\rmod}{\mathbf{mod}\mhyphen}
\newcommand{\rmodnd}{\mathbf{mod}\nd\mhyphen}
\newcommand{\rmodfinl}{\mathbf{mod}^\finl\mhyphen}
\newcommand{\Rep}{\operatorname{Rep}}

\newcommand{\Vect}{\operatorname{{Vect}}}

\newcommand{\bfi}{\mathbf{i}}
\newcommand{\bfr}{\mathbf{r}}

\makeatletter
\newcommand{\oset}[3][0.2ex]{%
	\mathrel{\mathop{#3}\limits^{
			\vbox to#1{\kern-2\ex@
				\hbox{$\scriptstyle#2$}\vss}}}}
\makeatother

\newcommand{\ev}{\mathsf{ev}}
\newcommand{\finl}{\mathsf{fl}}
\newcommand{\nd}{^\mathsf{nd}}

\newcommand{\tr}{\mathsf{tr}}

\newcommand{\gr}{\mathsf{gr}}

\newcommand{\inv}{^{-1}}

\newcommand{\SL}{\operatorname{SL}}

\newcommand{\wG}{{\widetilde{G}}}
\newcommand{\wP}{{\widetilde{P}}}
\newcommand{\wK}{{\widetilde{K}}}

\newcommand{\wQ}{{\widetilde{Q}}}

\newcommand{\wL}{{\widetilde{L}}}
\newcommand{\wM}{{\widetilde{M}}}
\newcommand{\wT}{{\widetilde{T}}}

\newcommand{\Go}{{G^\circ}}
\newcommand{\dGo}{{\Go\times\Go}}
\newcommand{\Lo}{{L^\circ}}

\newcommand{\wGo}{{\wG^\circ}}
\newcommand{\wLo}{{\wL^\circ}}

\newcommand{\dpi}{\pi\boxtimes\pi^\vee}
\newcommand{\dpio}{\pi_0\boxtimes\pi_0^\vee}
\newcommand{\dpis}{\pi_\sigma\boxtimes\pi_\sigma^\vee}

\DeclareMathOperator{\Ext}{Ext}
\DeclareMathOperator{\Stab}{Stab}

\DeclareMathOperator{\Hom}{Hom}
\DeclareMathOperator{\End}{End}

\DeclareMathOperator{\rank}{rank}

\DeclareMathOperator{\RHom}{RHom}

\DeclareMathOperator{\Coh}{Coh}
\DeclareMathOperator{\supp}{supp}

\DeclareMathOperator{\id}{\mathsf{id}}

\DeclareMathOperator{\Id}{\mathrm{Id}}
\newcommand{\Irr}{\operatorname{Irr}}

\newcommand{\proj}{{\operatorname{proj}}}
\newcommand{\ind}{\operatorname{ind}}
\newcommand{\Ind}{\operatorname{Ind}}
\newcommand{\Res}{\operatorname{Res}}

\newcommand{\can}{\mathsf{can}}
\newcommand{\op}{{op}}

\newcommand{\git}{/\mkern-6mu/}

\newcommand{\indu}{{\mathrm{ind}}}

\newcommand{\perf}{\mathrm{perf}}
\newcommand{\Nak}{\mathrm{Nak}}
\newcommand{\sm}{\mathrm{sm}}

\mathchardef\mhyp="2D  
\newcommand{\mhyphen}{\mhyp}

\theoremstyle{plain}
\newtheorem{theorem}{Theorem}[section]
\newtheorem*{theorem*}{Theorem}

\newtheorem{lemma}[theorem]{Lemma}
\newtheorem*{lemma*}{Lemma}
\newtheorem{cor}[theorem]{Corollary}
\newtheorem{corollary}[theorem]{Corollary}

\newtheorem{proposition}[theorem]{Proposition}

\theoremstyle{definition} \theoremstyle{definition}
\newtheorem{remark}[theorem]{Remark}
\newtheorem{example}[theorem]{Example}

\newtheorem{definition}[theorem]{Definition}

\theoremstyle{remark}

\newcommand{\cs}{{\rm cs}}
\newcommand{\I}{\mathfrak{I}}
\newcommand{\HH}{\mathcal{H}}
\numberwithin{equation}{section}
\begin{document}

	\title{Homological duality for covering groups of reductive $p$-adic groups}
	\author{Drago\c s Fr\u a\c til\u a and Dipendra Prasad}
	\date{\today}

\renewcommand\rightmark{Homological Duality for $p$-adic groups}
\renewcommand\leftmark{   Drago\c s Fr\u a\c til\u a and Dipendra Prasad       }

	\maketitle
		\begin{center}
		\emph{To Dick Gross, with admiration.}
	\end{center}	
	
		\begin{abstract}	
		In this largely expository paper we extend properties of the homological duality functor
                $\RHom_\cH(-,\cH)$
                where $\cH$ is the Hecke algebra of a
                reductive $p$-adic group, to the case where 
it is the Hecke algebra of a
finite central extension of a reductive $p$-adic group.
		The most important properties being that $\RHom_\cH(-,\cH)$ is concentrated in a single degree for irreducible representations and that it gives rise to Schneider--Stuhler duality for Ext groups (a Serre functor like property). 
		Our simple proof is self-contained and bypasses the localization techniques of \cite{SchStu,BezrPhD} improving slightly on \cite{NoriPras}.
		Along the way we also study Grothendieck--Serre duality with respect to the Bernstein center and provide a proof of the folklore result  that on admissible modules this functor is nothing else but the contragredient duality.
		We single out a necessary and sufficient condition for when these three dualities agree on finite length modules in a given block. 
		In particular, we show this is the case for all cuspidal blocks as well as, due to a result of Roche \cite{Roche-parab}, on all blocks with trivial stabilizer in the relative Weyl group.
	\end{abstract}

	\setcounter{tocdepth}{1}
	
	\tableofcontents
		
	\section{Introduction}
	\subsection{}
	Let $G$ be a reductive $p$-adic group or a covering group (finite central extension) of such a group.
	The homological duality for the abelian category  $\cM(G)$ of smooth representations of $G$
	is defined at the level of derived categories as
	\[ D_h:=\RHom_\cH(-,\cH)\colon \cD^b(\cM(G))\to \cD^b(\cM(G))^\op\]
	where $\cH$ is the Hecke algebra of $G$.
	
	An important property established by Bernstein in his Harvard notes \cite{BerNotes} is that if $\pi$ is irreducible then $D_h(\pi)$ is concentrated in a single degree (and is irreducible).
	We will denote this representation by $\bD_h(\pi)$.
	
	On the other hand, in \cite{SchStu} Schneider and Stuhler also prove this property using localization techniques on the
	Bruhat--Tits building.
	Moreover, they show  that this functor $\bD_h$ enjoys a Serre functor like property, namely that for any irreducible representation $\pi$ of $G$ and any smooth representation $\pi'$ of $G$, there is a natural isomorphism
	\begin{align}\label{Eq:intro Sch-Stuh duality}
		\Ext^i_G(\pi,\pi')^*\simeq  \Ext_G^{d(\pi)-i}(\pi',\bD_h(\pi)^\vee),
	\end{align}
	where $d(\pi)\geq 0$ is an integer  which is bounded by the split rank of $G$ that we define later. Actually in \cite{SchStu}, the isomorphism \eqref{Eq:intro Sch-Stuh duality} was proved only in the subcategory of representations with a fixed central character. 
	The more general version was proved in \cite{NoriPras} and \cite{BBK}.
	
	One of the aims of this work is to present a direct proof of \eqref{Eq:intro Sch-Stuh duality} for covering groups that does not make use of \cite{SchStu} or of localization techniques.
	The strategy is to first show that the full homological duality $D_h$ satisfies a Serre functor like property
	\[ \RHom_\cH(\pi,\pi')^*\simeq \RHom_\cH(\pi',D_h(\pi)^\vee) \]
	for any smooth representations $\pi$ and $\pi'$ with $\pi$ finitely generated.
	This is easy and follows from basic adjunctions naturally extended to idempotented algebras.
	For convenience of the reader, and for lack of a precise reference, we provide all the details in \cref{S:abstract duality theorem}.
	
	The second step, namely showing that $D_h(\pi)$ is concentrated in a single degree for an irreducible representation, is non-formal and is based on several ingredients: Bernstein decomposition, second adjointness and a special property of the algebra governing a cuspidal block (Frobenius-symmetric algebra over its center).
	Although the proof is already in \cite{BerNotes} we hope that the argument that we present is more streamlined.
	
	Another property of the homological functor $D_h$ is that on finite length cuspidal representations it agrees (up to a shift) with the contragredient functor. 
	A sketch of this result for irreducible representations appears already in \cite{BerNotes} but we were not able to supply the details so we decided to include a different proof with full details.
	To this effect we were led to study the Grothendieck--Serre duality over the Bernstein center $D_{GS/\cZ}$ (see \cref{S:dualities on finite length} for the definition).
	In particular, we provide a proof of the folklore result that $D_{GS/\cZ}$
	agrees with the contragredient for admissible representations.
	Moreover we identify a necessary and sufficient condition (see condition (FsG) in Definition \ref{FsG})
        for the two functors $D_h$ and $D_{GS/\cZ}$ to agree (up to a shift) on a Bernstein block.
	In order to state the condition, we need to introduce some notation.
	
	If $\fs=[L,\rho]$ is a cuspidal datum, up to conjugation and inertia,  then there is an equivalence of the Bernstein block $\cM(G)_\fs$ with the module category $\rmod\cR_\fs$ for some algebra $\cR_\fs$ with center $\cZ_\fs$ that can be described explicitly (see \cref{T:center of M(G)_s as invariants}).
	The condition alluded to before, which we call Frobenius-symmetric-Gorenstein, since it is reminiscent of both these properties, is
	\[ D_{GS/{\cZ_\fs}}(\cR_\fs)\simeq \cR_\fs[d] \text{ as $\cR_\fs$-bimodules}\]
	where $d=d(\fs)$ is the split rank of the center of $L$ (equals the Krull dimension of $\cZ_\fs$).
	We are able to check this condition on a cuspidal block $\fs=(\rho,G)$ because in this case $\cZ_\fs$ is a Laurent-polynomial algebra and $\cR_\fs$ is an  Azumaya algebra (see \cref{P:R_rho is Azumaya and Z Laur pol} and \cref{C:Frob sym cond for R_rho}).
	Therefore we deduce that 
	the homological duality and the Grothendieck--Serre duality agree with the contragredient duality on all finite length cuspidal representations 
	in a given block.
	Moreover, if $\fs=[L,\rho]$ is a cuspidal datum, by a result of Roche \cite[Theorem 3.1]{Roche-parab}, parabolic induction induces an equivalence of the blocks $\cM(L)_{[\rho]}\to\cM(G)_{\fs}$ if and only if the stabilizer of the inertia class $[\rho]$ in the relative Weyl group $N_G(L)/L$ is trivial, i.e., if and only if there is no non-trivial $w\in N_G(L)/L$ such that ${}^w\rho \simeq \rho\chi$ for some unramified character $\chi$ of $L$.
	We also deduce that  for these blocks,  $D_h$ and $D_{GS/\cZ}$ agree (up to a shift).
	
	There is yet another duality\footnote{It is indeed involutive but  \emph{co}variant and  so the name duality is slightly misleading.} on smooth representations, namely the Aubert--Zelevinski duality.
	Originally, it arose in representation theory of finite
	groups of Lie type where it was defined on the Grothendieck group of finite dimensional representations.
	Deligne--Lusztig introduced what is now called the Deligne--Lusztig complex in order to have a definition of the involution
	at the level of representations.
	For $p$-adic groups as well as their covering groups, the involution so constructed
	is usually called the Aubert--Zelevinski involution.
	It can be proved (\cite{Aub,BBK,BezrPhD}) that the Aubert--Zelevinski involution defines an exact functor
	from the abelian category $\cM(G)$ to itself, hence it
	extends trivially to a functor on the derived category $\cD^b(\cM(G))$:
	\[D_{AZ}\colon \cD^b(\cM(G))\to \cD^b(\cM(G)).\]

	The same involution was considered by Schneider--Stuhler in \cite{SchStu} and by Bezrukavnikov in \cite[Theorem 4.13]{BezrPhD} where the following isomorphism of functors was  proved using localization techniques 
	\begin{align}\label{intro:D_AZ D_GS=D_h}
		D_{AZ}\circ D_{GS/\cZ}\simeq D_h.
	\end{align}
	Actually, in \cite{SchStu} this was shown only for finite length modules and in the Grothendieck group whereas in \cite{BezrPhD} the more general result was established.
	The isomorphism \eqref{intro:D_AZ D_GS=D_h} was recently revisited in \cite{BBK} where a simpler proof was given using the geometry of the wonderful compactification.
	
	At the level of Grothendieck groups, it was  shown in \cite{Aub} that $D_{AZ}$ commutes with the contragredient and the  commutation formulae with the parabolic induction and restriction functors were proved. 
	None of these are known for representations, say of finite length, of covering groups,
	as even for linear groups, all these
	assertions are proved by relating Aubert--Zelevinski involution to the homological duality (see \eqref{intro:D_AZ D_GS=D_h}) which is not yet
	 available for covering groups.
	 
	Results similar to Aubert's were proved in the context of smooth representations of covering groups in the work of Ban--Jantzen \cite{DubJan-I,DubJan-II} (again for the Grothendieck group of finite length representations). 
	In this work we do not address the question of whether \eqref{intro:D_AZ D_GS=D_h} holds for covering groups nor do we study the Aubert--Zelevinsky involution in this context.
	An immediate obstacle to generalizing the approach of \cite{BBK} is that as
	covering groups are no longer linear,  we do not have at our disposal an obvious wonderful compactification.
	However, in favor of the isomorphism \eqref{intro:D_AZ D_GS=D_h} is the fact that $[d]\circ D_h=D_{GS/\cZ}$ on cuspidal blocks, a result that we prove in \cref{S:dualities on finite length}.
	By the aforementioned result of Roche, this continues to be the case for blocks $\fs=[L,\rho]$ such that the stabilizer of the inertia class of $\rho$ is trivial in the relative Weyl group $N_G(L)/L$.
	
	One reason for interest in homological duality, and the attendant Serre functor property, is in the context of
	Ext analogues of branching laws, as discussed in \cite{Pra18}. Recall that usually the branching laws in the context of $p$-adic group are considered for a representation
	$\pi_1$ of a reductive group $G$  to a representation $\pi_2$ of a closed subgroup $H$, as $\Hom_H(\pi_1,\pi_2)$ (and never as
	$\Hom_H(\pi_2,\pi_1)$). It has been suggested
	in \cite{Pra18} that the success of these branching laws in the various cases studied stems from the vanishing of $\Ext^i_H(\pi_1,\pi_2), i>0$,
	whereas by the Serre functor property,  $\Ext^i_H(\pi_2,\pi_1) = 0, i< d(\pi_2)$.
	In particular,  $\Hom_H(\pi_2,\pi_1)$ is typically zero (so no wonder it is never considered!), and shows up only through the higher extension groups, i.e.,  $\Ext^{d(\pi_2)}_H(\pi_2,\pi_1)$.

	A big chunk of the paper is expository: we give an exposition of some basic representation theoretic results in this context to show that they generalize in a  naive way to the case of covering groups culminating with Bernstein's decomposition (following closely the linear setting). We include a short discussion of Grothendieck--Serre duality limiting ourselves to the context that is sufficient for the applications we have in mind.
	The reader familiar with the above classical results should directly consult \S\ref{S:homological duality}, \S\ref{S:duality SchSt}, \S\ref{S:dualities on finite length}.
	
	\subsection{} 
	Below follows a more precise description of our work.
	Let $G=\bG(F)$ be a $p$-adic reductive group and let $\wG$ be a finite covering group of $G$.
	Denote by $\cH$ the Hecke algebra of $\wG$ and by $\cM(\wG)$ the category of smooth complex representations of $\wG$.
	The Levi subgroups and parabolic subgroups of $\wG$ are, by definition, the pullbacks of those of $G$.
	We have parabolic induction and restriction functors (Jacquet modules)  defined as in the linear case.
	A representation is called cuspidal if it is killed by all proper parabolic restriction functors.
	
	We denote by $\cB(\wG)$ the equivalence classes (conjugation and inertia) of pairs $(\wL,\rho)$ of a Levi subgroup $\wL$  and an irreducible cuspidal representation $\rho$ of $\wL$.
	
	The following is the Bernstein decomposition for $\wG$ (see \cite{BerNotes,BerDel,Renard} for the linear case): 
	\begin{theorem} \label{T:Bernstein dec-intro}
		We have a block decomposition
		\[ \cM(\wG)\simeq \prod_{\fs\in \cB(\wG)} \cM(\wG)_\fs. \]
		Moreover, each block $\cM(\wG)_\fs$ is equivalent to the category of modules over some $\bC$-algebra $\cR_\fs$ containing a finitely generated commutative subalgebra over which it is finite.
	\end{theorem}
	We will outline the proof in \cref{S:Bernstein dec} which follows closely the linear case as exposed in \cite{BerNotes} or \cite{Renard}.
	To that end, and to fix the notation, we will collect in Sections 1, 2, 3, 4 all the necessary results from the classical theory (i.e., the linear case) as well as some rudiments from category theory that go into the proof of it.
	This is meant to convince the reader that the same strategy as in the linear case gives also the Bernstein decomposition for covering groups.
	
	Using the second adjointness theorem of Bernstein, we prove the equivalence of
        $\cM(\wG)_\fs$ with $\rmod\cR_\fs$ in \cref{S:Blocks as module cats}.
        For completeness, we also give in \cref{S:second adj},  a skeleton of the proof of the second adjointness theorem
        which is meant to convince the reader that the proof from the linear case (for which the first complete proof was
        published by Bushnell in		\cite{Bush-loc})         goes through without changes for covering groups. 
        In doing this, we have followed the exposition in \cite[VI.9.7]{Renard} of the linear case which makes use of completion functors in order to streamline the arguments.

	The argument we present follows \cite{Renard} rather than \cite{BerNotes} in that it uses completion
        of modules in order to streamline the proof.
	We take advantage of this opportunity to introduce the completion functors in \cref{SS:completion} and we show that (at least on admissible representations) it allows one to recover the lost properties of the invariants $(-)^N$ functor.
	See also \cref{C:invar and coinvar are isomorphic} for a far more advanced version (essentially equivalent to the second adjointness theorem).
	
	For an irreducible representation $\pi$, the block $\fs=[\wL,\rho]$ that corresponds to $\pi$
	is called  the cuspidal support of $\pi$. We denote by $d(\pi)=d(\fs)=d(\wL)$, the dimension of a maximal split torus in the center of the algebraic group $\bL$. 
	
	The category of (smooth) representations of $\wG$ has finite global dimension (see \cref{SS:finite homological dimension}).
	Consider the homological duality functor $D_h$ on the bounded derived category of $\wG$-modules
	\[ D_h:=\RHom_\cH(-,\cH)\colon \cD^b(\cM(\wG))\to \cD^b(\cM(\wG))^\op. \]
	It is surprisingly easy to show that\footnote{In the context of finite dimensional algebras, this functor is called \emph{Nakayama duality}, and it was observed in \cite[3.2 Example 3]{BonKapr} that it is a Serre functor.} $(\,)^\vee \circ D_h$ satisfies a Serre-functor-like property for the full $\RHom$, namely:
	\begin{theorem}\label{T:Serre functor on derived cat-intro} (see \cref{C:natural pairing RHom is perfect})
		For $\pi, \pi' \in \cD^b(\cM(\wG))$ with $\pi$ finitely generated, we have a natural pairing of complexes of $\bC$-vector spaces
		\[ \RHom_\cH(\pi,\pi')\otimes^L_\bC  \RHom_\cH(\pi',D_h(\pi)^\vee)\to \RHom_\cH(\pi,D_h(\pi)^\vee)\to \bC\]
		providing a natural isomorphism
		\[ \RHom_\cH(\pi,\pi')^* = \RHom_\cH(\pi',D_h(\pi)^\vee) \]
		where $(-)^*$ is taking the dual vector space degree-wise.
	\end{theorem}
	Notice that we do not claim that $\RHom_\cH(\pi,D_h(\pi)^\vee)\to \bC$ is an isomorphism in general but merely that such a canonical map exists.

	For a parabolic $\wP$ with Levi decomposition $\wP=\wL N$,  we denote by $\bfi_{\wL,\wP}^\wG$ (resp., $\bfr_{\wL,\wP}^\wG$) the parabolic induction (resp., restriction) functors.
	Here are further homological properties that one can prove about $D_h$:
	\begin{theorem} \label{T:homol prop of D_h-intro}
		The functor $D_h\colon\cD^b(\cM(\wG))\to (\cD^b(\cM(\wG)))^\op$ enjoys the following properties
		\begin{enumerate}
			\item\label{T:subpoint:vanishing Ext for D_h} If $\pi\in\cM(\wG)_\fs^\finl$ is a finite length representation in a fixed Bernstein block, then $D_h(\pi)$ is concentrated in degree $d(\fs)$,
			\item\label{subp:intro-D_h2=1} $D_h^2\simeq \Id$,
			\item\label{subp:intro-D_h on cusp} If $\pi$ is cuspidal of finite length and lives in a single block, then 
			\[D_h(\pi) = \pi^\vee[-d(\pi)],\]
			\item $D_h \circ \bfi_{\wL,\wP}^{\wG} = \bfi_{\wL,\wP^-}^{\wG} \circ D_h$ and  $D_h\circ \bfr_{\wL,\wP}^{\wG} = \bfr_{\wL,\wP}^{\wG} \circ D_h$,
			\item\label{T:subpoint:exact involution} $\bD_h:=D_h[d(\fs)]$ restricted to finite length representations in $\cM(\wG)_\fs$ is an exact involution providing an equivalence
			\[ \bD_h\colon\cM(\wG)_\fs^{\finl}\xrightarrow{\sim} (\cM(\wG)_{\fs^\vee}^{\finl} )^\op,\]
			where $\fs^\vee$ is the contragredient block of $\fs$
		\end{enumerate}
	\end{theorem}
	
	Part \eqref{T:subpoint:vanishing Ext for D_h} is proved in \cref{SS:vanishing}.
	The involution statement is shown in \cref{C:subp main thm:involut}.
	Part \eqref{subp:intro-D_h on cusp} is showed in \cref{S:dualities on finite length} after a quick interlude on the Grothendieck--Serre duality.
	The commutation with parabolic induction and restriction is dealt with in \cref{SS:Dh and ind res}.
	The last part is a consequence of \eqref{T:subpoint:vanishing Ext for D_h},\eqref{subp:intro-D_h2=1} and \eqref{subp:intro-D_h on cusp}. 
	See also \cref{C:subp main thm:involut}.
	
	Let us say a few words about the Grothendieck--Serre duality which we will briefly
        review in \cref{SS:general algebra,S:dualities on finite length} in the generality that we need.
	Any representation $\pi\in\cM(\wG)$, besides having an action of $\wG$, has also an action of the Bernstein center $\cZ$, and by definition these two actions commute.
	Therefore, any operation that one does to a $\cZ$-module will produce again a $\wG$-module.
	The most trivial operation one can think of is taking the contragredient.
	This works well, in the sense that it is an involution, on admissible representations. 
	One would like to extend the contragredient  in a functorial  way to all finitely generated modules such that it stays an involution.
	Working with a fixed Bernstein block at a time, the question now is sent in the commutative algebra court: we have a nice ring (invariants under a finite group of a Laurent polynomial algebra, in particular Cohen--Macaulay) and finitely generated modules over it.
	We would like to extend the contragredient from finite length modules to all finitely generated modules in such a way that it stays an involution.
	One way to do this, coming from algebraic geometry, is to use the Grothendieck--Serre duality.
	Although all the properties are nice and formal, one must now work in the derived category since the abelian category of modules is not preserved anymore under this duality.
	
	For a commutative regular ring $A$ of Krull dimension $d$ (always containing a field $k$), the normalized dualizing (or canonical) complex is defined to be $\omega_A^\circ:=\omega_A[d]\in\cD^b(A\lmod)$, where $\omega_A=\Lambda^d\Omega_A^1$ is the module of top differentials on $A$.
	If $A\to B$ is a finite map of commutative rings and $\omega_A^\circ$ is the normalized dualizing complex of $A$, then $\omega_B^\circ:=\RHom_A(B,\omega_A^\circ)$ is the normalized dualizing complex of $B$.
	
	In our context, the (component of the) Bernstein center $\cZ_\fs$ is Cohen--Macaulay hence it is finite free over a regular subalgebra.
	The above paragraph allows us to see that it has a normalized dualizing complex which moreover lives in a single degree (this is a characterization of Cohen--Macaulay rings).
	
	For a commutative ring $A$ with a normalized dualizing complex $\omega_A^\circ$, one defines the Grothendieck--Serre duality as
	\begin{align*}
		D_{GS} &\colon \cD^b(A\lmod)\to (\cD^b(A\lmod))^\op \\
		M & \mapsto \RHom_A(M,\omega_A^\circ).
	\end{align*} 
	This has the nice property of being an involution $D_{GS}^2\simeq \Id$ when restricted to finitely generated modules. It is not obvious how to compute $D_{GS}(M)$ for a particular module $M$. 
	In general, it is a complex of modules even if $M$ is only a module.

        However, a particular case in which  $D_{GS}(M)$  is easy to compute is that of finite
        length modules $M$ (over an arbitrary ring $A$ containing the field $k$!).
	Indeed, we have $D_{GS}(M)=M^*$ for any $M\in A\lmod$ of finite length.
	For convenience of the reader we include (the short) proof of this in \cref{S:dualities on finite length}.

        Going back to the category of smooth representations of $\wG$, we have the Grothendieck--Serre duality with respect to the Bernstein's center (component by component, see \cref{S:dualities on finite length}.)
	\begin{align*}
		D_{GS/\cZ}&\colon\cD^b(\cM(\wG)) \to (\cD^b(\cM(\wG)))^\op\\
		\pi & \mapsto \RHom_\cZ(\pi,\omega_\cZ^\circ)^\sm
	\end{align*}
	where the superscript ``$\sm$'' means taking smooth vectors.
	On finitely generated representations it is an involution, i.e.,  $D_{GS/\cZ}^2\simeq \Id$.
	In order to make sure we extended the contragredient from admissible to all finitely generated,  one needs to check that $D_{GS/\cZ}(\pi)=\pi^\vee$
        for all admissible representations $\pi$ of $G$. 
	This is a folklore result whose (again easy and rather formal) proof we include for the convenience of the reader and for lack of a better reference (see \cref{C:D_GS on admissibles is contrag}).

        However, beyond finite length representations, $D_{GS/\cZ}(\pi)$ seems not to have been
        considered.  As a specific example,
        we may mention that we do not know  what $D_{GS}(M)$ is for
$M = \ind_{K}^{G} \bC$, the space of locally constant
        compactly supported $\bC$-valued functions on $K\backslash G$ for $K$ a compact open subgroup of $G$. Here is a more specific case. Suppose  $G$ is a split group, and $K=\I$ an Iwahori subgroup of $G$, with
        $\HH(\I)=C_c(\I\backslash G/\I)$, the Iwahori Hecke algebra of $G$ and $Z$ the center of $\HH(\I)$. This
        particular Bernstein component of $G$ (called the Iwahori block, consisting of smooth representations of $G$
        generated by their $\I$-fixed vectors)
        is described as the category of modules over $\HH(\I)$ by sending a smooth
        representation $M$ of $G$ 
        (generated by its $\I$-fixed vectors) to $M^\I$ which is a module for $\HH(\I)$. 
        Then for $M = \ind_{\I}^{G} \bC$, $M^\I = \HH(\I)$ with the left regular representation of
        $\HH(\I)$ on itself. In this case, 
        $D_{GS}(M)$
        is the $\HH(\I)$-module corresponding to $\Hom_Z(\HH(\I), Z)$, i.e.
        $\Ext_Z^i(\HH(\I),Z)=0$ for $i>0$.
        This follows first by noting that
by 
\cite[Theorem 5.4.1]{Chriss-Kamal}, $\ind_{\I}^{G} \bC$ is isomorphic to the universal unramified principal series $\ind_{B}^G(\Pi)$ where $\Pi= \ind_{T^o}^T \bC$, $B$ is a Borel subgroup of $G$, $T \subset B$ a maximal torus of $G$ and
$T^o$, the maximal compact subgroup of $T$. In this case, it follows from 
\cite[Theorem 1.9.1.1 and Lemma 1.10.4.1]{Roche-Ottawa} that $M^\I = \HH(\I)$ is a free module over $Z$,
proving that $\Ext_Z^i(\HH(\I),Z)=0$ for $i>0$. However, we do not know what $\HH(\I)$-module we get
for $\Hom_Z(\HH(\I),Z)$
--- and we do not have any guesses to offer except to ask if $\Hom_Z(\HH(\I),Z) \cong \HH(\I)$? The corresponding question for
the group asks about  the structure of the $G$-module 
$D_{GS}(M)$ for $M = \ind_{\I}^{G} \bC$.

	Homological duality allows us to prove part \eqref{subp:intro-D_h on cusp} of \cref{T:homol prop of D_h-intro}.
	Additionally, it permits us to understand precisely under what conditions do we have an isomorphism $[d]\circ D_h\simeq D_{GS/\cZ}$ on a given block.
	For precise details, we refer to \cref{C:D_h=D_GS if and only if} and the condition (FsG) in Definition \ref{FsG}.
	
	\medskip
	
	Returning to homological duality and its Serre functor property, let us state the main theorem proved
        in this work.
	For $\pi\in\cM(\wG)_\fs^\finl$, a finite length representation in a Bernstein block,
        we put $\bD_h(\pi) = H^{d(\pi)}(D_h(\pi))$ to be the unique non-zero cohomology of $D_h(\pi)$ (see \cref{T:homol prop of D_h-intro}\eqref{T:subpoint:vanishing Ext for D_h}).
	
	\begin{theorem}\label{T:Serre functor reps-intro} (see \cref{T:main duality Schneider-Stuhler})
		Let $\pi\in \cM(\wG)_\fs$ be of finite length and $\pi'\in \cM(\wG)$ be arbitrary. 
		Then the following natural pairing is perfect
		\[ \Ext^i_{\wG}(\pi,\pi')\otimes \Ext^{d(\pi)-i}_{\wG}(\pi',\bD_h(\pi)^\vee)\to \Ext^{d(\pi)}_{\wG}(\pi,\bD_h(\pi)^\vee)\to \bC.\]
	\end{theorem}
	The proof is an easy consequence of \cref{T:Serre functor on derived cat-intro} and \cref{T:homol prop of D_h-intro}.
	It is an improvement on the results of \cite{NoriPras} where the theorem was proved only for $\pi$ irreducible.
	
	\begin{remark}\hfill
		\begin{enumerate}
			\item  As in the linear case, it would be desirable to have an identification of $D_h$ with $D_{AZ}\circ (-)^\vee$ on the category of finite length representations of a covering group
                        $\wG$, where $D_{AZ}$ is the Aubert--Zelevinsky duality, similar to \cite{BBK}, \cite[Proposition IV.5.1]{SchStu} (Grothendieck group) or \cite[Theorem 4.2]{BezrPhD}.
		For the moment, it seems we do not know that $D_h(\pi) \cong D_{AZ}(\pi^\vee)$ even for an irreducible representation $\pi$ of $\wG$. We consider this as an important question 
	        which is not at all discussed in this paper. 
			However, one can consult \cite{DubJan-II} for a discussion of Aubert--Zelevinsky duality for finite central extensions of reductive $p$-adic groups (in the Grothendieck group of representations).
			
			\item As already mentioned, the duality \cref{T:Serre functor on derived cat-intro} was first proved in \cite[Duality Theorem]{SchStu} using a more involved argument.
			  Their strategy was to first show the vanishing of Ext groups from
                          \cref{{T:homol prop of D_h-intro}}  (\cref{T:subpoint:vanishing Ext for D_h}) and then proceed to prove the isomorphism from \cref{T:Serre functor on derived cat-intro} using this vanishing. 
			Both steps hinge on nice projective resolutions constructed from the Bruhat-Tits building.
			The advantage of the proof that we present is that by keeping the homological duality functor $D_h$ at the derived level, the duality theorem becomes very easy and requires no technology.
			Once the vanishing of Ext-groups \cref{{T:homol prop of D_h-intro}}(\cref{T:subpoint:vanishing Ext for D_h}) is proved (for which we follow the argument in \cite{BerNotes}) we immediately deduce the required duality theorem at the level of abelian categories.
		\end{enumerate}
	\end{remark}

	\begin{center}{\bf Acknowledgments}\end{center}
	D.F. would like to thank IIT Mumbai, where this work has started,  for their hospitality.
        The second author  thanks  SERB, India for its support
through the JC Bose
Fellowship, JBR/2020/000006. 
His work was also supported by a grant of
the Government of the Russian Federation
for the state support of scientific research carried out
under the  agreement 14.W03.31.0030 dated 15.02.2018.
	We would like to thank Roman Bezrukavnikov for some very useful remarks concerning homological duality and Grothendieck--Serre duality. The authors
	are grateful to the referees for a careful reading and helpful remarks.
	\section{Categorical generalities}
	
	\subsection{Preliminaries on decomposing categories and centers}\label{SS:prelim centers}
	
	Given $I$ a set (possibly infinite) and $\cC_i, i\in I$, a family of abelian categories, we can define the product category $\cC:=\prod_{i\in I}\cC_i$ in the 2-category of abelian categories through the usual universal property.
	Concretely, one constructs it as follows:
	\begin{itemize}
		\item the objects are tuples $(X_i)_{i\in I}$ with $X_i\in \cC_i$ for every $i\in I$;
		\item the morphisms are defined by 
		\[ \Hom_\cC((X_i)_i, (Y_j)_j):=\prod_{i\in I} \Hom_{\cC_i}(X_i,Y_i); \]
		\item the projection functors $\pi_i\colon \cC\to \cC_i$ send a tuple $(X_s)_s$ to $X_i$.
	\end{itemize}
	The universal property of a product is immediately verified.
	
	In addition to the projection functors $\pi_i\colon\cC \to \cC_i$, there are also natural inclusion functors
	\[ \iota_i\colon \cC_i\to \cC \] 
	sending an object $X\in \cC_i$ to the tuple that has $X$ in position $i$ and $0$ everywhere else.
	By construction, the functor $\iota_i$ is fully faithful so we can think of the category $\cC_i$ as a full subcategory of the product category $\cC$. 
	We will drop the functor $\iota_i$ from the notation most of the time.
	Two such subcategories $\cC_i$ and $\cC_j$ for $i\neq j$ are (derived) orthogonal, namely
	\[ \Ext^k_\cC(X_i,X_j) = 0,\text{ for all }k\ge 0.\]
	
	Moreover we notice that the object $(X_i)_i\in\cC$ is the direct sum of the objects $\iota_i(X_i)$ in $\cC$.
	Indeed, let us check the universal property for direct sums: for an arbitrary object $Y=(Y_i)_i\in\cC$ we have
	\begin{align*}
		\Hom_{\cC}((X_i)_i, (Y_i)_i)& = \prod_i \Hom_{\cC_i}(X_i,Y_i)\\
		& = \prod_i \Hom_{\cC}(\iota_i(X_i),(Y_j)_j)\\
		& = \Hom_{\cC}(\oplus_i \iota_i(X_i),(Y_j)_j).
	\end{align*}
	Actually, the same argument shows that $(X_i)_i$ is also the direct product of the objects $\iota_i(X_i)$, $i\in I$, in $\cC$.
	
	Caution: the above does not mean that in $\cC$  we always have $\oplus_i A_i\simeq \prod_i A_i$ for any $A_i\in\cC$.
	See also \cref{E:product of categories}.
	
	As a summary, we have a category $\cC$ with full subcategories $\cC_i$, $i\in I$ that are two by two derived orthogonal and moreover every object of $\cC$ is a direct sum of objects from $\cC_i$ (the subcategories $\cC_i$ split the category $\cC$).
	Conversely, a category $\cC$ with full subcategories $\cC_i$, $i \in I$, with the above properties is a direct product of
	$\cC_i$
	provided we assume a small technical condition on the categories that we work with which in practice is always verified.
	
	\begin{proposition}\label{P:categ is product if and only if} (cf. \cite[\S1.9]{BerDel})
		Let $\cC$ be an abelian category and let $\cC_i\subset \cC$, $i\in I$ be full abelian subcategories of $\cC$. 
		Assume that 
		\begin{enumerate}
			\item\label{i:sums} $\cC$ admits direct sums indexed by $I$,
			\item\label{i:obj dir sum} every object of $\cC$ can be written as a direct sum of objects of $\cC_i$,
			\item\label{i:orth} $\Hom_\cC(X_i,X_j)=0$ for all $X_i\in\cC_i$, $X_j\in\cC_j$ and $i\neq j$.
			\item\label{i:technical cond} If $X_i\in\cC_i$, $i\in I$ and $f\colon Y\to \oplus_i X_i$ is such that $\proj_i\circ f=0$ for all $i\in I$ then $f=0$.\footnote{There are pathological examples where this condition is not satisfied.}
		\end{enumerate}
		Then the natural functor $\prod_i \cC_i \to \cC$ is an equivalence of categories.
	\end{proposition}
	\begin{proof}
		There is a natural functor $\prod_i \cC_i \to \cC$ given by $(X_i)_i\mapsto \oplus_i X_i$ that is well defined by assumption \eqref{i:sums}.
		By \eqref{i:obj dir sum} it is essentially surjective.
		
		In order to show that it is an equivalence it remains to check that it is fully faithful. 
		
		Given this observation and taking into account assumptions \eqref{i:orth} and \eqref{i:technical cond}, we have
		\begin{align*}
			\Hom_{\prod_i \cC_i}((X_i)_i,(Y_i)_i)  & = \prod_i \Hom_{\cC_i}(X_i,Y_i)\\
			& = \prod_i \Hom_{\cC}(X_i,\oplus_j Y_j)\\
			& = \Hom_\cC (\oplus_i X_i,\oplus_j Y_j)
		\end{align*}
		which proves the fully-faithfulness.	
	\end{proof}

	The following example involving direct sum and direct product might be clarifying. 
	\begin{example}\label{E:product of categories}
		Consider the algebra $A=\bigoplus_{i\in \bZ} \bC = Func_c(\bZ,\bC)$, i.e., complex valued functions with compact support defined on $\bZ$ with pointwise multiplication.
		It is an algebra that has no unit but it has enough idempotents.
		We can consider the category of non-degenerate modules over $A$, call it $\cC:=A\lmod\nd$, where an object of $\cC$ is an $A$-module $M$ such that for any $m\in M$ there exists an idempotent $e\in A$ such that $em=m$.
		By considering the delta functions $\delta_i$, $i\in\bZ$, it is clear that any non-degenerate module $M$ satisfies
		\[M=\bigoplus_i M_i,\text{ where }M_i:=\delta_iM.\]
		Notice also that if $M=M_i$ and $M'=M'_i$ then $\Hom_\cC(M,M') = \Hom_{\mathrm{Vec}_\bC}(M_i,M'_i)$.
		Given two non-degenerate $A$-modules $M$ and $M'$ we see therefore that
		\[\Hom_\cC(M,M') = \Hom_\cC(\oplus_i M_i,\oplus_j M'_j) = \prod_i \Hom_{\mathrm{Vec}_\bC}(M_i,M'_i)\]
		because there are no $A$-module morphisms between $M_i$ and $M'_j$ if $i\neq j$. 
		
		For any $i\in\bZ$ we can consider the full subcategory of $\cC$ consisting of those modules $M$ on which $\delta_i$ acts as identity. In other words, modules $M$ such that $M=\delta_iM$. We denote it by $\cC_i$ and by what we said above we have $\cC_i\simeq\mathrm{Vec}_\bC$ for all $i\in \bZ$.
		From the above Proposition we can deduce an equivalence of categories
		\[\cC\simeq \prod_{i\in\bZ}\mathrm{Vec}_\bC.\]
		Since all the categories $\cC_i$ are equivalent to $\mathrm{Vec}_\bC$ one could confuse the direct sum $\oplus_i^\cC M_i\in \cC$ where each $M_i$ is a vector space, to the direct sum $\oplus_i^{\mathrm{Vect}_\bC}M_i\in \cC$ where the superscripts refer to the category where the direct sums/products are taken. The former is an "external" direct sum whereas the latter is an "internal" direct sum (all the objects are seen as belonging to the \emph{same} $\mathrm{Vec}_\bC$). 
		Moreover, we actually have $\oplus_i^\cC M_i \simeq \prod_i^\cC M_i$. One direct way to see this is to remember that $\cC = A\lmod\nd$ and that the direct product of non-degenerate modules is defined by
		\begin{align*}
		\prod_i{}^{A\lmod\nd} 
		M_i &= 
		\{(m_i)\in\prod_i{}^{\mathrm{Vect}_\bC} M_i
		\mid \exists e\text{ idempotent in }A 
		\text{ such that } em=m\} \\
		&= \{(m_i)\in\prod_i{}^{\mathrm{Vect}_\bC} M_i\mid \text{ only finitely many } m_i \text{ are non-zero}\}.
		\end{align*}	
	\end{example}	
	
	\begin{definition}
		Given an abelian category $\cC$, we define its center $\cZ(\cC)$ to be $\End(\Id_{\cC})$, the endomorphisms of the identity functor.
	\end{definition}
	
	\begin{remark}\label{R:equiv of cats center}
		The center of a category is preserved under categorical equivalence.	
	\end{remark}
	
	Since $\cC$ is additive we see that $\cZ(\cC)$ is a commutative ring with unit.
	Moreover, the category $\cC$ becomes naturally $\cZ(\cC)$-enriched, i.e., the $\Hom$ spaces in $\cC$ have a natural action of $\cZ(\cC)$ making them $\cZ(\cC)$-modules and $\cC$ into a $\cZ(\cC)$-linear category.
	
	\begin{example}\label{Ex:center of A-mod}
		If $\cC = A\lmod$ for an algebra $A$ with unit, then it is not hard to see that $\cZ(\cC)$ is the center $\cZ(A)$ of $A$.
	\end{example}
	
	For $e\in \cZ(\cC)$ an idempotent we denote by $e\cC$ its image in $\cC$: it is the full subcategory consisting of those objects on which $e$ acts by identity.
	
	A similar argument as in \cref{P:categ is product if and only if} proves the first part of
	\begin{proposition}\label{P:center of product of cats}
		If $\cC\simeq \prod_i \cC_i$ then $\cZ(\cC) \simeq \prod_i \cZ(\cC_i)$.
		Conversely, if $\cZ(\cC) \simeq \prod_i Z_i$ then the identity of each $Z_i$ provides an idempotent $e_i\in \cZ(\cC)$ and  we have $\cC\simeq \prod_i e_i\cC$.
	\end{proposition}
	\begin{proof}
		Only the last part is non-trivial and it follows by applying \cref{P:categ is product if and only if}. 
	\end{proof}
	
	The following is a generalization of \cref{Ex:center of A-mod} for algebras without unit but with enough idempotents.
	An algebra $A$ is said to have enough idempotents if for any element $a\in A$ there exists an idempotent $e\in A$ such that $ae=ea=a$.
	
	A left $A$-module $M$ is said to be non-degenerate if for any $m\in M$ there exists an idempotent $e\in A$ such that $em=m$.
	The category of non-degenerate left $A$-modules is denoted as $A\lmod\nd$.
	
	If $A^o$ is the opposite algebra to $A$,  then $\overline{A}:=\Hom_{A^o}(A,A)$ is the space of right $A$-invariant maps from $A$ to $A$. 
	We can write $\overline{A}= \Hom_{A^0}(\varinjlim eA,A)= \varprojlim Ae$ where the limit is taken over all the idempotents of $A$.
	Clearly $\overline{A}$ is an $A$-bimodule.
	
	\begin{proposition}[{\cite[Lemme 1.5]{BerDel} or \cite[I.1.7]{Renard}}]\label{P:center module idempot algebra}\hfill\\
		The center of the category $A\lmod\nd$ is identified with the center of $\overline{A}$ as an $A$-bimodule (i.e., with the $A$-bimodule endomorphisms of $A$). 
		It can be further identified with $\varprojlim Z(eAe)$ where $e$ goes over all idempotents of $A$.
	\end{proposition}
	
	In our context, the abelian categories that we will encounter will be equivalent to module categories over a unitary ring and as such their center is just the center of the corresponding ring.
	
	We use a standard tool from category theory to detect when an abelian category is equivalent to a module category.
	
	Recall that a functor $F\colon \cC\to\cD$ is called \emph{faithful} (resp., \emph{fully-faithful}) if for any objects $A,B\in\cC$ it induces an injective (resp., bijective) map on Hom spaces:
	\[ \Hom_\cC(A,B)\to \Hom_\cD(F(A),F(B)).\]
	One calls $F$ \emph{essentially surjective} if every object in $\cD$ is isomorphic to the image through $F$ of an object from $\cC$.
	
	A basic theorem in category theory (see \cite[IV.4 Theorem 1]{MacLane}) says that a functor $F\colon \cC\to \cD$ is an equivalence of categories if and only if it is essentially surjective and fully-faithful.
	
	An object $P$ of an abelian category $\cA$ is called 
	\begin{enumerate}
		\item compact (or finite) if $\Hom_\cA(P,-)$ commutes with arbitrary coproducts(=direct sums),
		\item projective if $\Hom_\cA(P,-)$ is exact,
		\item a generator if $\Hom_\cA(P,-)$ is faithful, a notion that  we just recalled above
                and which for $P$ a projective object amounts to the assertion that for $X \not = 0$,
                $\Hom_\cA(P,X) \not = 0$. 
		\item progenerator if it is projective and a generator.
	\end{enumerate}
	
	\begin{proposition}[see {\cite[4.11]{Pareigis-categories}}]\label{P:equiv module category}
		Let $\cA$ be an abelian category admitting coproducts. 
		For a compact progenerator $P$ of $\cA$,  the functor
		\begin{align*}
			\cA & \to \rmod\End_\cA(P)\\
			X&\mapsto \Hom_\cA(P,X)
		\end{align*}
		is an equivalence of categories between $\cA$ and the category of left modules over the ring $\End_\cA(P)$.
	\end{proposition}
	
	\subsection{Serre functors}\label{SS:Serre functors}
	The notion of a Serre functor for an additive category was introduced in \cite{BonKapr} in order to capture duality phenomena similar to Serre's duality theorem on smooth projective varieties.
	It is intimately related to questions of representability of certain functors.
	Moreover this turns out to be a very rigid notion.
        	In particular, if a Serre functor exists then it is unique (up to natural isomorphism).
	
	\begin{definition}\cite[Definition 1.28]{Huyb-FM}\cite[Definition 3.1]{BonKapr}
		Let $\cC$ be a $k$-linear category. 
		A Serre functor $\bS$ on $\cC$ is an equivalence of categories $\bS\colon \cC\to \cC$ together with natural isomorphisms
		\[\eta_{A,B}\colon  \Hom_\cC(A,B)^*\to \Hom_\cC(B,\bS(A)), \,\text{ for all } A,B\in\cC.\]
	\end{definition}
	\begin{remark}
		It is observed in \cite[Lemma 1.30]{Huyb-FM} that if one assumes that the categories $\cC_1,\cC_2$,  with Serre functors $\bS_1,\bS_2$, 
				have finite dimensional Hom spaces,  then any functorial isomorphism  between
                the categories  $\cC_1,\cC_2$ commutes with the Serre functors.
	\end{remark}
	
	An example of a Serre functor, actually one that motivated the notion, is the classical Serre duality.
	Namely, let $X$ be a smooth projective variety over $k$ of dimension $d$ and denote by $\omega_X$ its canonical sheaf.
	Then Serre duality stipulates a natural isomorphism
	\[ \Ext^i_{\cO_X}(\cF,\cG)^* \simeq \Ext^{d-i}_{\cO_X}(\cG,\cF\otimes\omega_X), \text{ for all } \cF,\cG\in\cD^b(\Coh(X)).\]
	In other words, the functor
	\[ -\otimes \omega_X[d]\colon \cD^b(\Coh(X))\to \cD^b(\Coh(X)) \]
	is a Serre functor.
	
	A simpler example comes from finite dimensional algebras over a field $k$ and is known as the Nakayama functor, see \cite[\S3.2, Example 3]{BonKapr}.
	Let $A$ be a finite dimensional algebra over $k$ and suppose it has finite homological dimension.
	(Although there are no nontrivial finite dimensional commutative algebras of finite homological dimension over a field $k$, there are many non-commutative ones, for example, there are  constructions of such algebras using Quivers.)
 We consider the derived category $\cC:=\cD^b_{fd}(A\lmod)$ of finite dimensional left $A$-modules.
        	There are two duality functors
	\[ \cD^b_{fd}(A\lmod) \stackrel{\delta}{\to} \cD^b_{fd}(\rmod A) \stackrel{(-)^\vee}\to \cD^b_{fd}(A\lmod),\]
	where $\delta(M) = \RHom_A(M,A)$ and $M^\vee = \RHom_{\cD^b(k)}(M,k)$ for any object $M\in \cD^b(A\lmod)$.
	We put $D_{Nak}:= (-)^\vee\circ \delta$ and call it the Nakayama functor.
	
	The following proposition is very easy to prove from classical adjunctions  and our results in \cref{S:abstract duality theorem} are essentially a detailed version of it for idempotented algebras (see \cref{T:Nakayama functor is Serre}):
	\begin{proposition}
		The Nakayama functor $D_{Nak}\colon \cD^b_{fd}(A\lmod)\to\cD^b_{fd}(A\lmod)$ is a Serre functor.
	\end{proposition}
	\begin{remark}
		Notice that we have  to restrict ourselves to finite dimensional modules because of  the appearance of the dual vector space which provides an equivalence of categories only for finite dimensional modules.
	\end{remark}
	In order to extend the notion to more general categories, in particular to finitely generated modules over an algebra finite over its center, Bezrukavnikov and Kaledin propose the notion of relative Serre functor (see \cite[\S 2.1]{BezKal}).
	Since we do not prove anything about relative Serre functors we prefer to defer this discussion to a later work.

	\subsection{An abstract duality theorem}\label{S:abstract duality theorem}
	
	The purpose of this section is to prove an abstract duality theorem reminiscent of the Serre functor property of the Nakayama functor $\RHom_A(-,A)^*$ (see \cite{BonKapr} and \cref{SS:Serre functors}).
	This is achieved in \cref{T:Nakayama functor is Serre} and \cref{C:D_h square to id for fin.h.dim}.
	Although the results presented in this section are well-known and we do not claim any originality, for lack of a
	precise reference,  we provide all the details. We do not strive for maximal generality, so sometimes we make hypotheses which might not be necessary but which hold in our applications to $p$-adic groups.

	Let $k$ be a field and let $A$ be an idempotented $k$-algebra, i.e.,  for every $a\in A$ there exists an idempotent $e\in A$ such that $ae=ea=a$. 
	Clearly any unitary algebra is idempotented.
	We suppose moreover that $A$ has a countable filtered set of idempotents.
	A left $A$-module $M$ is said to be \emph{non-degenerate} if $AM=M$, equivalently, if
        for any $m\in M$, there exists an idempotent $e\in A$ such that $em=m$.

\begin{remark}
	For a non-unitary ring $A$, a free $A$-module is not necessarily projective.
	The basic projective left $A$-modules are $Ae$ with $e$ an idempotent.
	Any projective finitely generated non-degenerate $A$-module is a direct summand of $\oplus_i Ae_i$ where $\{e_i\}$ is a finite collection of idempotents.
	This follows as in the unitary case using the fact that the module is non-degenerate and finitely generated.
\end{remark}
	
	We denote by $A\lmod$, the category of all left $A$-modules, and by $A\lmod\nd$ the full subcategory of non-degenerate left $A$-modules. 
	We use a similar notation for right $A$-modules.
	
	Consider the functors
	\begin{align}\label{Eq:functors Hom and tensor product}
		\begin{split}
			\Hom_A(-,-)\colon &(A\lmod\nd)^{op}\times A\lmod\nd \to \bZ\lmod\\
			-\otimes_A- \colon & \rmodnd A\times A\lmod \to \bZ\lmod
		\end{split}
	\end{align}
	
	The category of non-degenerate left $A$-modules, $A\lmod\nd$, has enough projective objects\footnote{We do not know if the category of all left modules over $A$ has enough projective objects, but this plays no role for us.} and so we can  derive the functors \eqref{Eq:functors Hom and tensor product} on the first argument since non-degenerate projective modules are still acyclic:
	\begin{align}\label{Eq:derived functors Hom and tensor product}
		\begin{split}
			\RHom_A(-,-)\colon &(\cD^b(A\lmod\nd))^{op}\times \cD^b(\rmodnd A) \to \cD(\bZ\lmod)\\
			-\otimes^L_A- \colon & \cD^b(\rmodnd A) \times \cD^b(A\lmod) \to \cD(\bZ\lmod)
		\end{split}
	\end{align}
	
	\begin{remark}
		The functor $\Hom_A(-,A)\colon (A\lmod\nd)^{op}\to \rmod A$ does not land inside the subcategory of non-degenerate modules.
		This can already be seen for $A$ itself: $\Hom_A(A,A) \simeq \varprojlim_e eA$ where the  inverse limit is taken over the poset of idempotents $e$ in $A$ (can take a filtered subset).	
		However, if $M$ is a \emph{finitely generated} non-degenerate module then $\Hom_A(M,A)$ is non-degenerate.
	\end{remark}
	
	For two $A$-modules $M,N\in A\lmod\nd$, there is a canonical morphism 
	\begin{align}\label{Eq:dual Hom to Hom natural morphism}
		\can_{N,M}\colon \Hom_A(M,A)\otimes_A N\to \Hom_A(M,N) 
	\end{align} 
	that extends to the derived category $\cD^b(A\lmod\nd)$
	\begin{align}\label{Eq:can Hom in derived cat}
		\can_{M,M}\colon \RHom_A(M,A)\otimes_A^L N\to \RHom_A(M,N). 
	\end{align}

	\begin{definition}
		A non-degenerate module $M$ over $A$ is said to be \emph{perfect} if it has a finite resolution by finitely generated non-degenerate projective $A$-modules.
		An object of $\cD^b(A\lmod\nd)$ is said to be perfect if it is isomorphic to a finite complex of finitely generated non-degenerate projective $A$-modules.
	\end{definition}
	
	The next lemma tells us that $\can_{N,M}$ is an isomorphism when $M$ is perfect:
	\begin{lemma}\label{L:Hom(M N)=Hom(M A)otimesN}
		If $M,N\in\cD^b(A\lmod\nd)$ and $M$ is perfect, the canonical morphism \eqref{Eq:can Hom in derived cat}
		\[ \RHom_A(M,A)\otimes_A^L N\to \RHom_A(M,N) \]
		is an isomorphism.
	\end{lemma}
	\begin{proof} 
		In this proof all modules are non-degenerate.
		We split the proof into several steps.
		
		First notice that for $M=Ae$ with $e$ an idempotent we have canonical identifications $\RHom_A(Ae,N)=\Hom_A(Ae,N)=eN$ and in particular $\Hom_A(Ae,A)=eA$ is a projective right $A$-module and so one can compute the derived tensor product with it. 
		Hence in this situation we have
		\[ \RHom_A(Ae,A)\otimes^L N = eA\otimes_A N=eN=\Hom_A(Ae,N) \]
		and we are done.
		
		Second, notice that the  morphism $\RHom_A(M,A)\otimes_A^L N\to \RHom_A(M,N)$
                in the statement of the lemma
is compatible with finite direct sums (in both arguments).
		Since a finitely generated projective module is a direct summand of $\oplus_{i=1}^n Ae_i$ for some idempotents $e_i$ we deduce the validity of the lemma for $M$ a finitely generated projective module.
		
		Next, let $M$ be a perfect complex, quasi-isomorphic to $P_\bullet$ where each $P_r$ is a finitely generated non-degenerate projective module.

		For $N=N_\bullet\in\cD^b(A\lmod\nd)$, $\RHom_A(M,N)$ is computed by the totalization of a double complex with entries $\Hom_A(P_r,N_s)$ which by what we said above is isomorphic to $\Hom_A(P_r,A)\otimes_A N_s$ through $\can_{P_r,N_s}$.
		The latter are the entries of the double complex $\Hom_A(P_\bullet,A)\otimes_A N_\bullet$ whose totalization computes $\RHom_A(M,A)\otimes^L_A N$.
		The naturality of \eqref{Eq:dual Hom to Hom natural morphism} ensures that these isomorphisms commute with all the differentials in the double complexes and therefore the total complexes are isomorphic.
		In other words, the natural map $\can_{N,M}\colon \RHom_A(M,A)\otimes^L_A N\to \RHom_A(M,N)$ is an isomorphism.
	\end{proof}

	\begin{definition}
For an $A$-module $P$, we denote by $P\nd$ its non-degenerate submodule, i.e.,
the subspace  of $P$ consisting of elements that are fixed by some idempotent in $A$.		
	\end{definition}
	
	It is clear that a morphism from a non-degenerate $A$-module to an arbitrary $A$-module lands inside its non-degenerate part.
	Said otherwise, we have 
\begin{proposition}
	The right adjoint of the natural inclusion $A\lmod\nd\to A\lmod$ is the functor $(-)\nd$ of taking the non-degenerate part of a module.
	Moreover $(-)\nd$ is exact,  so the adjunction extends trivially to derived categories.
\end{proposition}

	\textbf{Contragredient.} Any non-degenerate left or right $A$-module $M$ has a natural structure of $k$-vector space.
	We can therefore construct the contragredient module $M^\vee:=\Hom_k(M,k)^{\nd}$ as the non-degenerate $k$-linear maps from $M$ to $k$.
	As $M\rightarrow M^\vee$ is an exact functor from $A\lmod\nd$ to $\rmodnd A$, it extends naturally  to the bounded derived categories 
	\[ (-)^\vee\colon \cD^b(A\lmod\nd)\to\cD^b(\rmodnd A)  \]
	by applying it degree-wise.

	For a right $A$-module $M$, a left $A$-module $N$, both non-degenerate and a $k$-vector space $V$, we have a natural adjunction morphism
	\[ \tau_{M,N,V}\colon \Hom_k(M\otimes_A N,V)\to \Hom_A(N,\Hom_k(M,V)^{\nd}) \]
	that can be checked (as in the classical situation) to be an isomorphism.
	This isomorphism  extends to bounded derived categories and we get the usual derived tensor-hom adjunction
	\begin{align}\label{Eq:tensor hom adjunction non unitary}
		\tau_{M,N,V}\colon \RHom_k(M\otimes_A^L N,V) \stackrel{\sim}{\to} \RHom_A(N,\RHom_k(M,V)^{\nd}) 
	\end{align}

	For a complex of vector spaces $V\in \cD^b(k)$, we denote by $V^*$, the complex of dual vector spaces. 
	
	Denote by $\cD^b_{\perf}(A\lmod\nd)$, the full subcategory of $\cD^b(A\lmod\nd)$ consisting of perfect complexes.
	
	\begin{theorem}\label{T:Nakayama functor is Serre}
		The following functor
		\begin{align*}
			D_{\Nak}\colon  \cD^b_{\perf}(A\lmod\nd)  & \longrightarrow  \cD^b(A\lmod\nd) \\
			M\quad   & \mapsto  \RHom_A(M,A)^\vee,
		\end{align*}
		satisfies the Serre functor duality, i.e., for any perfect object $M$ and any object $N\in \cD^b(A\lmod\nd)$, we have a natural isomorphism of $k$-vector spaces:
		\[ \Hom_{\cD^b(A)}(M,N)^* \simeq \Hom_{\cD^b(A)}(N,D_\Nak(M)). \]
	\end{theorem}
	\begin{remark}
		A Serre functor (see \cref{SS:Serre functors}) is, 
                moreover, required to be an equivalence of categories. In our situation, in order for $D_{\Nak}$ to be an equivalence, we must restrict to finite dimensional modules (because of duality over $k$).
		In the context of representations of $p$-adic groups this means restricting to admissible modules.
		Replacing the contragredient by Grothendieck--Serre duality over the Bernstein center extends $D_{\Nak}$ from admissible modules to finitely generated modules and indeed gives a Serre functor (relative to the center) for the whole category of finitely generated representations.\footnote{We are indebted to Roman Bezrukavnikov for explaining this to us.}
		This statement already appears in \cite{BBK}.
	\end{remark}
	
	\begin{proof} Apply \cref{L:Hom(M N)=Hom(M A)otimesN} and the tensor-hom adjunction \eqref{Eq:tensor hom adjunction non unitary} to get natural quasi-isomorphisms
		\begin{align*}
			\RHom_k( \RHom_{A}(M,N), k) & \simeq    \RHom_k(\RHom_A(M,A) \stackrel{L}{\otimes} _A N, k) \\
			& \simeq    \RHom_A(N, \RHom_k(\RHom_A(M,A), k)^{\nd})\\
			& \simeq \RHom_A(N,\RHom_A(M,A)^\vee)\\
			& \simeq \RHom_A(N,D_\Nak(M)).
		\end{align*}
		Taking $\rH^0$ gives the desired result.
	\end{proof}
	
	Putting $N=M, D_\Nak(M)$ above, we get the following:
	\begin{corollary}\label{C:Hom M into DNak(M) is one dim}
		Suppose $M$ is a non-degenerate irreducible left $A$-module which belongs to $\cD^b_\perf(A\lmod\nd)$ and for which Schur's lemma holds.
		Then, \[\Hom_{\cD^b(A)}(D_\Nak(M),D_\Nak(M) )^*\simeq k.\]
	\end{corollary}
	\begin{proof}
	It follows from the following chain of natural isomorphisms
			\[k\simeq \Hom_{\cD^b(A)}(M,M)^* \simeq \Hom_{\cD^b(A)}(M,D_\Nak(M))
                \simeq\Hom_{\cD^b(A)}(D_\Nak(M),D_\Nak(M) )^*.\]
	\end{proof}
	\begin{remark}
		The above corollary shows that for $M$ irreducible, $D_\Nak(M)$ is indecomposable. 
		So if $D_\Nak(M)$ is concentrated in degree $0$, then it is isomorphic to $M$.
		In particular, if $M$ is irreducible and projective, then $D_\Nak(M)\simeq M$.
		If we apply it to a finite dimensional semisimple algebra $A$, then we get the simple fact that $D_\Nak$ leaves stable every irreducible module. 
		See \cref{C:D_h on finlen cusp is contrag} for a statement for $p$-adic groups.
		In general however, it looks like $D_\Nak(M)$ might be an interesting, non-trivial, involution.
	\end{remark}
	
	
	Taking $N=D_\Nak(M)$ in the above theorem we get a canonical map 
	\begin{align}
		\can_M\colon \Hom_{\cD^b(A)}(M,D_\Nak(M))\to k
	\end{align}
	and moreover, unwinding the definitions (see the proof of \cref{T:Nakayama functor is Serre}), we obtain
	\begin{corollary}\label{C:natural pairing RHom is perfect} For $M,N\in\cD^b(A\lmod\nd)$ with $M$ perfect, the natural pairing
		\[ \Hom_{\cD^b(A)}(M,N)\times \Hom_{\cD^b(A)}(N,D_\Nak(M))\to \Hom_{\cD^b(A)}(M,D_\Nak(M))\stackrel{\can_M}{\to} k \]
		is perfect giving rise to the isomorphism
		\[ \Hom_{\cD^b(A)}(M,N)^*\simeq \Hom_{\cD^b(A)}(N,D_\Nak(M)) \]
		from \cref{T:Nakayama functor is Serre}.
	\end{corollary}
	
	Define the 
        \emph{homological duality} functor as
	\begin{align}\label{Eq:def functor homological duality for alg A}
		D_h\colon \cD^+(A\lmod)&\longrightarrow \cD^-(\rmod A)^\op\\
		M&\mapsto \RHom_A(M,A)\notag
	\end{align}
	The same formula sends a right module into a left module.
	By abuse of notation we will still write $D_h$ for this functor.
	
	Notice that without further hypotheses, the target category of $D_h$ is only bounded from below.

	
	\begin{proposition}\label{P:abstract D_h is an involution}
		The functor $D_h$ preserves perfect complexes and is an involution on $\cD^b_\perf(A\lmod\nd)$.
	\end{proposition}
	\begin{proof}
		Notice that if $P$ is a finitely generated projective left $A$-module, then $\Hom_A(P,A)$ is a finitely generated projective right $A$-module.
		It follows that $D_h$ preserves perfect complexes.
		
		Let us now prove that $D_h^2\simeq \Id$ on perfect complexes.
		There is a natural transformation $\ev\colon\Id\to D_h^2$ coming from the evaluation morphism
		\[ \ev_P\colon P\to \Hom_A(\Hom_A(P,A),A). \]
		Note that $\ev_P$ is compatible with finite direct sums, i.e., $\ev_{\oplus P_i} = \oplus_i \ev_{P_i}$.		
		Hence in order to show that $\ev_P$ is an isomorphism for every finitely generated projective, it is enough to show it for $P=Ae$, where $e$ is an idempotent.
		But then $\Hom_A(P,A)=eA$ and $\Hom_A(eA,A) = Ae=P$ and so clearly $\ev_P$ is an isomorphism.
		
		Since $\ev$ is a natural transformation (functorial) it behaves well on complexes and hence we deduce that $\ev_{P^\bullet}$ is also an isomorphism for every perfect complex $P^\bullet\in\cD^b_\perf(A\lmod\nd)$.
	\end{proof}
	
	We denote by $\cD^b_{fg}(A\lmod\nd)$, the bounded derived category of finitely generated, non-degenerate, left $A$-modules.
	(Similarly for right $A$-modules.)
	\begin{corollary}\label{C:D_h square to id for fin.h.dim}
		Suppose $A$ has both left and right finite homological dimension. 
		Then the homological duality $D_h$ gives a functor
		\[ D_h\colon \cD^b_{fg}(A\lmod\nd)\to \cD^b_{fg}(\rmodnd A)^\op \]
		whose square is isomorphic to the identity, i.e., $D_h^2\simeq \Id$.
	\end{corollary}
	\begin{proof}
		Since $A$ has finite left and right homological dimension, $\cD^b_{fg}(A\lmod\nd) = \cD^b_{\perf}(A\lmod\nd)$, and similarly for right modules.
		Now, the conclusion follows from \cref{P:abstract D_h is an involution}.
	\end{proof}

	\subsection{Grothendieck--Serre duality and homological duality}\label{SS:general algebra}
		We provide some general discussion of Grothendieck--Serre and homological duality, and then we apply it in \cref{S:dualities on finite length} to the context of smooth representation of $p$-adic groups.

	The tools that we will use are the dualizing (or canonical) complex $\omega_A^\circ\in\cD^b(A\lmod)$ for a commutative ring as well as the exceptional pull-back functor $f^!\colon D^+(A\lmod)\to D^+(B\lmod)$ for any finite map of algebras $f\colon A\to B$.
	We will provide some of the details without striving for optimal generality.
	For a thorough discussion of these matters, one could consult \cite[Ch III]{Hart-RD} or \cite[\href{https://stacks.math.columbia.edu/tag/0A7A}{0A7A}]{SP}.
	Check \cite[\href{https://stacks.math.columbia.edu/tag/0AU3}{0AU3}]{SP} for a summary.

	For an algebra $R$, we denote by $R\lmod$ the abelian category of left modules over $R$.
	We put $\cD^+(R\lmod)$ for the bounded below derived category of $R\lmod$ and $\cD^b(R\lmod)$ for the \emph{bounded} derived category.
	For right modules we use the notation $\rmod R$. 
	The full subcategory of $R\lmod$ consisting of finite length modules is denoted by $R\lmod^\finl$ and similarly for right modules.
	
	To simplify the discussion of dualizing complexes and upper-shriek functoriality, we start with the following definition:
	\begin{definition}\label{D:dualizing complex for A^n}
		For the polynomial ring $A=k[X_1,\dots,X_d]$, we put $\omega_A^\circ = A[d]$ as an object in $\cD^b(A\lmod)$ and call it the normalized dualizing complex.
	\end{definition}
	\begin{remark}
		Notice that in general one speaks of dualizing complex\emph{es} over a ring $A$ (or more generally over a scheme $X$), which, however,  are not unique in general. 
		In order to make it unique, one normalizes it in such a way to make it compatible with exceptional pullback functors $f^!$, i.e.,  such that $f^!\omega_Y^\circ = \omega_X^\circ$ for any map of  schemes $f\colon X\to Y$. 
		Since a discussion about $f^!$ for a general map of algebras $f\colon A\to B$ would take us too far afield, we prefer to start by \cref{D:dualizing complex for A^n} and only discuss $f^!$ for finite maps.
	\end{remark}
	As a trivial example, notice that for a field $k$, we have $\omega_k^\circ = k[0]$.
	
	If $f\colon A\to B$ is a map of (commutative) algebras, then the natural restriction functor $B\lmod \to A\lmod$ is exact and extends in an obvious way to derived categories.
	We will denote this functor by $f_*\colon \cD^b(B\lmod)\to \cD^b(A\lmod)$.
	Notice that it can also be described as $f_* = {}_AB_{B}\otimes_B -$ where we have denoted by ${}_AB_B$,  the $A$-$B$-bimodule $B$.
	
	\begin{lemma}\label{L:upper shriek for finite}
		Let $f\colon A\to B$ be a finite map of commutative algebras. 
		Then the restriction functor $f_*\colon \cD^+(B\lmod)\to \cD^+(A\lmod)$ is left adjoint to \[ f^!:=\RHom_A(B,-)\colon \cD^+(A\lmod)\to\cD^+(B\lmod). \]
	\end{lemma}
	\begin{proof}
		Since $f_* = {}_AB_{B}\otimes -$
		the tensor-Hom adjunction gives for a $B$-module $M$ and an $A$-module $N$ the following natural isomorphism:
		\[ \RHom_A(B\otimes_B M,N) = \RHom_B(M,\RHom_A(B,N)).\]
		In other words, $f^!$ is right adjoint to $f_*$.
	\end{proof}
	\begin{remark}
		The previous proof does not need the map $f$ to be finite. 
If the map $f$ is not finite, the right adjoint $\RHom_A(B,-)$ is not denoted by $f^!$.
		In general, $f^!$ is right adjoint to a functor $f_!$ (restriction with compact support)  that we will not define (see for example \cite[Appendix]{Hart-RD} or \cite[\href{https://stacks.math.columbia.edu/tag/0G4Z}{0G4Z}]{SP}). 
		See \cite[\href{https://stacks.math.columbia.edu/tag/0A9Y}{0A9Y}]{SP} for a construction of $f^!$ in general.
		If $f$ is a finite map, then $f_*=f_!$,  so the right adjoint of $f_*$ deserves the name $f^!$.
	\end{remark}
	
	\begin{definition}\label{D:dualizing from finite maps}
		Suppose we have defined for a commutative $k$-algebra $A$,  a (normalized) dualizing complex $\omega_A^\circ\in\cD^b(A\lmod)$.
		Then for any finite map $f\colon A\to B$ of commutative algebras, we put $\omega_B^\circ:=f^!(\omega_A^\circ)$ and call it the (normalized) dualizing complex of $B$.
	\end{definition}
	
	\begin{remark}\label{R:fin.inj.dim} 
		(1) This definition of $\omega_B^\circ$ seems to depend on $A$ and on the map $f$.
		It turns out that this is not the case, see
                 \cite{SP}.\\
                
		(2) By construction, the object $\omega_A^\circ$ has finite injective dimension for the polynomial algebra $A$ because $A$ is smooth.
		If $A\to B$ is any finite map, also the dualizing complex $\omega_B^\circ$ will have finite injective dimension. 
		This can be seen using adjunction and the fact that $B$ has a finite projective resolution as an $A$-module since $A$ is smooth.\\
                (3) Here is a special case of the above considerations. If $R$ is a local Cohen-Macaulay ring with canonical module $\omega_R$, and $A$ is
                an algebra over $R$ which is Cohen-Macaulay, and is a finitely generated $R$ module, then if
                $c = \dim R - \dim A$,
                \[ \omega_A= \Ext^c_R[A, \omega_R],\] 
is a canonical module for $A$, and other Ext are zero.
\end{remark}
	\begin{example}
		For any ideal $I\le A$, the map $A\to A/I$ is finite, hence $\omega_{A/I}^\circ = \RHom_A(A/I,\omega_A^\circ)$.
		In particular, for $\fm$ a maximal ideal of $A$ with residue field $k$, we recover
		$\RHom_A(k,\omega_A^\circ) = \omega_k^\circ = k[0]$.
	\end{example}
	
	\begin{example}\label{E:dualizing for fdim com alg}
		If $A$ is a commutative finite dimensional $k$-algebra, then $\omega_A^\circ = A^*[0]$.
		Indeed, using the finite map $k\to A$, and the fact that $\omega_k^\circ = k[0]$,
                we deduce immediately the equality
		$\omega_A^\circ = \RHom_k(A,k) = A^*[0]$.
	\end{example}
	
As $\omega_A$ has finite injective dimension (see \cref{R:fin.inj.dim}(2)), the boundedness of the derived categories in the following definition is justified:
	\begin{definition}\label{D:D_GS for commutative k-algebras}
		For a commutative $k$-algebra $A$ with normalized dualizing complex $\omega_A^\circ$, we define the Grothendieck--Serre duality functor as
		\[ D_{GS}:=\RHom_A(-,\omega_A^\circ)\colon \cD^b(A\lmod)\to (\cD^b(A\lmod))^\op.\]
	\end{definition}
	\begin{remark}
		The canonical evaluation morphism $\Id\to D_{GS}^2$ is an isomorphism on finitely generated modules. This follows by the definition of a dualizing complex, see \cite[\href{https://stacks.math.columbia.edu/tag/0A7C}{0A7C}]{SP}.
	\end{remark}
	
	\begin{example}\label{E:D_GS f.dim. com alg}
		For $A$ a finite dimensional commutative $k$-algebra we have that $D_{GS} = (-)^*$. 
		To see this, use  $\omega_A^\circ = A^*[0]$,  discussed in \cref{E:dualizing for fdim com alg},  and the tensor-Hom adjunction:
		\[ \RHom_A(M,\omega_A^\circ) = \RHom_A(M,\Hom_k[A,k]) = \RHom_k(M,k).\]
	\end{example}
	
	\begin{lemma}\label{L:D_GS on finl for commutative A}
		For any commutative algebra $A$ with normalized dualizing complex $\omega_A^\circ$, 
		the functor $D_{GS}$ restricts to a duality 
		\[ D_{GS}\colon A\lmod^\finl  \to (A\lmod^\finl)^\op\]
		which moreover can be identified with the contragredient duality $(-)^* = \Hom_k(-,k)$.
	\end{lemma}
	\begin{proof}
		Let $M$ be a finite length $A$-module. 
		Then there exists a finite codimension ideal $I\le A$ such that the $A$-module structure on $M$ factors through the (finite) map of algebras $A\to A/I$.
		Using the right adjoint of restriction (see \cref{L:upper shriek for finite}) together with the definition of the normalized dualizing complex, we get natural isomorphisms of $A$-modules:
		\begin{align*}
			D_{GS}(M) & = \RHom_A(M,\omega_A^\circ)\\
			& = \RHom_{A/I}(M,\RHom_A(A/I,\omega_A^\circ))\\
			& = \RHom_{A/I}(M,\omega_{A/I}^\circ)\\
			& = \RHom_{A/I}(M,(A/I)^*)\\
			& = M^*
		\end{align*}
		where in the last two equalities we have used \cref{E:D_GS f.dim. com alg}.
	\end{proof}
	\begin{remark}
		Part of the content of the above lemma is that $D_{GS}(M)$ is concentrated in a single degree if $M$ is a finite length $A$-module. 
		This is not a priori obvious from the definition of $D_{GS}$ and it is a feature of dualizing complexes. 
		The fact that it lives in degree $0$ has to do with the normalization that we chose.
	\end{remark}

\begin{remark}
In this remark, we recall a part of the  duality theory for Cohen-Macaulay local rings $R$ of dimension $d$ with residue field $k$, 
and with a canonical module $\omega_R$ ($\omega_R=R$ if $R$ is Gorenstein)
which we recall has the following two properties:
\begin{enumerate}
\item $\omega_R$ has finite injective dimension.
\item $\dim_k \Ext_R^d(k, \omega_R)= 1.$
\end{enumerate}
Then if $M$ is a Cohen-Macaulay $R$-module of dimension $t$, we have:
\begin{enumerate}
\item $\Ext^i_R(M,\omega_R)=0$ for $i \not = d-t$.
\item $\Ext_R^{d-t}(M, \omega_R)$ is Cohen-Macaulay of dimension $t$.
\item $\Ext_R^{d-t}(\Ext_R^{d-t}(M, \omega_R), \omega_R) \cong M$.
\end{enumerate}

\end{remark}

	We would like to apply all this to non-commutative algebras.
	From now on, we consider $R$ a (possibly) non-commutative $k$-algebra together with an algebra map $A\to Z(R)$ from a commutative $k$-algebra $A$ to the center of $R$.
	Suppose moreover that $R$ is finite as an $A$-module and that $A$ has a (normalized) dualizing complex $\omega_A^\circ$.
	The category of left $R$-modules becomes linear over $A$ and therefore one can consider the following functor
	\begin{align*}
		D_{GS/A}\colon &\cD^{b}(R\lmod)\to (\cD^{b}(\rmod R))^\op \\
		&M\mapsto \RHom_A(M,\omega_A^\circ)
	\end{align*}
	from the (bounded) derived category of left $R$-modules to the derived category of right $R$-modules.
	We will abuse notation and denote by the same symbol the analogous functor from right modules to left modules.
	
	\begin{definition}
		A \emph{duality} on a $k$-linear category $\cC$ is a $k$-linear functor $D\colon \cC\to \cC^\op$ such that $D^2\simeq \Id$.
	\end{definition}
	Notice that for abelian  and triangulated categories, a duality is necessarily exact.
	
	\begin{corollary}\label{C:D_GS/A on finite length R-mod}
		In the above setting, the functor $D_{GS/A}$ restricts to a duality 
		\[ D_{GS/A}\colon R\lmod^\finl \to (\rmodfinl R)^\op \]
		that can moreover be identified with the contragredient $(-)^*$.
	\end{corollary}
	\begin{proof}
		Observe that since $R$ is a finite $A$-module, any finite length $R$-module restricts to a finite length $A$-module.
		We can therefore apply \cref{L:D_GS on finl for commutative A} to deduce that $D_{GS/A}(M)\simeq M^*$ as $A$-modules for any $M$ a finite length left (resp., right) $R$-module.
		The naturality implies the isomorphism $D_{GS/A}(M)\simeq M^*$ holds also as right (resp., left) $R$-modules.
	\end{proof}
	
	\begin{remark}
		Notice that the proof above still applies to $R$-modules that are of finite length as $A$-modules
               even  if we do not assume that $R$ is finite over $A$.
                	\end{remark}
	
	Recall that for a $k$-algebra $R$, the homological duality was defined in \cref{S:abstract duality theorem} as the functor 
	\[ D_h:=\RHom_R(-,R)\colon \cD^\pm(R\lmod)\to \cD^\mp(\rmod R)^\op,\]
	where it was shown that it is a duality on perfect complexes.
	If moreover $R$ has finite global dimension, then we get a duality on the bounded derived category of finitely generated $R$-modules (see \cref{P:abstract D_h is an involution}):
	\[ D_h\colon \cD^b_{fg}(R\lmod)\to \cD^b_{fg}(\rmod R)^\op.\]
	In general, there is no reason for this functor to preserve any abelian subcategories or to have good properties.
	However, in the case of representations of $p$-adic groups it does (see \cref{T:homological duality single degree})  thanks to second-adjointness and the following technical condition that is satisfied on the cuspidal blocks (keeping the assumptions on $R$ and $A$)

\begin{definition} (FsG algebra) \label{FsG}
  Let $R$ be an associative algebra containing a commutative algebra $A$ in its center over which it is finite with $A$ a finitely generated
  $k$-algebra for $k$ a field. Then $R$ is called a FsG algebra,
or a  \emph{Frobenius symmetric and Gorenstein  algebra}, if there is an integer $d$ such that 
			\[ D_{GS/A}(R)\simeq R[d] \tag{FsG}\]
			as  $R${-bimodules}. \end{definition}
	\begin{remark}
The   condition (FsG) in Definition \ref{FsG}  is satisfied for all cuspidal blocks (see \cref{C:Frob sym cond for R_rho}) but it is not satisfied for all Bernstein
        blocks
        for a  $p$-adic group, for example it does not hold for the Iwahori block for $\SL_2(F)$ as we discuss in
        \cref{12.9} below. 
\end{remark}

\begin{remark}
The name FsG 
is inspired by the notions of \emph{Frobenius symmetric} and \emph{Gorenstein} algebras.
Recall that a finite dimensional algebra $A$ over a field $k$ is called \emph{Frobenius} if there exists a nondegenerate
bilinear form $B: A \times A \rightarrow k$ such that $B(ab,c)=B(a,bc)$ for all $a,b,c$ in $A$, and is called Frobenius symmetric if furthermore $B(a,b)=B(b,a)$ for all $a,b$ in $A$. One can see that a finite dimensional $k$ algebra $A$ is Frobenius if and only if there exists a codimension one subspace of $A$ not containing any nonzero right-ideal of $A$. Before proceeding further, recall that if $M$ is a right-$A$ module (resp., left-$A$ module), then
$\Hom_k(M,k)$ is naturally a   left-$A$ module (resp., right-$A$ module).
Now note that
given a bilinear form $B: A \times A \rightarrow k$, the associated map  $A \rightarrow \Hom_k(A,k)$ is equivariant for the right-$A$ module structures if and only if $B(ab,c)=B(a,bc)$ for all $a,b,c$ in $A$, whereas the map
$A \rightarrow \Hom_k(A,k)$ is equivariant for the left-$A$ module structures if and only if $B(ba,c)=B(a,cb)$ for all $a,b,c$ in $A$. Thus a finite dimensional
associative algebra $A$ over $k$ is Frobenius if and only if $\Hom_k(A,k) \cong A$ as right $A$-modules, 
 whereas $\Hom_k(A,k) \cong A$ is a bi-$A$-module isomorphism if and only if the bilinear form $B: A \times A \rightarrow k$ is both Frobenius and symmetric. If $A$ is commutative and local then  $A$ is Frobenius if and only if $A$ is  Gorenstein. Among non-commutative algebras, the best known Frobenius symmetric algebra is the group algebra
 $k[G]$ of a finite group $G$ for $k$ a field of characteristic zero or nonzero.
	\end{remark}

	\begin{proposition}\label{P:D_h=D_GS = (-)^* for finite length R-mod under cond (FsG)}
		Keeping the same assumptions, suppose moreover that $R$ satisfies condition (FsG) in Definition \ref{FsG}.
		Then the (shifted by $d$) homological duality functor $[d]\circ D_h$ is isomorphic to $D_{GS/A}$.
		In particular, it restricts to a duality on finite length $R$-modules 
		\[ [d]\circ D_h\colon R\lmod^\finl \to (\rmodfinl R)^\op\]
		that is moreover isomorphic to the contragredient.
	\end{proposition}
	\begin{proof}
		Using condition (FsG), tensor-hom adjunction and \cref{C:D_GS/A on finite length R-mod} we get natural isomorphisms of functors:	
		\begin{align*}
			[d]\circ D_h & = [d]\circ \RHom_R(-,R) \\
			& = \RHom_R(-,\RHom_A(R,\omega_A^\circ))\\
			& = \RHom_A(-,\omega_A^\circ)\\
			& = D_{GS/A}\\
			& = (-)^* \qquad\text{ for finite length modules}.\qedhere
		\end{align*}
	\end{proof}

	\begin{remark}\label{R:replace bimodule by left/right in condition FsG}
		If in Condition (FsG),  we replace $R$-bimodule by left $R$-module, then we still get that $[d]\circ D_h$ sends $R\lmod^\finl$ to $\rmodfinl R$ but we can not identify it with the contragredient.
	\end{remark}
	The above proof gives us a bit more:
	\begin{corollary}\label{C:D_h=D_GS if and only if FsG}
		Keeping the same assumptions, we have that $[d]\circ D_h \simeq D_{GS/A}$ on all  $R$-modules if and only if condition (FsG) is satisfied.
	\end{corollary}
	\begin{proof}
		By the definition of $D_{GS/A}$ and adjunction we have
		\begin{align*}
			D_{GS/A} &= \RHom_A(-,\omega_A^\circ)\\ 
			& = \RHom_R(-,\RHom_A(R,\omega_A^\circ))
		\end{align*}
		and by Yoneda's lemma this last functor is isomorphic to $[d]\circ D_h=\RHom(-,R[d])$ if and only if condition (FsG) is satisfied.	
	\end{proof}

	Let us consider a particular case that will be of importance to us. 
	Suppose that $R$ is a finite projective $A$-algebra where $A=k[X_1^{\pm1},\dots,X_d^{\pm1}]$.
	We have that $\omega_A^\circ = A[d]$ and then condition (FsG) becomes 
	\begin{align}\label{Cond: Hom(R,A)=R as bimodules}
		\Hom_A(R,A) \simeq R \text{ as $R$-bimodules}.
	\end{align}
	\begin{corollary}\label{C:D_h=D_h/A=D_GS=(-)^* on finite length if (FsG) satisfied}
		Keeping the above notation and assuming condition \eqref{Cond: Hom(R,A)=R as bimodules}, we have the following isomorphisms of functors when restricted to finite-length $R$-modules
		\[ \RHom_R(-,R)[d]\simeq \RHom_A(-,A)[d] \simeq \Hom_k(-,k).\]
	\end{corollary}

\section{Conventions on groups and representations}		
		
	\subsection{Central extensions}\label{SS:central ext}
	Let $G=\bG(F)$
        be the locally compact group of $F$-rational
	points of a reductive linear algebraic group $\bG$ over a non-archimedean local field $F$.
	Let $\wG$ be a finite central extension of $G$ with kernel a finite abelian group ${\mu}$:
	\[ 1\to \mu\to\wG\to G\to 1 \]
	that is moreover a topological covering.
	In this situation it is proved, for example in \cite[Lemma 2.2]{DubJan-I}, that $\wG$ admits a
        basis of open sets around the identity element of $\wG$ consisting of compact open subgroups lifted from $G$
        (i.e., for which the restriction of the map $\wG\to G$ is an isomorphism).
	So $\wG$ is a totally disconnected group or an $l$-group in the terminology of \cite{BerNotes}.
	
	\subsection{Representations}
	In this section we recall some notions around the representation theory of locally compact totally disconnected groups $G$. 
	All representations will be on complex vector spaces.
	One can consult \cite{BerZel-GL,BerNotes,Renard} for details.
	We will work with the category $\cM(G)$ of smooth complex representations of such a group $G$.

	Let $H\le G$ be a closed subgroup of $G$. 
	Then restricting a representation from $G$ to $H$ we obtain an exact functor between the categories of smooth representations
	\begin{align*}
		\Res_H^G\colon \cM(G)\to \cM(H).
	\end{align*}
	The restriction functor has a right adjoint given by induction
	\[ \Ind_H^G\colon \cM(H)\to \cM(G)\]
which is defined by, 	\[\Ind_H^G(V):=\{f\colon G\to V\mid f(hg) = hf(g),\text{ for all }h\in H \text{ and }g\in G\}^\sm.\]
	The pair of functors $\Res_H^G\dashv \Ind_H^G$ has
        the adjointness property (for $V$ a representation of $G$, and $W$ of $H$)
        \[\Hom_H[\Res_H^G V, W] = \Hom_G[V, \Ind_H^G(W)],\]
        which goes under the name of Frobenius reciprocity 
	 for which we refer to  \cite{BerNotes} or \cite[III.2.5]{Renard}.
	
	The induction functor admits a subfunctor $\ind_H^G\subset \Ind_H^G$ called \emph{compact induction} consisting of  functions with compact support modulo $H$:
	\[ \ind_H^G(V):=\{f\in \Ind_H^G(V)\mid H\backslash\supp(f) \text{ is compact}\}. \]
	
	In case $G/H$ is compact we clearly have $\ind_H^G=\Ind_H^G$.
	
	\begin{definition}
		If $V$ is a smooth representation of $G$ then the contragredient representation of $V$ is defined to be the smooth part of the linear dual $V^\vee:=(V^*)^\sm$.
	\end{definition}
	
	One can prove (see \cite[III.2.7]{Renard}) that induction and compact induction are related to each other through the contragredient.
	More precisely, for $V$ a smooth representation of $H$ we have
	\[ \Ind_H^G(V^\vee) = \ind_H^G(V\otimes \delta_{H\backslash G})^\vee, \] 
	where $\delta_{H\backslash G}$ is the modular character of $G$ divided by the one of $H$.
	
	Suppose now that $H\le G$ is open. 
	Since $H$ is the complement of the union of all non-trivial left cosets $Hg$, $g\in G$, it is also a closed subgroup.
	\begin{lemma} (see \cite[III.2.6.5]{Renard})\label{L:induction from open adjunction}
		If $H\le G$ is an open subgroup then the restriction functor $\Res_H^G$ is right adjoint to the compact-induction functor $\ind_H^G$.
		In particular, $\ind_H^G$ sends projective objects to projective objects.
	\end{lemma}
	
	\subsection{Projectives and injectives} 
	We continue with $G$ being a locally compact, totally disconnected group.
	Let $\cM(G)$ be the abelian category of smooth representations of $G$.
	Let $\Hecke(G)$ be the Hecke algebra of locally constant compactly supported  functions on $G$ endowed with convolution (one needs to choose a left invariant Haar measure on $G$).
	The algebra $\Hecke(G)$ is an algebra without unit but with a rich supply of idempotents\footnote{It is
        what is called an \emph{idempotented} algebra: for every finite set of element $a_i$, there is an idempotent $e$ such that $ea_i=a_ie=a_i$ for each $a_i$.} because $G$ has a basis of neighborhoods of identity consisting of open compact subgroups. 
	A representation $V$ of $\Hecke(G)$ is said to be  \emph{non-degenerate}  if it has the  property that $\Hecke(G)V=V$.
	There is the well-known equivalence of the category of smooth representations of $G$ and the category of non-degenerate representations of $\Hecke(G)$:
	\[ \cM(G)\simeq \cH(G)\lmod\nd. \]
	
	\begin{remark}\label{R:H(G) is projective}
		Suppose that $G$ admits a countable basis of neighborhoods of identity consisting of compact open subgroups.
		Then the Hecke algebra $\cH(G)$ is a projective object in $\cM(G)$.
		Indeed, one writes $\cH(G) = \bigcup_{i=1}^\infty \cH(G)e_i$ where $e_i=e_{K_i}$ are idempotents corresponding to a countable basis of compact open subgroups of $G$.
Since $ \Hom_G(\cH(G)e_i,V) = V^{K_i}$, each of the $\cH(G)e_i$ is a projective object in $\cM(G)$,
and then one notices that
                $\Hom_G(\cH(G), -) = \lim_i \Hom_G(\cH(G)e_i,-)$
                and this latter inverse limit is exact because the transition functions are surjective (and hence the Mittag-Leffler condition is automatically satisfied).
	\end{remark}
	
	The abelian category $\cM(G)$ has a good supply of injective and projective objects, for example
	for any open compact  subgroup $K$ of $G$, $\ind_K^G (\bC)$ is a projective object (see \cref{L:induction from open adjunction}), and its smooth dual $\Ind_K^G (\bC)$ is an injective object. 
	
	We use $\Ext^i_{G}(V,V')$ to denote Ext groups in $\cM(G)$. 
	\subsection{Homological dimension}\label{SS:finite homological dimension}
	Here $G=\bG(F)$ and $\wG$ is a central extension of $G$ as defined in \cref{SS:central ext}.
	
	It is shown in \cite[Theorem 29]{BerNotes} that the category of smooth representations of $G$ has finite homological dimension.
	The argument uses the building of $G$ to give a resolution
	of the trivial module by projective modules which are sums of representations of $G$ of the form  $\ind_K^G (\bC)$ (where $K$ are compact open subgroups of $G$):
	\[0\rightarrow P_d \rightarrow  \cdots \rightarrow P_1 \rightarrow P_0 \rightarrow \bC \rightarrow 0,\] 
	where $d$ is  the split rank of $G$.
	Tensoring this resolution with any $G$-module $V$, and observing that  \[\ind_K^G (\bC) \otimes V = \ind_K^G(V|_K),\]
	we find a projective resolution of length $\leq d$ for  any $G$-module $V$.  This argument with
	the resolution of the trivial module $\bC$ for $G$ works just as well for covering groups
        $\wG\to G$ since $\ind_K^G (\bC)$ treated as a $\wG$-module is $\ind_\wK^\wG (\bC)$, and
        $\ind_\wK^\wG (\bC) \otimes V = \ind_\wK^\wG(V|_\wK)$ is a projective $\wG$-module.
	Therefore,  $\Ext^i_{\wG}(V,V')=0$ for any representations
	$V,V'$ of $\wG$,  if $i>d$.

	Once we have the Bernstein decomposition for covering groups one can prove exactly as for linear group that the category of smooth representations $\cM(\wG)$ is noetherian (\cref{T:M(G) is noetherian}).
      Since it is also of finite homological dimension, any finitely generated module admits a finite resolution by finitely generated projective modules. (If $d$ is the projective dimension of a finitely generated representation $V$, resolve $V$ by finitely generated projective modules  $ P_{d-1} \rightarrow  \cdots \rightarrow P_1 \rightarrow P_0 \rightarrow V \rightarrow 0,$ and then observe as in the proof of Hilbert Syzygy that the kernel of the map $P_{d-1}\rightarrow P_{d-2}$ must be projective.)
	This will be useful when we apply the abstract results from \cref{S:abstract duality theorem} to representation theory in \cref{S:homological duality} and \cref{S:duality SchSt} as all finitely generated modules will be perfect.

\section{Splitting the category of representations}\label{S:splitting reps}
	\subsection{Compact representations}\label{SS:compact split}
	In this section $G$ denotes an arbitrary  locally compact td-group which is
        countable at infinity.
	The most important result we need is that compact representations split the category of smooth representations of $G$.
	
	\begin{definition}
		A smooth representation $(\pi,V)$ of $G$ is said to be compact if all its matrix coefficients have compact support.
	\end{definition}
	\begin{remark} The existence of a compact irreducible representation
		implies that $G$ has compact center.
	\end{remark}

	The first important theorem about compact representations is
        that any compact representation which appears as a subquotient in any smooth representation appears as a direct summand by  {\cite[Theorem 2.44]{BerZel-GL}}. The following theorem is an easy consequence of it:

	\begin{theorem}\label{T:compact rep as GxG}
		Let $\rho$ be a compact irreducible representation of $G$. 
		Then matrix coefficients of $\rho$ provide us with an injective map of $G\times G$-modules 
		\[\rho\boxtimes \rho^\vee \subset\cH(G),\]
		and  this is the only subquotient of $\cH(G)$ isomorphic to $\rho\boxtimes\rho^\vee$.
		
		Moreover, if $\rho'$ is another irreducible compact representation of $G$, non-isomorphic to $\rho$, then the $G\times G$ representation $\rho\boxtimes\rho'^\vee$ does not appear as a subquotient of $\cH(G)$.
	\end{theorem}

\begin{proof}
There is a natural map from $\rho\boxtimes \rho^\vee$ to functions on $G$  given by \[v \otimes \ell \in \rho\boxtimes \rho^\vee  \mapsto F_{v \otimes \ell}(g) = \ell(gv),\] which if $\rho$ is a compact smooth representation of $G$ lands inside
compactly supported smooth functions on $G$, i.e., in $\cH(G)$. If $\rho$ is furthermore irreducible, the resulting map
$\rho\boxtimes \rho^\vee \subset\cH(G)$ is injective.

As any compact representation which appears as a subquotient in any smooth representation appears as a direct summand by  {\cite[Theorem 2.44]{BerZel-GL}}, the rest of the theorem is a direct consequence of
the Frobenius reciprocity:
\[ \Hom_{G \times G}[\ind_{\Delta G}^{ G \times G} (\bC), \rho\boxtimes \rho'^\vee] = \Hom_{\Delta G}[   \rho\boxtimes \rho'^\vee,  \bC] .\qedhere\] 
\end{proof}

	We denote by $\cM(G)_c$ the full subcategory of $\cM(G)$ consisting of those smooth
        representations of $G$ all whose irreducible subquotients are compact representations.
	It is clearly a subcategory closed under subquotients, direct sums and extensions.
	
	We denote by $\cM(G)_{nc}$ the full subcategory of $\cM(G)$ formed of representations that have no compact irreducible subquotient.

	For $\cS$, a collection of irreducible representations of $G$, we denote by $\cM(G)_{[\cS]}$ the subcategory of $\cM(G)$ formed of representations such that all their irreducible subquotients are isomorphic to an object in $\cS$. 
	Denote by $\cM(G)_{[\text{out }\cS]}$ the subcategory of representations such that none of their irreducible subquotients are isomorphic to an object in $\cS$.

	The following theorem summarizes the main properties of compact representations, cf. 
        {\cite[Theorem 2.44]{BerZel-GL}}:
	\begin{theorem}
        \label{T:main on cpct reps}	
		For $\cS$ a finite collection of compact irreducible representations of $G$,
		we have:
		\begin{enumerate}
			\item The category $\cM(G)_{[\cS]}$
                        is semisimple: all the objects in $\cM(G)_{[\cS]}$ are isomorphic to a direct sum of objects in $\cS$.
			\item There is  an equivalence of categories 
			\begin{align}\label{Eq:decompose subquotients in A and outside A}
				\cM(G) \simeq \cM(G)_{[\cS]}\times \cM(G)_{[\text {out } \cS]}.
			\end{align}
		\end{enumerate}
		Moreover, the category $\cM(G)_c$ admits a decomposition
		\begin{align}
			\cM(G)_c = \prod_{\tau} \cM(G)_{[\tau]}		
		\end{align}
		where the product runs over all isomorphism classes of compact irreducible representations of $G$.
		In particular, $\cM(G)_c$ is a semisimple category.
	\end{theorem}
	
	It is natural to ask if we can always decompose a representation into a direct sum of compact and non-compact.
	This is not completely automatic from \cref{T:main on cpct reps},
and one needs a further finiteness condition:
	
	Consider the following condition (see \cite[IV.1.7 (KF)]{Renard}) called "compact finite":
	\begin{align}\label{Eq:condition (KF)}
		\tag{KF}\quad \begin{minipage}{300pt}
			for any compact open subgroup $K\le G$ there is only a finite number of isomorphism classes of compact irreducible representations of $G$ having a non-zero $K$-fixed vector.
		\end{minipage}
	\end{align}
	\begin{theorem}\label{T:dec compact times non-compact}
		If the group $G$ satisfies the above condition (KF), then we have a decomposition of categories
		\begin{align}\label{Eq:dec compact times non-compact}
			\cM(G) = \cM(G)_c\times \cM(G)_{nc}.
		\end{align}
	\end{theorem}
	
	\begin{remark}
		Notice that in light of the theorem \cref{T:dec compact times non-compact}, every compact representation is projective-injective in $\cM(G)$. 
		This is a remarkably strong homological property that will be useful in the sequel.
	\end{remark}
	
	In our situation, the condition (KF) is satisfied thanks to the uniform admissibility theorem, see \cref{T:uniform adm}.

	\subsection{Compact modulo center}\label{SS:compact mod center}
	We have seen that compact representations behave as nicely as one could hope but for many interesting groups there are no such representations.
	It turns out that the issue comes from the non-compactness of the center.
	In this subsection we present an analogue of the decomposition of categories \cref{T:main on cpct reps} and \cref{T:dec compact times non-compact}.
	
	\begin{definition}
		A representation of $G$ is called \emph{compact modulo center}
		if its matrix coefficients have compact support modulo $Z(G)$.
	\end{definition}
	
Denote by $G^\circ$ the subgroup of $G = \bG(F)$,
       for $\bG$ a reductive group over $F$, defined by:
\[ G^\circ = \cap\,  \{ {\rm ker}  |\chi|\colon G \rightarrow \bR^+\} ,\]
where $\chi\colon G \rightarrow F^\times$ are the algebraic characters of $G$ defined over $F$.  
The subgroup $G^\circ$  of $G$ has the property that $G/G^\circ\simeq \bZ^d$ for some $d\ge 0$. 
Further, $G^\circ Z(G)(F)$ is a normal subgroup of finite index in $G(F)$. 
The subgroup $G^\circ \subset G(F)$ contains all compact subgroups of $G(F)$ and, in fact, $G^\circ$
is generated by all compact subgroups of $G$. Since we could not find a convenient reference, we provide a proof
for which we thank Sandeep Varma.

\begin{lemma}
	The subgroup $G^\circ$ is generated by all compact subgroups of $G$.
\end{lemma}

\begin{proof}
Let $G^d=[G,G]$ be the derived subgroup of $G$, and $j: G^s \rightarrow G^d $, the simply connected cover of $G^d$.
The group $G(F)$ operates by conjugation on $G^d$, and this action lifts to $G^s$, thus
$G(F)$ acts on the set of Iwahori subgroups of $G^s(F)$ which are permuted transitively by $G^s(F)$. Let $\I$ be a fixed Iwahori subgroup of $G^s(F)$.
We know that the only elements of $G^s(F)$ that normalize $\I$ belong to  $\I$, and
the stabilizer of $\I$ in $G^\circ$, call it  $N(\I)$, is a compact subgroup of $G(F)$ (this uses the assertion contained in the Lemma for tori which is equivalent to the assertion that anisotropic tori are the same as compact tori). Hence
$N(\I)$ surjects onto $G^\circ/ j(G^s(F))$. It follows that if we can prove the assertion of the Lemma
for the simply connected group $G^s$, then the Lemma follows for the group $G$.

The assertion of the Lemma for the simply connected group $G^s$, assuming that it is isotropic (else $G^s(F)$ is compact
and there is nothing to prove),
follows from the Bruhat decomposition $G^s(F)=\I W^{\rm aff} \I$ where $W^{\rm aff}$ is a Coxeter group which comes equipped
with a set of generators $\{ s \in S\}$ of order 2 where each $s$ treated as an element of $G^s(F)$ is compact
(since $\I \cup \I s \I$ is a compact parahoric subgroup of $G^s(F)$).
\end{proof}

	We put $\cX(G):=\Hom_\gr(G/\Go, \bC^\times)$ and call it the group of \emph{unramified characters} of $G$.
	
	\begin{remark}
		A representation $\pi$ of $G$ is compact modulo center if and only if its restriction $\pi|_{G^\circ}$ is compact.
	\end{remark}

	The following well-known proposition, although simple, is the main ingredient allowing one
        to pass from $\Go$ to $G$ (since $\Go\cdot Z(G)$ is of finite index in $G$, the proof of this proposition
        is the same as the Clifford theory in finite groups):

	\begin{proposition}[{\cite[Lemmas 3.26 and 3.29]{BerZel-GL}}]
        \label{P:irred res to Go and inertia classes}\hfill
		\begin{enumerate}
			\item  Let $(V,\pi)$ be an irreducible representation of $G$. 
			The irreducible representations of $\Go$ appearing in $\Res_\Go^G(\pi)$ are all conjugate under $G$.
			Moreover, the representation $\Res_\Go^G(\pi)$ is semisimple and of finite length.
			\item Let $(V_1,\pi_1)$ and $(V_2,\pi_2)$ be two irreducible representations of $G$.
			Then the  following are equivalent.
			\begin{enumerate}
				\item $\Res_\Go^G(\pi_1)\simeq \Res_\Go^G(\pi_2)$,
				\item there exists an unramified character $\chi\in\cX(G)$ such that
				\[ \pi_1\otimes\chi\simeq \pi_2, \]
				\item $\Hom_{\Go}(\Res_\Go^G(\pi_1),\Res_\Go^G(\pi_2))\neq 0$.
			\end{enumerate}
		\end{enumerate}
	\end{proposition}
	
	We denote by $\cM(G)_\cs$, the full subcategory of $\cM(G)$ formed of those representations all whose subquotients are compact modulo center.
	The set of irreducible objects of $\cM(G)_\cs$ is denoted by $\Irr(G)_\cs$.
	
	\begin{definition}
		We say that two irreducible representations $\rho_1, \rho_2\in\Irr(G)_\cs$ are in the same \emph{inertia class} if there exists a character $\chi\in\cX(G)$ such that $\rho_1\simeq \chi\rho_2$.
		We write the corresponding equivalence relation as  $\rho_1\sim \rho_2$, and denote the inertia class containing $\rho \in \Irr(G)_\cs$ by the square bracket $[\rho]$.
                We denote the set of  inertia classes in $\Irr(G)_\cs$ by $[\Irr(G)_\cs]$.
	\end{definition}
	
	Given $\pi\in\Irr(G)_\cs$, we denote by $\cM(G)_{[\pi]}$, the full subcategory of $\cM(G)$ consisting of those
        smooth representations of $G$ whose restriction to $\Go$ have all the irreducible subquotients only among those of $\pi|_\Go$ (a finite set of
        compact representations of $\Go$).
	
	\begin{remark}
		\cref{P:irred res to Go and inertia classes} says that two irreducible representations of $G$ are in the same inertia class if and only if their restriction to $\Go$ are isomorphic.
		Therefore the objects of the category $\cM(G)_{[\pi]}$ are those smooth representations of $G$
                all of whose irreducible subquotients are in the same inertia class as $\pi$.
	\end{remark}
	
	Using the above proposition and \cref{T:main on cpct reps}, one immediately deduces:
	\begin{theorem}[{\cite{BerDel}}]\label{T:dec of comp mod center}
		For an irreducible compact modulo center representation $\pi$ of $G$,
		we have a decomposition of the category $\cM(G)$ as
		\[ \cM(G)\simeq \cM(G)_{[\pi]}\times \cM(G)_{[\text{out }\pi]}. \]
		Moreover, the category $\cM(G)_\cs$ decomposes as
		\[\cM(G)_\cs =\prod_{[\pi]\in\Irr(G)_\cs} \cM(G)_{[\pi]}.\]
	\end{theorem}
	
	The full subcategory of $\cM(G)$ formed of those representations that have no irreducible subquotient that is compact modulo center is denoted by $\cM(G)_\indu$.
	Using \cref{P:irred res to Go and inertia classes} and \cref{T:dec compact times non-compact} one deduces easily:
	\begin{theorem}[{\cite{BerDel}}]\label{T:dec com mod center x others}
		If the group $\Go$ satisfies condition \eqref{Eq:condition (KF)} then we have a decomposition of categories
		\[ \cM(G) = \cM(G)_\cs\times \cM(G)_\indu.\]
	\end{theorem}
	
	\subsection{The simplest Bernstein component}
	The goal of this subsection is to formulate and prove an analogue of \cref{T:compact rep as GxG}
	
	Fix $\pi\in\cM(G)$ an irreducible, compact modulo center representation of $G$.
	Then $\pi\boxtimes \pi^\vee\in\cM(G\times G)$ is irreducible and compact modulo center.
	We consider the Hecke algebra $\cH(G) = \ind_{\Delta G}^{G\times G}\bC$ as a $G\times G$-module and we try to understand the part of $\cH(G)$, denoted by $\cH(G)[\pi\boxtimes\pi^\vee]$, that lives in $\cM(G\times G)_{[\pi\boxtimes\pi^\vee]}$ (see \cref{T:dec of comp mod center}). 
	
	Let us fix $\pi_0\subset \pi\mid_\Go$, an irreducible representation of $\Go$.
	Put $G_1 = \{g\in G\mid \pi_0\simeq {}^g\pi_0\}\le G$. 
	It is a normal subgroup of $G$  containing  $\Go Z(G)$
     which is of  finite index in $G$, hence $G_1$ is of finite index in $G$. 
	Denote by $\Sigma$ the group $G/G_1$ and its order by $f$.
	
	Put $H:=\Delta(G) (\dGo)\le G\times G$.
	The following is the main result of this subsection:
	\begin{proposition}\label{P:cusp component of H(G) as GxG module}
		For $\pi\in\cM(G)$ an irreducible, compact modulo center representation of $G$, 
		we have an isomorphism of $G\times G$ representations
		\begin{align}\label{Equation 1}
			\cH(G)[\dpi]\simeq \ind_{\Delta(G_1)(\dGo)}^{G\times G}(\pi_0\boxtimes \pi_0^\vee).
		\end{align}
		Further,
		\begin{align}\label{Equation 2}
			\begin{array}{rl}(\dpi)\otimes \ind_H^{G\times G}\bC &\simeq \Hom_{\Go}(\pi|_{\Go},  \pi|_{\Go}) \otimes  \cH(G)[\dpi]\\
				&\simeq e^2f \cH(G)[\dpi], 
			\end{array}
		\end{align}
		where $ \Hom_{\Go}(\pi|_{\Go},  \pi|_{\Go})$ is a finite dimensional representation of $G/\Go$, treated as a representation
		of $G\times G$ trivial on $H=\Delta(G) (\dGo)\le G\times G$.
	\end{proposition}
	\begin{proof}
		Put $\cS = \{\pi' \mid \pi'\subset \pi\boxtimes\pi^\vee |_{G^\circ\times G^\circ} \text{ irreducible}\}$.
		The set $\cS$  is a finite set of compact irreducible representations of $\dGo$ and given \cref{T:main on cpct reps} we can write
		\[ \cH(G)|_{\dGo} = \cH(G)_{[\cS]}\oplus \cH(G)_{[out\,\cS]}\]
		as $\dGo$-representations.
		Moreover, since $\cS$ is stable under conjugation by $G\times G$, the above decomposition holds as $G\times G$-modules.
		Let us write $\cH(G){[\pi\boxtimes \pi^\vee]}$ for the $G\times G$ representation $\cH(G)_{[\cS]}$.

		For $\sigma\in\Sigma$, let $\pi_\sigma:=\sigma\pi_0\subset \pi$, an irreducible
		sub $\Go$-representation of $\pi|_\Go$. 
		Notice that $\pi_\sigma$ is isomorphic to the conjugated representation ${}^\sigma\pi_0$ where the action on $\pi_0$
		is conjugated by $\sigma$.
		The irreducible summands of $\pi|_\Go$ are isomorphic to $\pi_\sigma$ for some $\sigma\in\Sigma$ and $\pi_\sigma\not\simeq\pi_{\sigma'}$ if $\sigma\neq\sigma'$.
		Let's write the restriction of $\pi$ to $\Go$ as (for an integer $e\geq 1$):
		\[ \pi|_\Go = \bigoplus_{\sigma\in\Sigma} e \pi_\sigma. \]
		
		For an irreducible representation $\pi$ of $G$, the set of unramified characters $\chi: G/\Go\to \bC^\times$ with $\pi \otimes \chi \simeq \pi$ is a finite abelian group, call it $X(\pi)$.  Then, the space $\Hom_{\Go}(\pi|_{\Go},  \pi|_{\Go}) $ comes equipped with an action of the abelian group $G/\Go$, diagonalizing which
		gives a canonical basis (up to scalars) which are nothing but the intertwining operators $\pi \otimes \chi \simeq \pi$, thus one sees that  $\dim \Hom_{\Go}(\pi|_{\Go},  \pi|_{\Go}) = e^2f$
		is the number of  the intertwining operators $\pi \otimes \chi \simeq \pi$, and 
		$\Hom_{\Go}(\pi|_{\Go},  \pi|_{\Go})$ comes equipped with a basis $e_\chi$ such that $g\cdot e_\chi = \chi(g)e_\chi$.

		The $\Go$-representation $\pi_0$ does not extend to a representation of $G_1$ but the $(\dGo)$-representation $\dpio$ does extend to a representation of $\Delta(G_1) (\dGo)$: this is most easily seen by noticing that $\pi_0$ extends to a projective
		representation of $G_1$, and therefore $\pi_0\boxtimes \pi_0^\vee$ is canonically a representation of $\Delta(G_1)$,
		hence of  $\Delta(G_1) (\dGo)$.
		The representations $\pi_\sigma\boxtimes \pi_{\sigma'}^\vee$ are also representations of $ \Delta(G_1) (\dGo) $ as
		they can be written as $(\sigma,\sigma')\dpio$.
		
		In what follows, let $\tau$ be the irreducible representation of $H=  \Delta(G)(\dGo),$
		\[\tau =   \bigoplus_{\sigma \in \Sigma} \pi_\sigma \boxtimes \pi_{\sigma}^\vee = \ind_{ \Delta(G_1)(\dGo)}^{H}  (\pi_0\boxtimes \pi_{0}^\vee) .\]

		By induction in stages: $G \times G \supset {\Delta(G)}(\Go \times \Go) \supset {\Delta(G)}$,
		we can write:
		\[ \cH(G) = \ind_{\Delta(G)}^{G\times G} \bC = \ind_H^{G\times G}\cH(\Go), \]
		as a representation of $G\times G$, where we recall that $H=\Delta(G) (\dGo)\le G\times G$
		is a normal subgroup of $G\times G$ with a natural action on $\cH(\Go)$
		(where $\Delta(G)$ acts on $\Go$ by the conjugation action, and $\dGo$ acts on $\Go$ by left and right translations). 
		Since the set $\cS$ is stable under conjugation by $G\times G$, we deduce the following
		\[ \cH(G)[\dpi] = \ind_H^{G\times G}(\cH(\Go)_{[\cS]}).\]

		\cref{T:compact rep as GxG} tells us that in $\cM(\dGo)$, we have a natural isomorphism 
		\[\cH(\Go)_{[\cS]} \simeq \bigoplus_{\sigma\in\Sigma} \dpis \simeq
		\ind_{\Delta(G_1)(\dGo)}^{ \Delta(G)(\dGo) }(  \pi_0\boxtimes \pi_{0}^\vee        ),\]
		which is moreover an isomorphism of $H$-modules, proving the isomorphism \eqref{Equation 1}.
		
		All the irreducible subrepresentations of $\dpi|_H$ are of the form $(\sigma,1)\tau\simeq {}^{(\sigma,1)}\tau$, hence
		\begin{align}\label{Eq:dpi restricted to H}
			\dpi|_H = \bigoplus_{\sigma\in\Sigma}e^2 (\sigma,1)\tau. 
		\end{align}
		Put these together and use the Mackey's relation $V\otimes \ind_K^{G} W = \ind_K^G(V|_{K}\otimes W)$ to get
		\begin{align*}
			(\dpi)\otimes \ind_H^{G\times G}(\bC)  & \simeq \ind_H^{G\times G}(\dpi|_H)\\  & \simeq
			e^2f\ind_H^{G\times G}(\tau)  \\ & \simeq e^2f \ind_{ \Delta(G_1)(\dGo)}^{ G \times G }  (\pi_0\boxtimes \pi_{0}^\vee) \\
			& \simeq
			\Hom_{\Go}(\pi|_{\Go},  \pi|_{\Go})
			\otimes  \ind_{ \Delta(G_1)(\dGo)}^{ G \times G }  (\pi_0\boxtimes \pi_{0}^\vee) \\
			& \simeq
			\Hom_{\Go}(\pi|_{\Go},  \pi|_{\Go}) \otimes \cH(G)[\dpi], 
		\end{align*}
		where in the 2nd isomorphism above, we have used \eqref{Eq:dpi restricted to H}
		together with  the fact that $H$ is  a normal subgroup of $G \times G$, and  $\ind_H^{G\times G}(\tau) =  \ind_H^{G\times G}(\tau^{(\sigma, 1)})$
		for all $\sigma \in \Sigma$; the last isomorphism is the  isomorphism \eqref{Equation 1}, and the isomorphism
		previous to that follows   
		since $\Hom_{\Go}(\pi|_{\Go},  \pi|_{\Go})$ consists of characters of $G/\Go$, treated as a
		character of $G\times G$ trivial on $\Delta(G)(\Go \times \Go)$. Thus we have  proved the isomorphism \eqref{Equation 2}.
	\end{proof}

	\begin{remark}
		The isomorphism \eqref{Equation 2} of
		the above proposition  can be interpreted as saying that for the cuspidal representation $\pi \boxtimes \pi^\vee$ of $G\times G$,
		the corresponding Bernstein component in $\cH(G)$ ``contains'' all representations
		$(\pi\alpha)  \boxtimes (\pi \alpha)^\vee$ where $\alpha$ runs over all the unramified characters of $G/\Go$.
		Since  $\pi \otimes \chi \simeq \pi$ for $e^2f$ many characters, this number shows up as multiplicity in the right hand side of \eqref{Equation 2}. 
		This means that  for each irreducible cuspidal representation $\pi$ of $G$,
		$\pi \boxtimes \pi^\vee$ ``appears'' in $\cH(G)$ with multiplicity 1, where the precise meaning of  ``appears'' is to be understood in the derived sense, see \cref{dual}.
	\end{remark}
	
	\begin{remark} Suppose the group $G$ is either a reductive $p$-adic group, or a finite cover of it. 
		Then for a standard parabolic $P=LN$ inside $G$, and any cuspidal pair $[L, \rho]$ 
		defining a Bernstein block $\fs$ (see \cref{S:Bernstein dec}), similar to \cref{P:cusp component of H(G) as GxG module}, it is natural to propose the following isomorphism in $\cM(G\times G)$
		\[\ind^{ G \times G}_{P \times P^-}(\rho \boxtimes \rho^\vee \otimes \ind_{\Delta(L) (L^o \times L^o)}^{ L\times L} (\bC)) \simeq e^2fw \cH(G)[\fs \times \fs^\vee],\] 
		where $e,f$ are defined as before
		for the cuspidal representation $\rho$ of $L$, and $w$ is the order of $N_G(L, [\rho])/L$ where  $N_G(L, [\rho])$ is the normalizer
		of $L$ preserving $\rho$ up to an unramified character. This will be a kind of Plancherel decomposition in the smooth category,
		see \cite{Hei-Planch} for some results in this direction.
	\end{remark}

	The following corollary is a consequence of the proof of the above proposition.
	\begin{cor}
		The component of $\cH(G)$ in $\cM(G)_{[\pi_1\boxtimes\pi_2^\vee]}$ is zero unless $\pi_2\simeq\pi_1\chi$ for some unramified character $\chi$ of $G$.
	\end{cor}
	\begin{remark}
		In the language of Bernstein decomposition reviewed in \cref{S:Bernstein dec}, this is equivalent to saying that the Bernstein component of $\cH(G)$ corresponding to $\pi_1\boxtimes\pi_2^\vee$ is zero unless $\pi_2\simeq\pi_1\chi$ for some unramified character $\chi$ of $G$.
	\end{remark}
	
	We note the following corollary that will be useful to us in \cref{S:Blocks as module cats} when studying contragredients.
	It also appears in \cite[Lemma B.5]{Aiz-Say}
	\begin{cor}\label{C:full cuspidal homological dual of it}
		For any irreducible representation $\pi\in\cM(G)$, compact modulo center, we have an isomorphism of $G$-modules
		\[ \Hom_G(\ind_{\Go}^G(\pi|_{\Go}),\cH(G))\simeq \ind_{\Go}^G(\pi^\vee|_{\Go}),\]
	where in the left hand side, one considers $\Hom_G(-, \cH(G))$
        for the left $G$-module
        structure on $\cH(G)$, and as $\cH(G)$ is both left and right $G$-module,
        $\Hom_G(-, \cH(G))$ continues to be a right $G$-module.
\end{cor}
	\begin{proof}
		By the Mackey theory, since $(\Go \times G) H = G \times G$, it follows that
		the restriction of $\ind_H^{G\times G}(\bC)$ to    $\Go \times G$ is $\bC \otimes   \ind_{\Go}^G\bC$.
		Hence by 
		the isomorphism \eqref{Equation 2}, and Frobenius reciprocity:
		\begin{eqnarray*} 
			e^2f\Hom_G(\ind_{\Go}^G(\pi|_{\Go}),\cH(G)) & \simeq &
			e^2f\Hom_G(\ind_{\Go}^G(\pi|_{\Go}), \cH(G)[\dpi]) \\
			& \simeq & \Hom_G(    \ind_{\Go}^G(\pi|_{\Go}),   \pi \otimes \pi^\vee \otimes  \ind_H^{G\times G}(\bC)    ) \\
			& \simeq &
			\Hom_{\Go}(\pi|_{\Go},  \pi|_{\Go} \otimes (\pi^\vee \otimes  \ind_{\Go}^{G}\bC) \\
			& \simeq & \Hom_{\Go}(\pi|_{\Go},  \pi|_{\Go}) \otimes (\pi^\vee \otimes  \ind_{\Go}^{G}\bC) \\
			&\simeq & e^2f \ind_{\Go}^G(\pi^\vee|_{\Go}),
		\end{eqnarray*}
		where in the 4th isomorphism we have used the following lemma. 
		\begin{lemma}
			Let $H_1$ (resp., $H_2$) be a totally disconnected group, and $\pi_1$ (resp., $\pi_2$) any
			smooth representation of $H_1$ (resp., $H_2$). 
			If $\pi_1$ has finite length, then treating the representation $\pi_1 \otimes \pi_2$ of $H_1 \times H_2$
                        as a representation of $H_1 \subset H_1 \times H_2$, we have,
			\[\Hom_{H_1}(\pi_1,\pi_1 \otimes \pi_2) \simeq \Hom_{H_1}(\pi_1,\pi_1) \otimes \pi_2\]
			as representations of $H_2$.\qedhere
		\end{lemma}	\end{proof}
The following lemma will be used in the next proposition.	
		\begin{lemma} \label{Hom-Ext} Let $G$ be a totally disconnected group, $G^o$ a normal subgroup of $G$ such that $G/G^o$ is a discrete group. Let $M, N$ be two smooth representations of $G$. Assume that $M$ is an irreducible
                representation of $G$ which restricted to $G^o$ is a finite sum of irreducible representations 
which are                both injective and projective as a smooth representation of $G^o$. Then,
                \[  \Ext^i_{G}[M,N] =  H^i(G/\Go, \Hom_{\Go}[M,N]).\]
		\end{lemma}

\begin{proof}
Let,
\[\hspace {5cm}\cdots  \rightarrow P_1  \rightarrow P_0  \rightarrow  M  \rightarrow 0, \hspace {2.9cm} (1)\]
be a projective resolution of $M$ as a $G$-module. Since $M$ restricted to $G^o$
is both an injective and a projective representation of $G^o$, assume that $P_i$ restricted to $G^0$ is a direct sum of
copies of the irreducible components of $M$ restricted to $G^o$.

The exact sequence (1) gives rise to an exact sequence of $G/G^o$-modules (the exactness follows since we are assuming as discussed in the last paragraph that $P_i$ are projective modules restricted to $G^o$):  
\[ 0\rightarrow \Hom_{G^o} [M, N] \rightarrow \Hom_{G^o} [P_0, N]  \rightarrow \Hom_{G^o} [P_1, N]  \rightarrow  \Hom_{G^o}[ P_2, N]
\rightarrow \cdots \hspace{.1cm}(2). \]

To prove the lemma is suffices to prove that for $P$ a projective module for $G^0$, and $N$ arbitrary smooth module for
$G$, $\Hom_{G^0}[P,N]$ is an injective module for $G/G_o$ (as then the exact sequence (2) will be an injective resolution of $\Hom_{G^o} [M, N]$ as $G/G^o$-modules).  Injectivity of $\Hom_{G^o} [P, N]$ as a $G/G^o$-module
is a simple consequence of the adjunction (for $Q$ any module for $G/G^o$):
\[ \Hom_G[Q, \Hom_{G^o}(P, N)] = \Hom_G[Q \otimes P, N] = \Hom_G [P, \Hom(Q, N)]. \qedhere \]
\end{proof}

	The next result, based on \cref{P:cusp component of H(G) as GxG module}, seems worth including as its proof is fairly elementary.
	A generalization of this proposition to finite length representations plus functoriality is proved in \cref{C:D_h on finlen cusp is contrag}.
	Recall that $G/\Go\simeq \bZ^d$ for some integer $d\ge 0$.
	\begin{proposition} \label{dual}
		Let $\pi$ be an irreducible, compact modulo center representation of $G$.
		Then $\RHom_{G}^{\bullet} (\pi, \cH(G)),$ lives only in degree $d\geq 0$ (where $d$ is the integer for which $G/\Go\simeq \bZ^d$), where it is isomorphic to $\pi^\vee$.
	\end{proposition}
	\begin{proof}
		By \cref{P:cusp component of H(G) as GxG module},
		the proof of this proposition boils down to proving that (for $H= \Delta(G)(G^o\times G^o)$),
		$\RHom_{G}^{\bullet} (\pi, \pi \boxtimes \pi^\vee \otimes \ind_H^{G\times G}(\bC))$
		lives only in degree $d$ and 
		\[\Ext_{G}^{d} (\pi, \pi \boxtimes \pi^\vee \otimes \ind_H^{G\times G}(\bC)) = (e^2f)\pi^\vee.\]

We apply Lemma \cref{Hom-Ext}  to $M=\pi$ (treated as a module for $G\times G$
		on which $1 \times G$ acts trivially), and $N= \pi \boxtimes \pi^\vee \otimes \ind_H^{G\times G}(\bC)$.
		Note that $ \ind_H^{G\times G}(\bC)$ restricted to $\Go \times G $ is $\ind_{\Go \times \Go}^{\Go \times G }(\bC)$.		Thus, $N$ restricted to $\Go \times G \subset G \times G$
		is $\pi|_{\Go} \otimes \pi^\vee \otimes \ind_{\Go}^{G}(\bC)$.
				Thus
		\[ \Hom_{\Go}[\pi , \pi \otimes \pi^\vee \otimes \ind_H^{G\times G}(\bC) ] =\Hom_{\Go}[\pi , \pi] \otimes \pi^\vee
		\otimes \ind_{\Go}^{G}(\bC) ,\]
                as a module for $(G/\Go) \times G$ where the actions of $G/\Go$ and $1 \times G$ 
		coincide on $\ind_{\Go}^{G}(\bC)$.
		Therefore \[H^i(G/\Go, \Hom_{\Go}[M,N]) =H^i(G/\Go, \Hom_{\Go}[\pi , \pi] \otimes \ind_{\Go}^{G}(\bC)) \otimes \pi^\vee.\]
		
		Now we know that  $\Hom_{\Go}[\pi , \pi]$ is a representation of $G/\Go$ of dimension $e^2f$, therefore,
		\[\Hom_{\Go}[\pi , \pi] \otimes \ind_{\Go}^{G}(\bC) \simeq e^2f  \ind_{\Go}^{G}(\bC) ,\]
		both as a module for $G\times 1$ and $1 \times G$.
		
		Hence we are reduced to understanding $H^i(\bZ^d, \bC[\bZ^d])$ which is well-known to be zero for $i<d$, and
		$H^d(\bZ^d, \bC[\bZ^d]) = \Ext_{\bZ^d}^d(\bC, \bC[\bZ^d]) \simeq \bC$, completing the proof of the Proposition.
	\end{proof}

\begin{remark}
Later, in Corollary \cref{C:D_h on finlen cusp is contrag}, we will prove that Proposition \cref{dual}
remains true for any finite length cuspidal representation of a $p$-adic group or of a finite cover of it. We take this occasion to mention that the Bernstein block even of a cuspidal representation is not totally trivial:
although it is true that  a finite length indecomposable cuspidal representation of $\wG$ is the successive
extension of a fixed irreducible cuspidal representation $\rho$ of $\wG$, 
it is not true
that such a finite length indecomposable cuspidal representation of $\wG$ is of the form $\rho \otimes \lambda$
where
$\rho$ is an irreducible cuspidal representation of $\wG$,
and 
$\lambda$ is a finite dimensional indecomposable representation of $\wG/\wGo$ (which is the case when $\rho|_{\wGo}$ is irreducible, in which case, indeed this Bernstein block is rather simple). 
\end{remark}
	\section{Basic notions of representation theory}
	We place ourselves in the setting $G=\bG(F)$ for $\bG$ a reductive group over a non-archimedean local field $F$ and $\wG$ a covering group of $G$ (see \cref{SS:central ext}).
	All the basic notions for linear groups have an analogue for the covering group $\wG\to G$.
	
	\subsection{Parabolics, Levi and the Weyl group}
	A parabolic subgroup $\wP$ for $\wG$ is simply the preimage of a parabolic subgroup $P$ of $G$.
	Similarly for Levi subgroups and tori\footnote{Notice that $\wT$ is not necessarily abelian but this is immaterial to us.} in $\wG$. 
	A Levi decomposition $P=LN$ for $P\le G$ lifts to a Levi decomposition $\wP = \wL N$ where $N\le \wG$ is the unique lift of $N$ to $\wG$ that is normalized by $\wL$ (in characteristic zero such a lift is obvious, in general see \cite[Appendix Lemma]{MoeWald}).
	
	For $P$ a parabolic subgroup of $G$ with Levi decomposition $LN$ we denote by $P^-$ the opposite parabolic with Levi decomposition $LN^-$.
	Similarly for their preimages in $\wG$.
	
	The Weyl group of $\wG$ is defined simply to be equal to $W$, the Weyl group of $G$.
		
	\subsection{Parabolic induction and restriction}\label{SS:ind and res}
	The notion of parabolic restriction (Jacquet modules), and parabolic induction, make sense for  $\widetilde{G}$, and they have the same adjointness properties as in the linear case. 
	
	Let $\wP = \wL N$ be a parabolic with a Levi decomposition.
	The (normalized) functors of parabolic induction and parabolic restriction are defined as follows (for details see \cite[VI.1]{Renard} whose convention for the modulus character is opposite that of \cite{BerZel-GL} and \cite{BerZel}, which explains
        interchange of the factors $\delta_\wP^{\pm 1/2}$ of 
\cite{BerZel-GL} and \cite{BerZel}
with their inverses here):
	\begin{align*}
		\bfi_{\wL,\wP}^\wG\colon &\cM(\wL)\to \cM(\wG) \\
		&\pi\mapsto \ind_{\wP}^{\wG}(\pi\otimes\delta_\wP^{-1/2})\\
		\bfr_{\wL,\wP}^{\wG}\colon& \cM(\wG)\to \cM(\wL) \\
		&\tau\mapsto (\Res_{\wP}^\wG (\tau)\otimes\delta_{\wP}^{1/2})_N
	\end{align*}
	where $(-)_N$ is the functor of coinvariants under $N$ and $\delta_\wP\colon \wP\to \bR^\times_{+}$ is the modular character of $\wP$.
	Both functors are exact.
	
	Frobenius reciprocity states that $\bfi_{\wL,\wP}^{\wG}$ is right adjoint to $\bfr_{\wL,\wP}^\wG$.
	As a consequence, parabolic induction preserves limits and parabolic restriction preserves colimits.
	
	It is easy to establish that parabolic induction preserves admissibility (see \cite[III.2.3]{Renard}): this is because the double-coset  $\wK\backslash \wG/\wP$ is finite for any open compact $\wK\le \wG$. 
	The analogous result for parabolic restriction is not so immediate (see \cref{C:res preserves adm}). 
	
	It is also rather straightforward to show that parabolic restriction preserves finite type representations. 
	However, the analogous result for parabolic induction is not so trivial and the proof is based on Bernstein's decomposition \cref{T:Bernstein dec for tildeG}.
	
	We define cuspidal representations of $\widetilde{G}$ just as in the linear case: those smooth representations of
	$\widetilde{G}$  whose all nontrivial parabolic restrictions vanish (i.e., all nontrivial Jacquet modules vanish). 
	
	By a similar argument as in the linear case (see \cite[VI.5.1]{Renard}), the geometric lemma holds for
        covering groups $\wG$ too
	and allows one to calculate the Jordan--Hölder series of the Jacquet modules of a parabolically induced representation of $\widetilde{G}$.
	Then one can conclude as in the linear case that every irreducible representation of $\widetilde{G}$
	is a subquotient of a parabolically induced  representation of $\widetilde{G}$ induced from a cuspidal representation of a Levi subgroup of $\widetilde{G}$.

	\subsection{Iwahori decomposition}
	\begin{definition}\label{D:Iwahori factorization}	
		A compact open subgroup $K$ of $G$ (or of $\wG$) is said to have an Iwahori factorization with respect to a parabolic
		$P=LN$ with opposite parabolic $P^-=LN^-$ if the natural map given by multiplication:
		\[ m\colon K^+ \times K^0 \times K^- \rightarrow K,\]
		is a bijection, where 
		\begin{align*} K^+ &=  K \cap N, \\
			K^0 &=  K \cap L ,\\
			K^- &=  K \cap N^-.
		\end{align*}
		For central extensions $\wG$, we replace $P=LN$ with $\wP=\wL N$ and $\wP^-=\wL N^-$ and everything makes sense to define Iwahori factorization in $\wG$.
	\end{definition}

	\subsection{Cartan decomposition}
	The Cartan decomposition for the group $G$ plays an important role in the proof of the uniform admissibility theorem.
	The analogous theorem for $\wG$ is proved by taking the inverse image from $G$.
	Let us recall some notions before we state the theorem.
	
	Recall the subgroup $\Go\le G$ from \cref{SS:compact mod center}. 
Let
$\wGo$ be the subgroup of $\wG$ which  is the preimage of $\Go\le G$.
	We put $\Lambda(\wG):=\wG/\wGo = G/\Go$; it is a finite rank free $\bZ$-module.
	
	\begin{definition}\label{D:unramified characters}
		The set of unramified characters of $\wG$ is defined to be 
		\[ \cX(\wG):=\Hom(\wG/\wGo,\bC^\times) = \Hom(G/\Go,\bC^\times). \]
	\end{definition}
	It is clear  that the image of $Z(\wG)$ has finite index in $\wG/\wGo$ and hence the rank of
        $Z(\wG)$ is the same as the rank of $Z(G)$.
	In other words, we have $\cX(\wG)\simeq (\bC^\times)^{d_G}$ where $d_G$ is the split rank of $Z(G)$.
	Notice that by definition, the unramified characters of $\wG$ are the same as those of $G$ (identified through the map $\wG \to G$).
	
	\textbf{Construction.}
	For $\bL$, a reductive algebraic group over $F$, we denote by $A_\bL$, the maximal split torus contained in the center of $\bL$, and denote 
	 $d(\bL) :=\dim(A_\bL)$, and call it the split rank of the center of $\bL$.
	We extend this definition to any finite central extension $\wL$ of $L=\bL(F)$, defining
        $d(\wL)=  d(\bL) =\dim(A_\bL)$.
	
	Put $A:=A_\bL(F)$ and notice that the image of $A/A^\circ \to L/L^\circ$ has finite index.
	Moreover, the surjection $A\to A/A^\circ$ admits a section that is a group homomorphism (use a uniformizer in $F$) and so we get an injection $A/A^\circ \hookrightarrow L$ whose image in $L/L^\circ$ has finite index.
	
	If $\wL$ is a finite central extension of $L$, then the natural map $Z(\wL)\rightarrow Z(L)$ induces a map of lattices $\Lambda(Z(\wL))\to \Lambda(Z(L))$ whose image has finite index.
	We can therefore find a lift $\Lambda(Z(\wL))\to \wL$ that is a group homomorphism and moreover its image in $\Lambda(\wL)=\Lambda(L)$ has finite index.
	Summarizing we have a diagram
	\[ \xymatrix{
		\wL \ar[r]& \Lambda(\wL) \ar[r]^{=}  & \Lambda(L)\\
		Z(\wL)\ar@{^(->}[u] \ar[r]& \Lambda(Z(\wL)) \ar@{^(->}[r] \ar@{^(->}[u]& \Lambda(Z(L)) \ar@{^(->}[u]
	} \]
	and from the discussion above we have a section $\Lambda(Z(\wL))\to Z(\wL)$ which is a group homomorphism.
	Denote its image by $C_A$; it is a central subgroup of  $\wL$.
	
	Fix also finitely many elements $F_A:=\{f_1,\dots,f_r\in \wL\}$ 
	that lift some system of representatives for $\Lambda(\wL)/\Lambda(Z(\wL))$.

	All in all, we have that $C_A F_A$ provides a set-theoretic lift of $\Lambda(\wL)$ to $\wL$ that is moreover a group homomorphism when restricted to $\Lambda(Z(\wL))$.
	
	Apply the above discussion to a parabolic subgroup $\wP=\wL N$ and recall the notion of dominant cocharacters of $L$ with respect to $P$.
	Let $L^+$ denote the preimage of dominant cocharacters of $L$ through the map $L\to \Hom_\gr(L,\bG_m)^*$ given by $m\mapsto (\chi\mapsto \chi(m))$.
	We let $\wL^+$ denote its preimage in $\wL$ under $\wL\to L$
	We can choose the above $F_A$ to lie in $\wL^+$ and similarly we can consider $C_A^+\subset C_A$ the subset of dominant elements in $C_A$.
	
	Let $P_0$ be a fixed minimal parabolic subgroup of $G$ with Levi decomposition $P_0=L_0N_0$ and denote by $\wP_0 = \wL_0 N_0$ the corresponding subgroups in $\wG$.

	Let $A=A_0$ be the split component of $P_0$ and $C_AF_A\subset \wL_0$ as above.
	\begin{theorem}[see {\cite{BruTits}}] (Cartan decomposition)
		There exists a maximal compact subgroup
$K_0$ of $G$ such that for its inverse image 
$\wK_0$ inside $\wG$, we have:
		\begin{enumerate}
			\item $\wG = \wP_0 \wK_0 = \wK_0\wP_0$
			\item $\wG = \bigsqcup_{af\in C_A^+F_A} \wK_0 af\wK_0$.
		\end{enumerate}
	\end{theorem}
	Put $\cH_0 = \cH(\wK_0,\wK)$ to be the Hecke algebra of $\wK$-biinvariant functions on $\wK_0$
        where  $\wK\le \wK_0 $ is any compact open subgroup of $\wK_0$.
	It is a finite dimensional algebra over $\bC$.
	
	A rather easy consequence of this theorem, which is essential in the proof of the uniform admissibility
        theorem (\cref{T:uniform adm} below), is the following decomposition of the Hecke algebra:
	\begin{theorem}\label{T:dec Hecke for compact}
	  With the notation as above, let $K\le K_0 \le G$ be  compact open subgroups of $G$, and
          $\wK\le \wK_0 \le \wG$ be the compact open subgroup of $\wG$ obtained by taking their
          inverse images in $\wG$.
		Then the Hecke algebra $\cH(\wG,\wK)$ decomposes as
		\[ \cH(\wG,\wK) = \cH_0 D\cC\cH_0 \]
		where $D$ is a vector subspace spanned by functions indexed by $F_A$ and $\cC$ is a subalgebra isomorphic to the group algebra of $C_A\simeq \Lambda(Z(\wL_0))$.
	\end{theorem}


	\section{Basic theorems}
	As before, $\wG$ is a covering group (see \cref{SS:central ext}) of a reductive $p$-adic group $G$, where $G=\bG(F)$ for $\bG$ a reductive group over $F$, a non-archimedean local field.
	All the basic theorems concerning representations of reductive $p$-adic groups hold also for $\wG$ with the same proofs. 
	We will give precise references to \cite{Renard} where the analogous results for $G$ are proved.
	
	The groups $G$ and $\wG$ do not have any compact representations in general because their centers may be non-compact. 
	However, we can ask for the next best thing (see \cref{SS:compact mod center}):
	\begin{definition}
		An irreducible smooth  representation $(V,\pi)\in \cM(\wG)$ is called compact modulo center if all its matrix coefficients have compact support in $\wG/Z(\wG)$.
	\end{definition}
	The classical theorem of Harish-Chandra still holds with the same proof:
	\begin{theorem} [Harish-Chandra, see {\cite[VI.2.1]{Renard}} for linear groups]\label{T:Harish-Chandra}
		An irreducible smooth
                representation $(\pi,V)$ of $\wG$ is cuspidal if and only if it is compact modulo center.
	\end{theorem}
	
	Using the easy fact that parabolic induction preserves admissibility, the fact that all compact representations are admissible and the above theorem we can prove the admissibility of irreducible representations (for a proof in the linear case see for example {\cite[VI.2.2]{Renard} or \cite[Theorem 12]{BerNotes}}):
	\begin{theorem}\label{T:irred is adm}
		Any irreducible representation of $\wG$ is admissible.
	\end{theorem}
	
	Not only irreducible representations are admissible but in fact they are so in a uniform way. 
	Below is the precise formulation of the Uniform Admissibility Theorem due to Bernstein whose proof is based on the decomposition of the bi-invariant Hecke algebra given in \cref{T:dec Hecke for compact} and some tricky linear algebra:
	\begin{theorem} [uniform admissibility, {\cite{Bernstein-unif-adm}}] \label{T:uniform adm}
		Let $\wK\le \wG$ be an open compact subgroup.
		There exists a constant $c=c(\wK)$ such that for all irreducible representations $(\pi,V)$ of $\wG$ we have
		$\dim(V^\wK)\le c$.
	\end{theorem}
	One can also look at {\cite[VI.2.3]{Renard}} for a proof.
	
	For establishing the Bernstein decomposition, and prior to that, verifying the condition \eqref{Eq:condition (KF)} which leads to \cref{T:dec compact times non-compact}, one makes use of the following 
	\begin{corollary}[see {\cite[VI.2.4]{Renard}} for the linear case]
		Given an open compact subgroup $\wK\le \wG$ there is a compact modulo center subset $\Omega\subset \wG$ such that for any irreducible cuspidal representation $(V,\pi)$ of $\wG$ and any vector $v\in V^\wK$ the function 
		\[ 	g\mapsto  e_{\wK}\cdot\pi(g)(v)\]
		has support contained in $\Omega$.
	\end{corollary}
	
	The following theorem, which goes under the name of  <<\emph{the geometric lemma}>>,
        is extensively used.
      It is due to Bernstein and Zelevinsky \cite[2.12]{BerZel}.
	
	\begin{theorem}[see {    \cite[2.12]{BerZel}       }]\label{T:geometric lemma conseq}
		Let $\wP=\wM N$ and $Q=\wL U$ be two parabolic subgroups of $\wG$ with Levi decompositions.
		Let $\rho$ be an irreducible cuspidal representation of $\wM$ and put $\tau:=\bfr_{\wL,\wQ}^\wG \bfi_{\wM,\wP}^\wG(\rho)$.
		Then we have
		\begin{enumerate}
			\item If $\wL$ has no Levi subgroup conjugate to $\wM$, then $\tau=0$.
			\item If $\wM$ is not conjugate to $\wL$, then $\tau$ has no cuspidal subquotient.
			\item If $\wM$ and $\wL$ are standard and conjugate,
			then $\tau$ has a filtration with subquotients isomorphic to ${}^w\rho$ for $w\in W(\wL,\wM)/W_\wL$.
			In particular, $\tau$ is cuspidal.
		\end{enumerate}
	\end{theorem}
	In the above statement $W_\wL$ is the Weyl group of $\wL$ and $W(\wL,\wM)$ is the subset of the Weyl group of $\wG$ conjugating $\wL$ into $\wM$.  
		
	The next result, proved using the above theorem, shows that the Jordan--Hölder factors of a parabolically induced representation of an irreducible cuspidal are independent of the chosen parabolic and depend only on the conjugacy
        class of the cuspidal datum.
	
	\begin{theorem}[see { \cite[2.9]{BerZel}}]\label{T:JH series of induced of cuspidals}
		Let $\wP=\wL N$ and  $\wP'=\wL' N'$ be two parabolic subgroups in $\wG$ with Levi decompositions.
		Let $\rho\in\Irr_\cs(\wL)$ and $\rho'\in\Irr_\cs(\wL')_\cs$ and $\pi = \bfi_{\wL,\wP}^\wG(\rho)$ and $\pi' = \bfi_{\wL',\wP'}^\wG(\rho')$ be the induced representations.
		Then $\pi$ has finite length and the following are equivalent.
		\begin{enumerate}
			\item The cuspidal data $(\wL,\rho)$ and $(\wL',\rho')$ are conjugate,
			\item The Jordan--Hölder series of $\pi$ and $\pi'$ are equivalent,
			\item The Jordan--Hölder series of $\pi$ and $\pi'$ have a common element.
		\end{enumerate}
		In particular, if $\wL=\wL'$ and $\rho=\rho'$, this shows that $\pi$ and $\pi'$ have the same Jordan--Hölder series, independent of the chosen parabolic.
	\end{theorem}
	
	We can now deduce
	\begin{corollary}
		Parabolic induction preserves finite length representations.
	\end{corollary}
	\begin{proof}
	If the parabolic induction involves a representation which is irreducible cuspidal then the
        corollary  follows from the above theorem. In general, any irreducible representation is a subrepresentation of a parabolically induced representation
       of an irreducible cuspidal representation. We conclude the above corollary
        by the transitivity of parabolic induction.
	\end{proof}
	
	In order to show the analogous result for parabolic restriction, we need to use the following theorem due to
        Howe whose proof for covering groups  is the same as in the linear case:	
	
	\begin{theorem}[see {\cite[Thm 4.1]{BerZel}} or {\cite[VI.6.3]{Renard}}]\label{T:Howe theorem}
	A smooth representation of $\wG$ is of finite length if and only if it is admissible and finitely generated.
	\end{theorem}
	
	A crucial tool that goes into the proof of Howe's theorem is Jacquet's lemma. Later, in \cref{T:gen Jacquet}, we will state a generalized version of it since it is necessary for the second adjointness. 
	Until then, let us give the classical statement
	\begin{theorem} [Jacquet's lemma]\label{T:classical Jacquet lemma}
	Let $V$ be an admissible representation of $\wG$, $\wK\le \wG$ an open compact subgroup admitting an Iwahori decomposition with respect to $\wP=\wL N$.
	Then the following canonical map is surjective
	\[V^{\wK}\to (V_N)^{\wK\cap \wL}.\]
	\end{theorem}
	Immediately from Jacquet's lemma we obtain
	\begin{corollary}\label{C:res preserves adm}
	Parabolic restriction preserves admissibility.
	\end{corollary}

	Use Howe's theorem, the above corollary and the easy fact that parabolic restriction preserves finite type to deduce
	\begin{corollary}\label{C:res preserves finite length}
	Parabolic restriction preserves finite length representations.
	\end{corollary}

	\section{The Bernstein decomposition}\label{S:Bernstein dec}
	We give an exposition of a small part of a theory due to Bernstein which 
	allows one to decompose the category $\cM(\wG)$ of smooth complex 
	representations of $\wG$  as a direct product of certain indecomposable full subcategories,
	now called the Bernstein components of $\cM(\wG)$.
	The results are due to \cite{BerDel} where it is also stated that they hold for finite central extensions of reductive $p$-adic groups. 
	Indeed, no essential modifications are needed to adapt the proof from \cite{BerDel} to the case of finite central extensions. 
	However, in the following we follow mostly the exposition of the linear case from \cite[VI.7]{Renard}  giving
        precise references in appropriate places.
	
	The idea is that whereas $\wG$ does not have compact representations, the group $\wGo$ does,  and moreover, from Harish-Chandra's \cref{T:Harish-Chandra}, all irreducible cuspidal representations of $\wG$ restrict to compact representations of $\wGo$.
	So using the results of \cref{SS:compact mod center}, we can decompose $\cM(\wG)$ into a cuspidal part and an induced part.
	Using induction we express the induced part in a similar way.
	
	\subsection{Cuspidals split}
	In this section we sketch the decomposition of $\cM(\wG)$ into a product of cuspidal and induced. 
	
	We define $\cM(\wG)_\cs$ (resp., $\cM(\wG)_\indu$) to be the full subcategory of $\cM(\wG)$ formed of representations all of whose irreducible subquotients are cuspidal (resp., that have no cuspidal irreducible subquotients).
	The set $\Irr_\cs(\wG)$ denotes the set of irreducible cuspidal representations of $\wG$.

	\begin{remark}
		The category $\cM(\wG)_\cs$  (resp., $\cM(\wG)_\indu$) is the pullback of the subcategory of compact representations (resp., non-compact representations) from $\cM(\wGo)$ through the functor $\Res_\wGo^\wG$.
	\end{remark}
	
	%
	
	Let $(V,\pi)\in\Irr_\cs(\wG)$ and denote by $[\pi]$ its inertia class, i.e., its orbit under $\cX(\wG)$.
	From \cref{P:irred res to Go and inertia classes} we have that the restriction of $\pi$ to $\wGo$ depends only on the inertia class and we know moreover that this restriction has finite length and is semisimple.
	
	\begin{definition}
		We define $\cM(\wG)_{[\pi]}$ to be the full subcategory of $\cM(\wG)$ formed of those representations
		all of whose subquotients belong to $[\pi]$.
		Similarly define $\cM(\wG)_{[\text{\rm out }\pi]}$.
	\end{definition}
	
	\begin{remark}
		Here is a different way of thinking about $\cM(\wG)_{[\pi]}$.
		If $\tau_1,\tau_2,\dots,\tau_l$ are the irreducible summands of $\pi$ restricted to $\wGo$, then  the category $\cM(\wG)_{[\pi]}$ consists precisely in the representations of $\wG$ whose restriction to $\wGo$ are direct sums of $\tau_i, i=1,\dots, l$,
		i.e.,  those
		representations of $\wG$ whose restriction to $\wGo$ belongs to $\cM(\wGo)_{[\cA]}$, where $\cA = \{\tau_i\mid i=1,\dots, l\}$. (See \cref{SS:compact split}.)
	\end{remark}
	
	Harish-Chandra's \cref{T:Harish-Chandra} together with the results from \cref{SS:compact mod center} give
	\begin{theorem}\label{T:cuspidal block}
		For an irreducible cuspidal representation $\pi$ of $\wG$,
		we have a decomposition of the category $\cM(\wG)$ as
		\[ \cM(\wG)\simeq \cM(\wG)_{[\pi]}\times \cM(\wG)_{[\text{out }\pi]}. \]
	\end{theorem}
	
	One uses the uniform admissibility \cref{T:uniform adm} to check the condition \eqref{Eq:condition (KF)} in \cref{SS:compact split} and then applies \cref{T:dec of comp mod center} and \cref{T:dec com mod center x others} to get the following decomposition
	\begin{theorem}[see {\cite[VI.3.5]{Renard}}]\label{T:dec of M(G) into cusp and indu}
		The subcategories $\cM(\wG)_\cs$ and $\cM(\wG)_\indu$ split the category $\cM(\wG)$:
		\begin{align}\label{Eq:dec cuspidal x induced}
			\cM(\wG) = \cM(\wG)_\cs\times \cM(\wG)_\indu 
		\end{align}
		and moreover the cuspidal part decomposes as
		\begin{align}\label{Eq:dec cuspidals}
			\cM(\wG)_\cs = \prod_{[\pi]\in[\Irr(\wG)_\cs]} \cM(\wG)_{[\pi]}.
		\end{align}
	\end{theorem}
	
	Since the center of a product of categories is the product of their centers (see \cref{P:center of product of cats}), we have
	\begin{corollary}
		\begin{align*}
			\cZ(\cM(\wG)) = \prod_{[\pi]\in[\Irr_\cs(\wG)]} \cZ_{[\pi]}\times \cZ(\cM(\wG)_\indu).
		\end{align*}
	\end{corollary}
	
	It turns out that for an irreducible cuspidal representation $\pi$ the center $\cZ_{[\pi]}$ of $\cM(\wG)_{[\pi]}$ is not very hard to determine.
	Recall the notation $\Lambda(\wG) = \wG/\wGo$ and $\cX(\wG) = \Hom_\gr(\Lambda(\wG),\bC^\times)$ and notice that the group algebra $\bC[\Lambda(\wG)]$ is the algebra of regular functions on the algebraic variety (a torus) $\cX(\wG)$.
	
	We first need a lemma whose easy proof follows by  considerations of  the central character.
	\begin{lemma}\cite[V.2.7]{Renard}
		Given $\pi\in\Irr_\cs(\wG)$, 
		its stabilizer in $\cX(\wG)$,
 \[ \cG_\pi:=\Stab_{\cX(\wG)}(\pi) = \left \{\chi \in \cX(\wG)\, \,  \left | \pi \otimes \chi \cong \pi \right .  \right \},\]
is finite.
	\end{lemma}


	\begin{theorem}[see {\cite[VI.10]{Renard}} or {\cite[1.12-1.14]{BerDel} for a slick proof}]\label{T:center for a cuspidal component}
		 Given $\pi\in\Irr_\cs(\wG)$, we have a canonical isomorphism 
		\begin{align*}
			\cZ(\cM(\wG)_{[\pi]})\simeq \bC[\Lambda(\wG)]^{\cG_\pi} = \cO(\cX(\wG)/\cG_\pi).
		\end{align*}
		In particular, $\cZ(\cM(\wG)_{[\pi]})$ is isomorphic to a ring of Laurent polynomials and hence is smooth of Krull dimension equal to the the rank of $\Lambda(\wG)$.
	\end{theorem}
	
	Actually, one can do a bit better and find an equivalence of categories $\cM(\wG)_{[\pi]}\simeq \rmod\cR_{[\pi]}$ with $\cR_{[\pi]} := \End_{\wG}(\Pi_{[\pi]})$ and $\Pi_{[\pi]} := \ind_{\wGo}^\wG (\Res_{\wGo}^\wG(\pi))$.
	This will be discussed in \cref{S:Blocks as module cats}.
	
	\subsection{Induced representations}
	In this section we look at the category $\cM(\wG)_\indu$ and decompose it into blocks.
	The main input is the geometric lemma, see  (\cite[2.12]{BerZel}),   \cref{T:geometric lemma conseq} and \cref{T:JH series of induced of cuspidals} above, 
	 as well as
        \cite[VI.5.1]{Renard} for a proof that works also for $\wG$.
	
	\begin{definition}\hfill
		\begin{enumerate}
			\item 	A cuspidal datum is a couple $(\wL,\rho)$ where $\wL$ is a Levi subgroup of $\wG$ and $\rho\in\Irr(\wL)_\cs$ is an irreducible cuspidal representation of $\wL$.
			\item We say that two cuspidal data $(\wL,\rho), (\wM,\tau)$ are conjugate (or associate) if there exists $g\in\wG$ such that 
			\begin{align*}
				\wL = {}^g\wM  \text{ and } \rho = {}^g\tau,
			\end{align*}
			\item We say that two cuspidal data $(\wL,\rho), (\wM,\tau)$ define the same inertial support if there exists $g\in\wG$ and $\chi\in\cX(\wL)$ such that 
			\begin{align*}
				\wL = {}^g\wM  \text{ and } \rho = {}^g\tau \chi.
			\end{align*}
		\end{enumerate}
	\end{definition}
	
	We denote by $\Omega(\wG)$ the set of cuspidal data up to conjugation and by $\cB(\wG)$ the cuspidal data up to conjugation and inertia.
	
\cref{T:JH series of induced of cuspidals} guarantees that the following notion is well defined.
	\begin{definition}\label{D:cuspidal support}
		Let $\pi\in\cM(\wG)$ be an irreducible representation. 
		We define the cuspidal support of $\pi$ to be a cuspidal datum $(\wL,\rho)\in\Omega(\wG)$ (well-defined
                up to conjugation in $\wG$)  
                such that $\pi$ appears in the Jordan--Hölder series of the parabolic induction $\bfi_{\wL,\wP}^\wG(\rho)$ for some parabolic subgroup $\wP$ with Levi decomposition $\wP=\wL N$.
	\end{definition}

	\begin{remark}
		In order to decompose $\cM(\wG)_\indu$ we would like to pull back, for each cuspidal datum up to conjugation $(\wL,\rho)\in\Omega(\wG)$, the decomposition from \eqref{Eq:dec cuspidal x induced} for $\wL$ to $\wG$ for each irreducible cuspidal representation of $\wL$.
		However, we need to take into account also the inertia because $\cM(\wL)_{[\rho]}$ is an indecomposable subcategory.
	\end{remark}
	
	Recall the set $\cB(\wG)$ of cuspidal data up to conjugation and inertia.
	For $(\wL,\rho)$ a cuspidal datum we denote by $[\wL,\rho]_\wG$ its class in $\cB(\wG)$.
	If no confusion can arise, we drop the subscript $\wG$.
	
	Given $[\wL,\rho]\in\cB(\wG)$, \cref{T:JH series of induced of cuspidals} allows us to define unambiguously the subcategory of representations of $\wG$ all whose irreducible subquotients have cuspidal support in $[\wL,\rho]$:
	\begin{align}
		\cM(\wG)_{[\wL,\rho]} = \left\{\pi\in \cM(\wG)\bigg| \begin{array}{l}\text{ all irreducible subquotients of $\pi$ as a $\wG$-module}\\
			\text{ have cuspidal support  in } [\wL,\rho]\end{array}\right\}.
	\end{align}
	The next lemma gives another characterization of the category $\cM(\wG)_{[\wL,\rho]}$.
	\begin{lemma}\label{L:block as generated by induction from cusp component}
		Let $[\wL,\rho]\in\cB(\wG)$ and $\wP=\wL N$ be a parabolic subgroup with Levi $\wL$.
		Then the subcategory $\cM(\wG)_{[\wL,\rho]}$ is the smallest full subcategory of $\cM(\wG)$ closed under subquotients and containing $\bfi_{\wL,\wP}^\wG(\cM(\wL)_{[\rho]})$.
	\end{lemma}
	\begin{proof}
		It follows easily from \cref{T:geometric lemma conseq} and \cref{T:JH series of induced of cuspidals}.
	\end{proof}
	
	Use Frobenius reciprocity, \cref{T:geometric lemma conseq} and exactness of parabolic induction and restriction to deduce the following proposition.
	\begin{proposition}\label{P:derived orthogonal of components of M(G)}
		If $[\wL,\rho], [\wL',\rho']$ are distinct elements of $\cB(\wG)$,  then the subcategories $\cM(\wG)_{[\wL,\rho]}$ and $\cM(\wG)_{[\wL',\rho']}$ are (derived) orthogonal, i.e., there are no non-zero Ext groups between them.
	\end{proposition}
	
	The above proposition together with \cref{P:categ is product if and only if} imply that we have a fully faithful functor
	\[ \prod_{\fs\in\cB(\wG)}\cM(\wG)_\fs \hookrightarrow \cM(\wG).\]
	Proving that this functor is essentially surjective\footnote{A functor $ F\colon \cC\to \cD$ is called essentially surjective if for any $X\in \cD$ there exists $A\in\cC$ such that $F(A)\simeq X$.} gives the Bernstein decomposition for $\cM(\wG)$:
	\begin{theorem}[see {\cite[VI.7.2]{Renard}}]\label{T:Bernstein dec for tildeG}
		The category of smooth representations of $\wG$ decomposes into blocks indexed by $\cB(\wG)$:
		\begin{equation}\label{Eq:Bernstein decomposition}
			\cM(\wG) = \prod_{{\fs} \in \mathcal{B}(\wG)} \cM(\wG)_\fs.
		\end{equation} 	
	\end{theorem}
	\begin{proof}
		Using \cref{T:dec of M(G) into cusp and indu} what is left to show is that every representation $\pi\in\cM(\wG)_\indu$ can be written as a direct sum of representations $\pi_\fs\in\cM(\wG)_\fs$:
		\begin{align}\label{Eq:decompose rep along subcategories indexed by B(G)}
			\pi \simeq \oplus_{\fs\in\cB(\wG)}\pi_\fs \quad \text{ for some  } \pi_\fs\in\cM(\wG)_\fs.
		\end{align}
		First one observes that \eqref{Eq:decompose rep along subcategories indexed by B(G)} holds  for an induced representation of a cuspidal representation from a Levi.
		This is tautological using \cref{L:block as generated by induction from cusp component}.
		
		Next one notices that if $\pi'\subset \pi$ is a subrepresentation and $\pi$ satisfies \eqref{Eq:decompose rep along subcategories indexed by B(G)} then the same is true of $\pi'$.
		Namely, consider $\pi'_\fs=  \pi_\fs \cap \pi' $, then	$\pi'= \oplus_{\fs\in\cB(\wG)}\pi'_\fs$ as follows easily only using the fact that no nonzero subquotients of  $\pi_\fs$ for different $\fs$ are isomorphic (analogous assertion in group theory is called Goursat's Lemma).

		Finally, one needs to show that every representation $\pi\in\cM(\wG)_\indu$ can be embedded into a sum of induced representations.
		This is done using a pair of adjoint functors that are direct sums of all parabolic induction functors (resp., parabolic restriction) indexed by all the standard parabolic subgroups.
		For details, see \cite[VI.7.2 second Lemme]{Renard}.
	\end{proof}
	
	As a consequence one can show that $\cM(\wG)$ is a noetherian category and this allows one to show that parabolic induction preserves finite type representations.
	Below we sketch the strategy.

	Recall that in an abelian category $\cA$ an object $X$ is said to be 
	\begin{itemize}
	\item noetherian if all increasing sequences of subobjects of $X$ are stationary. 
	\item of finite type if whenever $X=\cup_i X_i$ with $X_i\subset X$,  there exist finitely many indices $i_1,\dots,i_n$ such that $X=\cup_{k=1}^n X_{i_k}$.
	\end{itemize}
	An abelian category is called noetherian if all its finite type objects are noetherian.
	
	Thanks to Bernstein's decomposition theorem and the fact that each cuspidal block is equivalent to modules over a noetherian algebra (see \cref{P:R_rho presentation}) one can show:
	\begin{theorem}[see {\cite[\S 4.19]{BerZel-GL}}, {\cite[Rem 3.12]{BerDel}}, {\cite[VI.7.5]{Renard}}]\label{T:M(G) is noetherian}
	The category of smooth representation $\cM(\wG)$ is noetherian.
	\end{theorem}
	
	Using the above theorem and the simple fact that parabolic restriction preserves finite type (see the discussion in \cref{SS:ind and res}) one shows:	\begin{theorem}[see {\cite[VI.7.5]{Renard}}]\label{T:ind preserves finite type}
	Parabolic induction preserves finite type representations.
	\end{theorem}
	This theorem is used in \cref{L:projective generator of M(G)_s} to show that the generator $\Pi_{\fs}$ of a non-cuspidal block is of finite type.
	
	\section{The second adjointness theorem}\label{S:second adj}
	Recall that the Frobenius reciprocity theorem asserts that for $\wP$ a parabolic subgroup of $\wG$ with Levi subgroup
        $\wL$, the functor $\bfr_{\wL,\wP}^\wG$ is left adjoint to $\bfi_{\wL,\wP}^\wG$.
	It is natural to look for an adjoint in the other direction.
	
	Although showing that $\bfi_{\wL,\wP}^\wG$ admits a right adjoint is not difficult (see below), identifying
        this adjoint precisely is quite hard and is the content of Bernstein's second adjointness theorem: $\bfi_{\wL,\wP}^\wG$ is left adjoint to $\bfr_{\wL,\wP^-}^\wG$, where $\wP^-$ is the opposite parabolic of $\wP$.
	
	In this section we will sketch the main steps in the proof following the exposition in \cite[VI.9]{Renard} where all the details are spelled out for linear groups.
	In loc.cit., the argument is streamlined using the notion of \emph{completion} of a representations of $\wG$ which we recall below.
	Other references for the second adjointness theorem for the linear case include \cite[Theorem 19]{BerNotes},
        \cite{Bush-loc}; a geometric proof is given in \cite{BezKazh2nd}. 
	\subsection{Completion}\label{SS:completion}
	We recall the notion of completion for smooth $G$-modules where $G$ is a locally compact, totally disconnected group admitting a countable basis of neighborhoods of 1 given by compact open subgroups.
	For $V$ a smooth $G$-module its completion $\overline{V}$ is a $G$-module (not smooth anymore)
	sitting between $V$ and $(V^{\vee})^*$.
	
	This notion will play a simplifying role in the proof of the second adjointness theorem and it also helps clarify the relationship between $N$-invariants and $N$-coinvariants.
	
	Notice that for a representation $V$ of $G$ and compact open subgroups $L\subset K$ with corresponding idempotents $e_K,e_L\in\cH(G)$,  we have a natural map $e_L V\to e_K V$ sending $w$ to $e_Kw$.
	
	\begin{definition}
		For $V$, a smooth representation of $G$, we define the functor 
		\[ \overline{(\,\,)}\colon \cM(G)\to \Rep_\bC(G) \text{ by }\]
		\[ \overline{V}:=\varprojlim (e_K V) = \varprojlim (V^K)\]
		where the inverse limit is taken over all compact open subgroups of $G$.
	\end{definition}
A few remarks are in order.
	\begin{remark}	\hfill
		\begin{enumerate}
			\item 		Since the set of compact open subgroups of $G$ is invariant under conjugation by $G$, it follows that  $\overline{V}$is a $G$-module and that $\overline{V}^K=V^K$.
		Clearly $V\subset  \overline{V}$ and the smooth part of  $\overline{V}$ is $V$.	
		\item 	The completion functor is exact because the Mittag-Leffler condition is automatically satisfied.
		\item	Another way of writing the completion functor is to note that 
		 $\Hom_G(\cH(G),V)$
		acquires a $G$-module structure since $\cH(G)$ is a $G \times G$-module,
                and since $\cH(G)=\varinjlim \cH(G)e_K$, it follows that
\[ \overline{V} = \Hom_G(\cH(G),V). \]
		Note that the exactness of the functor $V \rightarrow \overline{V}$
		is equivalent to $\cH(G)$ being a projective $G$-module.

		\end{enumerate}
	\end{remark}

\begin{proposition}\label{P:completion of contragredient} For $V$ a smooth representation of $G$, we have
		\[ \overline{V^\vee} = V^*.\]
\end{proposition}	
\begin{proof}
	
	Suppose given $B\colon V\times W\to \bC$ a pairing of smooth $G$-representations inducing an injective map of representation $V\hookrightarrow W^*$.
	Since $(W^*)^K = (W^K)^*$ for any smooth representation $W$ of $G$ and any compact open subgroup $K \subset G$, we obtain
	\[ \overline{V} = \varprojlim (V^K) \subset \varprojlim(W^K)^* = (\varinjlim W^K)^* = W^*.\]
	
	In particular, taking $B\colon V^\vee\times V\to \bC$ to be 
	the natural pairing and using that $(V^\vee)^K = (V^*)^K$ for all open compact subgroups $K$,  we obtain the claimed equality.
\end{proof}
Applying it to the standard bilinear form $B\colon \cH(G) \times \cH(G) \to \bC,$ we find 
\begin{corollary}
	The completion $\overline{\cH}(G)$ of $\cH(G)$ is the space of distributions $D$ on $G$ such that $e_K*D$ is a compactly supported distribution on $G$ for all idempotents $e_K$. 
\end{corollary}
	In particular, coupled with \cref{P:center module idempot algebra},  we deduce:
\begin{corollary}
 The center of $\cM(G)$ can be described as $\overline{\cH}(G)^G$, i.e.,  the space of invariant distributions $D$ on $G$ such that  $e_K*D$ is a compactly supported distribution on $G$ for all idempotents $e_K$.
\end{corollary}

\begin{remark}
	It is known that for smooth representations taking invariants with respect to a unipotent subgroup $N$ is not a well behaved functor (for example, it is not exact) and as such it is almost never used. 
	On the other hand, taking coinvariants under $N$ leads to nice exact functors (Jacquet functors) which are extremely important in the theory. 
	The completion allows us to reconcile invariants and coinvariants for $N$ as shown below.
\end{remark}	

\begin{proposition}
	The functor $(-)^N\circ \overline{(-)} \circ (-)^\vee\colon \cM(G)\to\Vect_\bC$ is exact.
\end{proposition}
\begin{proof}	
	Notice that we have a natural isomorphism $(V^*)^N = (V_N)^*$ for any representation $V$ of $G$.
	In particular, for a smooth representation $V$ of $G$, using \cref{P:completion of contragredient} we obtain
	\[  \overline{V^\vee}^N = (V^*)^N = (V_N)^*.\]
	Taking $N$-coinvariants is an exact functor because $N$ is an increasing union of compact open subgroups. 
	So the functor $(-)^N\circ \overline{(-)} \circ (-)^\vee$ identifies with the composition of two exact functors $(-)^*\circ (-)_N$ which concludes the proof.
\end{proof}
In other words, taking $N$-invariants on completions of contragredient $G$-modules is an exact functor.
In particular, we deduce
\begin{corollary}
	The functor $(-)^N\circ \overline{(-)}$ is exact on admissible modules.
\end{corollary}
	
\subsection{Adjoints}
	From now on, $G$ is a reductive $p$-adic group and $\wG$ is a finite covering group of it (see \cref{SS:central ext}).
	We fix $\wP = \wL N$, a parabolic subgroup with a Levi decomposition in $\wG$, and we denote by $\wP^-$, the opposite parabolic subgroup with Levi decomposition $\wL N^-$.
	We denote by $\delta_\wP$ the modulus character of $\wP$.
	
	The induction functor $\ind_\wP^\wG \colon \cM(\wP)\to \cM(\wG)$ has a different, more algebraic, expression in terms of the Hecke algebras.
	First, recall that we have an equivalence of categories
	\[ \cM(\wG)\simeq \cH(\wG)\lmod\nd \]
	between smooth representations of $\wG$ and non-degenerate $\cH(\wG)$-modules and that this holds also for $\wP$ and $\wL$.
	
	The following result gives a different incarnation of the induction from $\wP$ to $\wG$:
	\begin{proposition}[{\cite[Théorème III.2.6]{Renard}}]\label{P:ind_P^G as tensor product}
		There is a natural isomorphism of functors 
		\[ \cH(\wG)\otimes_{\cH(\wP)} - \simeq \ind_\wP^\wG \circ (-\otimes_\bC \delta_\wP). \]
	\end{proposition}
	Now consider the functor of taking $\wG$-smooth vectors (or the non-degenerate submodule for the Hecke algebra)
	\[ (-)_{nd\mhyphen\wG}\colon \cH(\wG)\lmod \to \cH(\wG)\lmod\nd\]
	which to a module associates its non-degenerate submodule.
		Similarly for $\cH(\wP)$-modules.
\begin{remark}\label{R:non-deg is right adjoint}
The functor $(-)_{nd\mhyphen\wG}$ is right adjoint to the inclusion functor   \[\cH(\wG)\lmod\nd\hookrightarrow \cH(\wG)\lmod. \]
\end{remark}

	\begin{proposition}\label{P:right adjoint to ind_P^G}
		The functor 
		\[ \ind_\wP^\wG \circ (-\otimes_\bC \delta_\wP)\colon \cM(\wP)\to\cM(\wG) \]
		is left adjoint to
		\[ (-)_{nd-\wP}\circ \Res_\wP^\wG\circ \overline{(-)}\colon\cM(\wG)\to \cM(\wP).\]
	\end{proposition}
	\begin{proof}
	The usual tensor--Hom adjunction has the following analogue here: 
	the functor 
	\[ \cH(\wG)\otimes_{\cH(\wP)} -\colon \cM(\wP)\to\cM(\wG) \]
	is left adjoint to
	\[ (-)_{nd\mhyphen \wP}\circ \Hom_{\cH(\wG)}(\cH(\wG),-)\colon \cM(\wG)\to\cM(\wP) \]
	where $\cH(\wG)$ is viewed as a left $\cH(\wG)$-module and a right $\cH(\wP)$-module.
	
	Therefore the functor $\Hom_{\cH(\wG)}(\cH(\wG),-)\colon \cM(\wG)\to \cH(\wP)\lmod$ is identified with $\Res_\wP^\wG \circ \overline{(-)}$.
	
	Now we conclude using the fact that the image of a non-degenerate module is a non-degenerate module.
	\end{proof}
We immediately get some expression of the right adjoint of parabolic induction:
	\begin{corollary}\label{C:right adjoint of parab ind}
		The functor $\bfi_{\wL,\wP}^\wG\colon \cM(\wL)\to \cM(\wG)$ is left adjoint to 
		 \[ (-)^N\circ (-)_{nd-\wP}\circ \Res_\wP^\wG\circ \overline{(-)}.\]
	\end{corollary}
\begin{proof}
	Use \cref{R:non-deg is right adjoint} and that the right adjoint of the inflation functor $\Res_\wP^\wL\colon \cM(\wL)\to \cM(\wP)$ is $(-)^N$.
\end{proof}

	The key in the proof of Bernstein's second adjointness is the generalized Jacquet Lemma which was proved first by Casselman for admissible representations and later generalized by Bernstein to all smooth representations. 
	See \cite[p.65 Jacquet's Lemma]{BerNotes}, \cite{Bush-loc}, and \cite[VI.9.1]{Renard}.
	The same proof works for finite central extensions:
	\begin{theorem}\label{T:gen Jacquet} 
		Let $\wK\le \wG$ be an open compact admitting an Iwahori decomposition with respect to $\wP=\wL N$.
		Then for any representation $V\in\cM(\wG)$,  the projection
		\[V^\wK\to (V_N)^{\wK\cap \wL} \] 
		is surjective and has a natural (functorial in $V$) section, call it
		$s_{\wK}:  (V_N)^{\wK\cap \wL} \to V^\wK$. 
		Further, for $\wK' \subset \wK$ open compact, both admitting
		Iwahori decompositions with respect to  $\wP=\wL N$,
		we have the commutative diagram (see \cite[VI.9.6.6]{Renard}) :
		\[ 	\begin{tikzcd}
			(V_N)^{\wK' \cap \wL}	\ar[r,"s_{\wK'}"] \ar[d," e_{\wK \cap \wL}  "'] & 	V^{\wK'}  \arrow[d,"e_{\wK}"]\\
			(V_N)^{\wK \cap \wL}	 \ar[r,"s_{\wK}"] &  	V^\wK.
		\end{tikzcd}\]
	\end{theorem}

	Recall that for a $\wG$-module $V$,  we defined its completion to be 
	\[ \overline{V}:=\Hom_\wG(\cH(\wG),V) = \lim_\wK V^\wK \]
	where the limit is taken over all the open compact subgroups $\wK$ of $\wG$.
	The previous theorem is actually equivalent to the following very nice looking statement that is proved for linear groups in \cite[VI.9.7]{Renard}.
	The same proof works for finite central extensions. (The commutativity of the diagram in Theorem
	\cref{T:gen Jacquet} allows one to construct a map  $\overline{V_{N}} \to \overline{V}$ which is then proved to land inside
	the space of $N^-$-invariants.)
	\begin{corollary}\label{C:invar and coinvar are isomorphic}
		If $\wP=\wL N$ is a parabolic with opposite $\wP^- = \wL N^-$ and $V$ is a smooth  representation of $\wG$ then the natural map 
		\[ \overline{V}^{N^-}\to \overline{V_{N}} \] 
		is an isomorphism of $\wL$-representations.
	\end{corollary}
	In representation theory of $p$-adic groups taking invariants under unipotent groups
        is not a well behaved functor but co-invariants are. 
	The following corollary recovers the lost properties of invariants provided we are willing to work with completed representations.
	
	Using \cref{C:invar and coinvar are isomorphic} one can now prove Bernstein's second adjointness theorem easily (we follow the exposition in \cite[VI.9.7]{Renard} for linear groups).
	Other references for the linear case are \cite[Theorem 19]{BerNotes}, \cite{Bush-loc}, and \cite{BezKazh2nd}.
	
	\begin{theorem}\label{T:second adjointness}
		The functor $\bfi_{\wL,\wP}^\wG$ is left adjoint to $\bfr_{\wL,\wP^-}^\wG$.
	\end{theorem}
	\begin{proof}
		We have the following natural isomorphisms:
		\begin{align*}
			\Hom_\wG(\bfi_{\wL,\wP}^\wG(V),W) & 
			= \Hom_\wG(\ind_\wP^\wG \Res_\wP^\wL (\delta_P^{1/2}\otimes V),W)  & \text{ by definition}\\
			&=\Hom_\wG(\cH(\wG)\otimes_{\cH(\wP)}\Res_\wP^\wL(\delta_P^{-1/2}\otimes V),W) & \text{by \cref{P:ind_P^G as tensor product}}\\
			& = \Hom_\wP(\Res_\wP^\wL(\delta_P^{-1/2}\otimes V),\overline{W}) & \text{by \cref{P:right adjoint to ind_P^G}}\\
			& = \Hom_\wL(\delta_P^{-1/2}\otimes V,\overline{W}^N) & \Res_\wP^\wL \dashv (-)^N\\
			& = \Hom_\wL(V,\delta_P^{1/2}\otimes \overline{W_{N^-}})&\text{by \cref{C:invar and coinvar are isomorphic}}\\
			& = \Hom_\wL(V,\delta_P^{1/2} \otimes W_{N^-})& \text{ by \cref{R:non-deg is right adjoint}}\\
			& = \Hom_\wL(V,\bfr_{\wL,\wP^-}^\wG(W)) & \text{ by definition.}
		\end{align*}
	\end{proof}

	Since every functor that admits an exact right adjoint preserves projective objects, we deduce the rather non-trivial fact:
	\begin{corollary}\label{C:induction preserves projectives}
		The functor $\bfi_{\wL,\wP}^\wG\colon \cM(\wL)\to \cM(\wG)$ sends projective objects to projective objects.
	\end{corollary}
	
	
	\section{Blocks as module categories}\label{S:Blocks as module cats}
	In this section we describe each block $\cM(\wG)_\fs$, $\fs\in\cB(\wG)$, as the category of modules over some algebra $\cR_\fs$ and prove some basic homological properties of this algebra.
	
	The key results of this section are \cref{C:Frob sym cond for R_rho} and \cref{P:vanishing of ext over R}.

	\subsection{Cuspidals}
	Fix $(\wL,\rho)$  a cuspidal data and denote by $\fs\in\cB(\wG)$ its equivalence class (conjugation and inertia).

	Define the following representation of $\wL$:
	\[ \Pi_{[\rho]}:=\ind_{\wLo}^\wL (\rho\mid_{\wLo}) =\rho \otimes  \ind_{\wLo}^\wL (\bC),
        \]
	where we recall that $\wLo$ is the subgroup of $\wL$ generated by all compact subgroups.

	\begin{proposition}\label{P:projective generator for cuspidal block}
		The representation $\Pi_{[\rho]}$ is a finite type projective generator in $\cM(\wL)_{[\rho]}$.
	\end{proposition}
	\begin{proof}
		First we show the projectivity.
		Since $\wLo\le\wL$ is an open subgroup, the functor $\ind_{\wLo}^\wL$ is also left adjoint to $\Res_{\wLo}^\wL$ (see \cref{L:induction from open adjunction}) which is exact, hence $\ind_{\wLo}^\wL$ preserves projective objects. 
		It is therefore enough to show that $\rho|_\wLo$ is projective which follows immediately from  Harish-Chandra's \cref{T:Harish-Chandra} and the fact that compact representations are projective by \cref{S:splitting reps}.
		
		Let us now show that $\Pi_{[\rho]}$ is a generator of $\cM(\wL)_{[\rho]}$.
		Let $\pi$ be an arbitrary representation in $\cM(\wL)_{[\rho]}$.
		We must show that there's a non-zero morphism $\Pi_{[\rho]}\to \pi$.
		Recall (paragraph before \cref{T:cuspidal block}) that the category $\cM(\wL)_{[\rho]}$ is defined as consisting of those smooth representations of $\wL$ all  whose irreducible  subquotients belong to $[\rho]$.
		Since  $\Pi_{[\rho]}$ is a projective representation of $\wL$, it is  enough to prove that
		an irreducible representation belonging to $\cM(\wL)_{[\rho]}$, thus  isomorphic to $\chi\rho$ for some unramified character $\chi\in\cX(\wL)$, appears as a quotient of $\Pi_{[\rho]}$.
		By the adjunction \cref{L:induction from open adjunction}, we have
		\[ \Hom_{\wL}(\Pi_{[\rho]},\chi\rho) = \Hom_\wLo(\rho\mid_{\wLo},\chi\rho\mid_\wLo), \]
		and the latter is non-zero as it contains the identity.
		
		In order to show that
		$\Pi_{[\rho]}$
		is of finite type, one notices that $\wLo\backslash \wL$ is a discrete group and leads to a basis of $\ind_\wLo^\wL(\rho\mid_\wLo)$ indexed by these cosets and a basis of $\rho$.
		The action of $\wL$ permutes this basis and hence, by choosing a finite generating set of $\rho\mid_\wLo$, we obtain a finite generating set of $\Pi_{[\rho]}$ as an $\wL$-representation, proving that $\Pi_{[\rho]}$ is of finite type.
	\end{proof}
	
	We put $\cR_{[\rho]}:=\End_\wL(\Pi_{[\rho]})$. Generalities from
        category theory (see \cref{P:equiv module category}), give us the following corollary:
	\begin{corollary}		
	 We have an equivalence of categories
	\begin{align*}
		\cM(\wL)_{[\rho]} &\longrightarrow \rmod\cR_{[\rho]}\\
		\pi &\mapsto \Hom_\wL(\Pi_{[\rho]},\pi).\notag
	\end{align*}
	\end{corollary}

	
	Below we give a description of the algebra $\cR_{[\rho]}$ by generators and relations.
	
	Recall that $\cG_\rho\subset  \Hom_\gr(\Lambda,\bC^\times) = \cX(\wL)$ denotes the stabilizer of $\rho$, where $\Lambda = \wL/\wLo$, i.e.,  $\cG_\rho = \{\chi \in  \Hom_\gr(\Lambda,\bC^\times) | \rho \otimes \chi \simeq \rho\}$.
	The group $\cG_\rho$ naturally acts on the complex torus  $\cX(\wL)$
by translations, and hence also on the
algebra of regular functions on $\cO(\cX(\wL)) = \bC[\Lambda]:= \cA$ which we denote by $\chi(f)= {}^\chi f$ for
$\chi \in \cG_\rho$. By Schur's lemma, the isomorphism $\rho \otimes \chi \simeq \rho$ defines a unique map -- up to scalars --
from the space
underlying $\rho$ to itself, hence provides us with a 2-cocycle  $ c \in H^2(\cG_\rho, \bC)$, which in turn defines a
twisted group algebra $\bC[\cG_{\rho},c]$, the algebra generated by the intertwining operators from $\rho$ to itself
through the isomorphisms $\rho \otimes \chi \simeq \rho$.

\begin{proposition}[{\cite[Proposition 28]{BerNotes}}]
		\label{P:R_rho presentation}
		The algebra $\cR_{[\rho]}$ has the following presentation:
		\begin{enumerate}
			\item As a vector space $\cR_{[\rho]} = \cA\otimes \bC[\cG_{\rho},c]$,
			\item $\cA$ and $\bC[\cG_{\rho},c]$ are subalgebras,
			\item $f b_\chi = b_\chi {}^\chi f$ for all $f\in\cA$ and $\chi\in \cG_{\rho}$, with
                        $b_\chi = 1 \otimes \chi \in  \cA\otimes \bC[\cG_{\rho},c] =\cR_{[\rho]}$.
		\end{enumerate}
\end{proposition}
\begin{proof}
	The proof is fairly elementary but we will give a sketch.

	First notice that by adjunction, we have a natural isomorphism of vector spaces
	 \begin{align}\label{Eq:R_rho as v space}
	 	\cR_{[\rho]} = \End_\wL(\ind_\wLo^\wL(\rho|_\wLo)) \simeq \End_\wLo(\rho|_\wLo)\otimes \cA.
	 \end{align}
	
Since	$\End_\wL(\ind_\wLo^\wL(\bC)) = \cA$, and since $\ind_\wLo^\wL(\rho|_\wLo)\simeq \rho\otimes \ind_\wLo^\wL(\bC)$, this exhibits $\cA$ as a subalgebra of $\cR_{[\rho]}$ and moreover it shows that \eqref{Eq:R_rho as v space} is an isomorphism of right $\cA$-modules.
	
	Second, notice that $\Lambda=\wL/\wLo$ acts on $\End_\wLo(\rho|_\wLo)$ and the action factors through the finite quotient $\wL/Z(\wL)\wLo$.
	Moreover the action is diagonalizable and the eigenvalues (i.e., characters of $\Lambda$) that appear are, by definition, exactly the elements of $\cG_\rho$.
	We deduce a canonical isomorphism of vector spaces 
	\[ \End_\wLo(\rho|_\wLo) \simeq \bC[\cG_\rho]. \]
	If $b_\chi,b_\mu\in \End_\wLo(\rho|_\wLo)$ correspond to $\chi,\mu\in\cG_\rho$ through the previous isomorphism, then their product belongs to the eigenspace corresponding to $\chi\mu$, i.e., we have $b_\chi b_\mu = c(\chi,\mu)b_{\chi\mu}$ for some non-zero $c(\chi,\mu)\in\bC^\times$.
	Associativity plus unity implies that $c(-,-)$ is a two cocycle of $\cG_\rho$ with values in $\bC^\times$.
	This allows us to construct a twisted group algebra $\bC[\cG_\rho,c]$ and in turn provides us with an isomorphism of algebras
	\[ \End_\wLo(\rho|_\wLo) \simeq \bC[\cG_\rho,c].\]

	The last step consists in showing how the subalgebra $\cA$ and $\bC[\cG_\rho,c]$ interact inside $\cR_{[\rho]}$.
	This is done by writing down explicitly the isomorphism in \eqref{Eq:R_rho as v space}.
\end{proof}
	
	Recall (see \cite[Proposition III.12.1]{Artin-nonc}) that an Azumaya algebra $R$ over a commutative ring $Z$ with identity
        is a $Z$-algebra such that for some faithfully flat extension of commutative rings $Z\hookrightarrow Z'$ we have $R\otimes_ZZ'\simeq M_n(Z')$.
	Equivalently, $R$ is a projective module of constant rank over $Z$ and the multiplication map $R\otimes_Z R^\op \to \End_Z(R)$ is an isomorphism of rings.

	In particular, one deduces in a straightforward way the following proposition 
        from the above presentation of $\cR_{[\rho]}$.
        
	\begin{proposition}\label{P:R_rho is Azumaya and Z Laur pol}
		The natural inclusion $\cA^{\cG_\rho}\hookrightarrow \cR_{[\rho]}$ identifies $\cA^{\cG_\rho}$ with the center $Z(\cR_{[\rho]})$.
		In particular $Z(\cR_{[\rho]})$ is isomorphic to a Laurent polynomial algebra.
		Moreover $\cR_{[\rho]}$ is an Azumaya algebra over its center.
	\end{proposition}
	For a different proof, see \cite[Proposition 8.1]{Bus-Hen}. 
	Note that our algebra $\cR_{[\rho]}$ is a matrix algebra over the algebra $E_G$ of \cite{Bus-Hen}.

	The next lemma is easy and is valid for any Azumaya algebra (see \cite[Proposition IV.2.1]{Artin-nonc}).
	\begin{lemma}\label{L:Azumaya are selfdual}
		Let $R$ be an Azumaya algebra over its center $Z$ which is a noetherian ring.
		Then the trace map $\tr \colon R \to Z$ provides an isomorphism of $R$-bimodules
		\[ R\to \Hom_Z(R,Z).\]
	\end{lemma}
	\begin{proof}
		First,  this is obvious for $R=M_n(Z)$.
		Using faithfully-flat descent, one first defines the trace map unambiguously for any Azumaya algebra
		$R$        and then sees that it provides a non-degenerate symmetric pairing $R\otimes_Z R\to Z$ which leads to the isomorphism asserted in  the statement of the lemma.
	\end{proof}
	
	Applying this to our algebra $\cR_{[\rho]}$ with center $\cZ_{[\rho]}:=Z(\cR_{[\rho]})$ which is a Laurent polynomial ring, we get the following corollary (see \eqref{Cond: Hom(R,A)=R as bimodules}):
	\begin{corollary}\label{C:Frob sym cond for R_rho}
		There is an isomorphism of $\cR_{[\rho]}$-bimodules
		\[\cR_{[\rho]}\simeq \Hom_{\cZ_{[\rho]}}(\cR_{[\rho]},\cZ_{[\rho]}).\]
		In other words, $\cR_{[\rho]}$ satisfies condition (FsG) in Definition \ref{FsG}.
	\end{corollary}

	Let us now apply what we have learned about $\cR_{[\rho]}$ to prove the following proposition which plays a crucial role in the study of the homological duality functor (see \cref{T:homological duality single degree}).
	
	Recall that for $\rho$ a cuspidal representation of $\wL$, we defined $d(\rho)$ to be the split rank of the center of $L$ (which is a Levi subgroup of $G$).
	\begin{proposition}\label{P:vanishing of ext over R}
		For $V$, a $\cR_{[\rho]}$-module which is finite dimensional over $\bC$,  we have
		\begin{align*}
			\Ext^i_{\cR_{[\rho]}}(V,\cR_{[\rho]}) = 0 \text{ for all }i\neq d(\rho).
		\end{align*}
	\end{proposition}
	\begin{proof}
		For simplicity we will put $\cR = \cR_{[\rho]}$ and $\cZ:=Z(\cR_{[\rho]})$ in this proof.
		The forgetful functor 
		\begin{align*}	 
			F\colon \rmod\cR\to  \rmod\cZ 
		\end{align*}
		can also be written as $F = -\otimes_{\cR}\cR$ where ${}_\cR\cR_\cZ$ is viewed as an $\cR\mhyp\cZ$-bimodule.
		As such, $F$ has a right adjoint $F':=\Hom_\cZ(\cR_\cZ,-)$.
		
		Since $\cR$ is an Azumaya algebra (\cref{P:R_rho is Azumaya and Z Laur pol}), it is a projective $\cZ$-module, hence  both $F$ and $F'$ are exact functors.
		It follows that the adjunction extends to Ext groups.
		Applying \cref{C:Frob sym cond for R_rho} we therefore get
		\[ \Ext^i_\cR(V,\cR) = \Ext^i_\cR(V,F'(\cZ)) = \Ext^i_\cZ(V,\cZ),\]
		and now we conclude the proof of the proposition by a well-known result in commutative algebra
                since $\cZ$ is a Laurent polynomial algebra (by \cref{P:R_rho is Azumaya and Z Laur pol})
                of Krull dimension $d(\rho)$ and $V$ is a
$\cZ$-module which is a
finite dimensional vector space  over $\bC$. 
	\end{proof}
	\begin{remark}\label{R:what makes vanishing}
		The statement holds with the same proof if we only assume that $\cR_{[\rho]}\simeq\Hom_{\cZ_{[\rho]}}(\cR_{[\rho]},\cZ_{[\rho]})$ as right $\cR_{[\rho]}$-modules.
	\end{remark}
	
	\subsection{Non-cuspidal blocks: a projective generator}\label{SS:induced proj gen}
	Moving on to the block in $\wG$ corresponding to $\fs=[\wL,\rho]\in\cB(\wG)$, let us define the representation
	\[ \Pi_\fs:=\oplus_{\wP} \bfi_{\wL,\wP}^\wG(\Pi_{[\rho]}) = \oplus_{\wP}\bfi_{\wL,\wP}^{\wG} (\ind_{\wLo}^{\wL}(\rho|_{\wLo}))\]
	where the sum ranges over all parabolics $\wP$ with Levi $\wL$ (it is a finite sum).
	
	\begin{remark}
		Actually it would be enough to consider $\bfi_{\wL,\wP}^\wG(\Pi_{[\rho]})$ for a single parabolic as this module turns out to be (non-canonically) independent of the parabolic containing $\wL$.
		See  \cite[p. 96]{BerNotes}) for a proof as well as
                \cite[VI.10.1]{Renard} for some more details. 
                However, we will not need it.
		One advantage of taking the direct sum instead of a single parabolic induction is that it makes the isomorphism from \cref{P:homological dual of progren Pi_s} canonical.
	\end{remark}
	
	%

	\begin{lemma}\label{L:projective generator of M(G)_s}
		The module $\Pi_\fs$ is a finite type projective-generator of $\cM(\wG)_\fs$.
	\end{lemma}
	\begin{proof}

		The projectivity follows at once as a consequence of the second-adjointness \cref{C:induction preserves projectives} together with \cref{P:projective generator for cuspidal block}: the functor $\bfi_{\wL,\wP}^{\wG}$ has an exact right adjoint, hence it preserves projectives and $\Pi_{[\rho]}$ is projective. 
		Since parabolic induction preserves finite type (see \cref{T:ind preserves finite type}) we have that $\Pi_\fs$ is also finitely generated.
		
		Let $\pi\in \cM(\wG)_\fs$ be a  non-zero representation of $\wG$.
		We need to show that $\Hom_{\wG}(\Pi_\fs,\pi)$ is non zero and for that,
                since $\Pi_\fs$ is projective, it is enough to consider the case $\pi$ irreducible. 
		By definition of the block $\cM(\wG)_\fs$, $\bfr_{\wL,\wP^-}^\wG(\pi)$ is non-zero in the block
                $\cM(\wL)_{[\rho]}$.
		The second-adjointness \cref{T:second adjointness} and the fact that $\Pi_{[\rho]}$ is a generator of $\cM(\wL)_{[\rho]}$ (see \cref{P:projective generator for cuspidal block}) imply that
		\[ \Hom_\wG( \bfi_{\wL,\wP}^\wG(\Pi_{[\rho]}),\pi) = \Hom_\wL(\Pi_{[\rho]}, \bfr_{\wL,\wP^-}^\wG(\pi))\neq 0 \]
		and hence $\Hom_\wG(\Pi_\fs,\pi)\neq 0$.
	\end{proof}
	
	
	An abelian category admitting all coproducts and having a finitely generated (compact) progenerator is equivalent to the category of right modules over its endomorphism ring (see \cref{P:equiv module category}).
	Putting $\cR_\fs:=\End_{\wG}(\Pi_\fs)$, we thus deduce the following:
	\begin{proposition}\label{P:block equiv to modules over algebra}
		The functor
		\begin{align*}
			\cM(\wG)_\fs &\to \rmod\cR_\fs\\
			V&\mapsto \Hom_{\wG}(\Pi_\fs,V)
		\end{align*}
		is an equivalence of categories sending finite length representations of $\wG$ in
                $\cM(\wG)_\fs$
to $\cR_\fs$-modules which are finite dimensional over $\bC$.
	\end{proposition}
	
	\subsection{Center of $\cM(\wG)_\fs$}\label{SS:center of a block}
	This section plays no role in this work but we record it for completeness.
	We give the description of the center of the algebra $\cR_\fs$ for $\fs = [\wL,\rho]\in\cB(\wG)$.
	In turn this determines the center of the block $\cM(\wG)_\fs$ (see \cref{P:block equiv to modules over algebra}, \cref{R:equiv of cats center} and \cref{Ex:center of A-mod}).
	
	Recall that $\cX(\wL)$ acts on the irreducible cuspidals $\Irr(\wL)_\cs$ and we denoted by $\cG_\rho$ the stabilizer of $\rho$.
	Denote by $W_{[\rho]}:=\Stab_{W_{\wL,\wG}}([\rho])$,  the stabilizer of the inertia class $[\rho]$ in the relative Weyl group $W_{\wL,\wG}$.
	
	Recall also the $\bZ$-lattice $\Lambda(\wL) = \wL/\wLo = \Lambda(L) = L/\Lo$.
	The characters of $\Lambda(\wL)$ form the torus $\cX(\wL)$  whose ring of regular functions is identified with the group algebra of $\Lambda(\wL)$:
	\[ \bC[\Lambda(\wL)] = \cO(\cX(\wL)). \]
	
	The linear case of the following theorem is due to \cite[2.12]{BerDel} but the proof works also for finite central extensions. 
	Another reference is \cite[VI.10.4]{Renard}.
	\begin{theorem} \label{T:center of M(G)_s as invariants}
		The algebra $\cR_\fs$ contains the Laurent polynomial algebra $\cO(\cX(\wL)/\cG_\rho)$ as a subalgebra and it is finite as a left (or right)-module over it.
		Moreover, the center of $\cR_\fs$ is nothing but
		\[ \cZ(\cR_\fs) =\cO(\cX(\wL)/\cG_\rho)^{W_{[\rho]}}. \]
	\end{theorem}
	Since $\cZ(\cR_\fs)$ has no non-trivial idempotents, we immediately get the following corollary.
	\begin{corollary}\label{L:block M(G)_s is indec}
		The category $\cM(\wG)_\fs$ is indecomposable.
	\end{corollary}
	
	\subsection{Revisiting Howe's theorem}
	Since Bernstein's center plays a prominent role in the theory, it seems worthwhile mentioning the following applications whose proofs are parallel to the linear case. This section is only for completeness, it is not used in the sequel.
	
	We say that a smooth representation $V$ of $\wG$ is $\cZ(\wG)$-admissible (where $\cZ(\wG)$ is the Bernstein
        center of the category of smooth representations of $\wG$)
        if for any open compact $\wK\le \wG$,  the $\cZ(\wG)$-module $V^{\wK}$ is finitely generated.
	We have
	\begin{theorem}[{\cite[VI.10.8]{Renard}}]\label{T:finite type=Z adm and fin comp}
	A smooth representation $V$ of $\wG$ is of finite type if and only if it is $\cZ(\wG)$-admissible and it has only finitely many non-zero Bernstein components.
	\end{theorem}
	We say that a representation is $\cZ(\wG)$-finite if the action of $\cZ(\wG)$ on it factors through an ideal of finite codimension in  $\cZ(\wG)$.

	The following is a generalization of Howe's theorem.
	\begin{theorem} Let $V$ be a smooth representation of $\wG$.
	Any two of the following properties implies the third, and then $V$ is of finite length:
	\begin{enumerate}
		\item $V$ is of finite type,
		\item $V$ is admissible,
		\item $V$ is $\cZ(\wG)$-finite.
	\end{enumerate}
	\end{theorem}
	\begin{proof}
	Howe's theorem says that (1)+(2) imply finite length and hence by Schur's lemma we have (3).
	
	Assume (1)+(3). \cref{T:finite type=Z adm and fin comp} implies that $V$ is $\cZ(\wG)$-admissible. Hence for any open compact $\wK\le \wG$ the invariants $V^\wK$ are finitely generated over $\cZ(\wG)$ and now (3) implies that $V^\wK$ is finite dimensional.
	
	Assume (2)+(3). Clearly $V$ is $\cZ(\wG)$ admissible and since the action of $\cZ(\wG)$, a unitary ring, factorizes through a finite codimension ideal we get that $V$ has only finitely many non-zero Bernstein components. We conclude by \cref{T:finite type=Z adm and fin comp}.
	\end{proof}

	\section{Homological duality}\label{S:homological duality}
	\subsection{Vanishing}\label{SS:vanishing}
	Here we prove the vanishing part of our main result, namely \cref{T:homol prop of D_h-intro}\eqref{T:subpoint:vanishing Ext for D_h},
	following the strategy in \cite{BerNotes}.
	
	The main ingredients are \cref{S:Bernstein dec} on Bernstein's decomposition, Bernstein's second adjoint \cref{T:second adjointness},
	and the vanishing result from \cref{P:vanishing of ext over R}.
	
	Consider
	$\cM(\wG)_{fg}$ the full subcategory of smooth representations of $\wG$ which are finitely generated. 
	It is still an abelian category and the parabolic induction and restriction functors preserve it.
	The homological duality is defined as the following functor between derived categories
	
	\begin{align}
		D_h\colon &\cD^b_{fg}(\cM(\wG))\to \cD^b_{fg}(\cM(\wG))^\op\\
		& \pi \mapsto \RHom_\wG(\pi,\cH(\wG))\notag
	\end{align}
	where $\cH(\wG)$ denotes the Hecke algebra of $\wG$.
	The structure of left $\wG$ representation is through the right action of $\cH(\wG)$ on itself and through the involution $f\mapsto \breve{f}\colon \cH(\wG)\to \cH(\wG)$ defined by $\breve{f}(g):=f(g\inv)$.
	
	\begin{remark}\hfill
		\begin{enumerate}
			\item 	Notice that $D_h$ lands indeed in the bounded derived category because the category of smooth representations $\cM(\wG)$ has finite global dimension (\cref{SS:finite homological dimension}).			
			\item 	One could also define $D_h$ without the assumption on finitely generated but then $D_h$ would land outside smooth modules, indeed in the category of all $\cH(\wG)$-modules. 
			We could come back to smooth modules simply by taking the smooth part (see the paragraph after \cref{P:ind_P^G as tensor product}).
		\end{enumerate}
	\end{remark}

	The following theorem is the most important property of the homological duality functor and is due to Bernstein.
	Fix $\fs = [\wL,\rho]\in\cB(\wG)$, a cuspidal datum.
	\begin{theorem}\label{T:homological duality single degree}
		If  $\pi\in\cM(\wG)_\fs$ is of finite length, then $D_h(\pi)$ has cohomology only in degree $d(\fs)$.
	\end{theorem}
	\begin{proof} (following \cite[Theorem 31]{BerNotes})
		We need to show that $\Ext^i_\wG(\pi,\cH) = 0$ for $i\neq d(\fs)$, where we put for short $\cH = \cH(\wG)$.
		
		Decompose the representation $\cH(\wG)$ according to Bernstein's decomposition,
                \cref{T:Bernstein dec for tildeG}.
		Clearly only the component $\cH(\wG)_\fs$ is important to us.
		Moreover, since $\cH(\wG)$ is a projective representation (\cref{R:H(G) is projective}),
                and $\Pi_\fs$ is a projective generator of $\cM(\wG)_\fs$ (see \cref{L:projective generator of M(G)_s}), it is enough to show:
		\begin{align}\label{Eq:Ext vanishing with Pi_fs}
			\Ext^i_\wG(\pi,\Pi_\fs)=0,\text{ for all }i\neq d(\fs).
		\end{align}
		Recall from \S\ref{SS:induced proj gen} that $\Pi_\fs = \bfi_{\wL,\wP}^{\wG}(\Pi_{[\rho]})$.
		Since parabolic induction and restriction are exact functors, the Frobenius adjunction gives
		\begin{align}\label{Eq:Ext with Pi_fs equal with Pi_rho}
			\Ext^i_\wG(\pi,\Pi_\fs) = \Ext^i_\wL(\bfr_{\wL,\wP}^\wG(\pi),\Pi_{[\rho]}).
		\end{align}
		The representation $\pi$ is of finite length, hence by \cref{C:res preserves finite length} the same is true of $\bfr_{\wL,\wP}^\wG(\pi)$.
		Through the categorical equivalence $\cM(\wL)_{[\rho]}\simeq \rmod\cR_{[\rho]}$ (see \cref{P:block equiv to modules over algebra}), the vanishing of \eqref{Eq:Ext with Pi_fs equal with Pi_rho} follows from 
		\[ \Ext^i_{\cR_{[\rho]}}(V,\cR_{[\rho]}) = 0, \text{ for all }i\neq d(\rho)=d(\fs), \]
		for all finite dimensional $\cR_{[\rho]}$-modules $V$.
		This was proved in \cref{P:vanishing of ext over R}.
	\end{proof}

	Given $\pi\in\cM(\wG)_\fs^\finl$ a finite length representation in a given block, we will denote by $\bD_h(\pi)$ the representation $H^{d(\fs)}(D_h(\pi))$. 
	
	\subsection{Interaction with induction and restriction}\label{SS:Dh and ind res}
	The objective of this section is to investigate how does the (full, derived) homological duality commute with parabolic induction and restriction.
	We follow the proof of the linear case from \cite[Theorem 31(4,5)]{BerNotes}.
	\begin{proposition}\label{P:D_h an ind res}
		Let $\wP = \wL N$ be a parabolic with Levi decomposition in $\wG$.
		We have the following natural isomorphisms of functors when restricted to the bounded derived category of finitely generated smooth $\wG$-modules $\cD^b_{fg}(\cM(\wG))$:
		\begin{enumerate}
			\item $D_h \bfi_{\wL,\wP}^\wG \simeq \bfi_{\wL,\wP^-}^\wG D_h$,
			\item $D_h \bfr_{\wL,\wP}^\wG \simeq \bfr_{\wL,\wP}^\wG D_h$.
		\end{enumerate}
	\end{proposition}
	\begin{proof}
		To shorten the notation, we write $\bfi_\wP$ and $\bfr_\wP$ for the parabolic induction and restriction functors.
		
		Since all objects in $\cM(\wG)_{fg}$ admit a finite resolution by finitely generated
		projective objects, it is enough to define and prove the required isomorphisms for finitely
		generated projective objects.
		Therefore, using the second adjointness and the Frobenius reciprocity,  
		we see that the isomorphisms in part (1) and in part (2) of the proposition are equivalent to proving natural isomorphisms 
		\begin{align}
			\bfi_{\wP^-} \Hom_\wL(V,\cH(\wL)) &\simeq \Hom_\wL(V,\bfr_{\wP^-}(\cH(\wG))),\label{Eq:ind Hom}\\
			\bfr_\wP\Hom_\wG(W,\cH(\wG)) & \simeq \Hom_\wG(W,\bfi_\wP(\cH(\wL))),\label{Eq:res Hom}
		\end{align}
		for $V\in\cM(\wL)$ and $W\in\cM(\wG)$, both projective, and finitely generated representations, and where
		we have considered $\cH(\wL)$ as an $\wL\times \wL$ module and $\cH(\wG)$ as a $\wG\times \wG$-module.

		For the proof of these isomorphisms,  we begin by noticing
		the following identifications of $\wL\times\wG$-modules: 
		\begin{align}\label{Eq:ind of H(L)}
			(\id\times\bfi_\wP)(\cH(\wL)) \simeq \Ind_{ \wL \times \wP}^{\wL \times \wG}(\cH(\wL)\delta^{-1/2})   \simeq \delta^{1/2} \otimes \cC^\infty_c(\wG/N) 
			\simeq (\id\times\bfr_\wP)(\cH(\wG)), 
		\end{align}
		where in the first appearance, $\delta$ is a character of $\wL \times \wP$ trivial on $\wL$, whereas in the second appearance, $\delta$ is a character of $\wL \times \wG$ trivial on $\wG$; this also explains why
                the power of delta changes sign since $\cH(\wL) = \ind^{\wL \times \wL}_{\Delta(\wL)} \bC$.

		The following is an easy but crucial observation valid for any finitely generated
		smooth representation $V$ of $\wL$ (we will apply this to finitely generated projective modules
		as noted earlier):
		\begin{align} \label{Eq:easy crucial obs}
			\bfi_\wP\Hom_\wL(V, \cH(\wL)) &\simeq \Hom_\wL(V,(\id\times\bfi_\wP) (\cH(\wL)), \end{align}
		as representations of $\wG$ where
		in $\Hom_\wL(V,\cH(\wL))$,  both $V$ and $\cH(\wL)$ are considered as left $\wL$-modules,
		thus $\Hom_\wL(V,\cH(\wL))$ is a right $\wL$-module through the right $\wL$-action on $\cH(\wL)$; the
		$\wL \times \wG$-representation $(\id\times\bfi_\wP) (\cH(\wL)) = \bfi_\wP^{\wG} (\cH(\wL))$ has
		$\wL$ action through the left action
		of $\wL$ on  $\cH(\wL)$, and a right  $\wG$ action.
		To see the isomorphism in equation \eqref{Eq:easy crucial obs}, note that 
		by the definition of an induced representation, both sides of \eqref{Eq:easy crucial obs}
		give rise to  functions
		$F\colon \wG \times V \rightarrow \cH(\wL)$ which are linear maps $V \rightarrow \cH(\wL)$ when restricted
		to any $g \in \wG$, and satisfy:
		\begin{enumerate}
			\item $F(g, \ell v) = \ell F(g,v)$, for all $g\in \wG, \ell \in \wL, v \in V$.
			\item $F(gp, v) = F(g,v)\cdot \ell$ where $p \in \wP = \wL N$ has the form $p=\ell n$, with $\ell \in \wL$, and $n \in N$.
		\end{enumerate}
		
		Further, such an $F$ arises from the left hand side of the isomorphism in \eqref{Eq:easy crucial obs} if and only if
		the corresponding map $\wG \rightarrow \Hom (V, \cH(\wL))$ is locally constant on $\wG$, whereas
		such an $F$ arises from the right hand side of the equality in eq. \eqref{Eq:easy crucial obs} if and only if for each
		$v \in V$, $F(-,v)$  is a locally constant map from $\wG$ to $\cH(\wL)$. Thus the functions $F$ which arise from the left hand side of equation \eqref{Eq:easy crucial obs} are always contained in those which arise from the right hand side, and the converse holds if  $V$ is finitely generated over $\wL$,
		say by $v_1,\cdots, v_n$. Then assuming that each of the functions $F(-,v_i)$ are constant in a neighborhood $U(g_0)$
		of a fixed $g_0\in \wG$, then by equation 1. above,  $F(g,\ell v_i) = \ell F(g, v_i)=\ell F(g_0, v_i)$
		for all $g \in U(g_0)$, $\ell \in \wL$,
		hence $F(g,v) = F(g_0,v)$ for all $g \in U(g_0)$, $v \in V$. Now the isomorphism \eqref{Eq:easy crucial obs} (applied to
		$\wP^-$ in place of $\wP$) together with \eqref{Eq:ind of H(L)} proves the isomorphism \eqref{Eq:ind Hom}, and hence the isomorphism
		in part 1. of the proposition.

		Similarly for any finitely generated smooth, projective, representation $W$ of $\wG$ we have 
		\[ \bfr_\wP\Hom_\wG(W,\cH(\wG))\simeq \Hom_\wG(W,
		(\id\times\bfr_\wP)(\cH(\wG)),\]
		as $\wL$-modules and naturally in $W$. This assertion is equivalent to proving that the normalized Jacquet module of
		$\Hom_\wG(W,\cH(\wG))$ (considered as a right $G$-module) is the same as  $\delta^{1/2}\Hom_\wG(W,\cC^\infty_c(\wG/N) )$ which is equivalent to proving that:
		\begin{enumerate}
			\item the natural map from 
			$\Hom_\wG(W,\cH(\wG))$ to   $\Hom_\wG(W,\cC^\infty_c(\wG/N) )$ is surjective. This is a consequence of $W$ being projective.
			\item The kernel of the natural map $\Hom_\wG(W,\cH(\wG)) \to \Hom_\wG(W,\cC^\infty_c(\wG/N) )$ consists of
			$\Hom_\wG(W,\cH(\wG)) [N]$
			where for any smooth representation ${\mathcal W}$ of $N$,
			\[{\mathcal W}[N] = \{n\cdot w -w| w \in{\mathcal W}\}
			=  \{v \in {\mathcal W}| \int_{N_i}n\cdot v = 0 \}\] where $N_i$ is some compact open subgroup of $ N$ depending on
			$v \in {\mathcal W}$. It follows that
			\[\Hom_\wG(W,\cH(\wG)) [N] = \Hom_\wG(W,\cH(\wG) [N]),\]
			using that $ \Hom_\wG(W,\cH(\wG))$ is a smooth representation of $N$ which is the case as $W$ is a finitely generated $\wG$-module.
		\end{enumerate}
		
		Applying the functor $\Hom_\wG(W,-)$ to the exact sequence
		\[ 0 \to \cH(\wG)[N] \to \cH(\wG) \to \cC^\infty_c(\wG/N) \to 0,\]
		and noting that $W$ is projective, we have the exact sequence
		\[ 0 \to  \Hom_\wG(W,\cH(\wG) [N]) \to  \Hom_\wG(W,\cH(\wG)) \to \Hom_\wG(W, \cC^\infty_c(\wG/N)) \to 0,\]
		proving the assertion that the Jacquet module of
		$\Hom_\wG(W,\cH(\wG))$, considered as a right $\wG$-module, is the same as  $\Hom_\wG(W,\cC^\infty_c(\wG/N) )$.
		This completes the proof of the isomorphism
		in \eqref{Eq:res Hom}, and hence part 2. of the proposition.
	\end{proof}

	\subsection{On projective generators}\label{SS:projective generators and D_h}
	In order to better understand the homological duality functor $D_h$, we will compute its value on the projective generators $\Pi_{[\rho]}\in\cM(\wG)_{[\rho]}$ and $\Pi_\fs\in \cM(\wG)_\fs$ from \cref{S:Blocks as module cats}.
	The answer is as nice as it could possibly be.
	Then we will use this computation to describe the functors $(-)^\vee$ and $D_h$ on the $\cR_\fs$-side: we obtain the contragredient, resp. homological, duality for $\cR_\fs$.
	
	Recall that if $\rho\in\cM(\wG)$ is irreducible cuspidal, then $\Pi_\rho$ was defined to be $\ind_{\wGo}^\wG(\rho|_\wGo)$.
	Since a cuspidal representation is compact modulo center by Harish-Chandra's theorem, the following is a restatement of \cref{C:full cuspidal homological dual of it}:
	\begin{lemma}\label{L:homological dual of progen Pi_rho}
		If $\rho\in\cM(\wG)$ is irreducible cuspidal, then
		\[ D_h(\Pi_{[\rho]})\simeq \Pi_{[\rho^\vee]}.\]
	\end{lemma}
	
	Let $\fs = (\wL,\rho)\in\cB(\wG)$ be a cuspidal datum (up to inertia and conjugation). 
	The projective generator $\Pi_\fs\in\cM(\wG)_\fs$ was defined in \cref{SS:induced proj gen} as 
	\[ \Pi_\fs = \oplus_{\wP} \bfi_{\wL,\wP}^\wG \Pi_{[\rho]}.\]
	We denote by $\fs^\vee = (\wL,\rho^\vee)$ the contragredient cuspidal datum (it could be the same as $\fs$ but in general it is not).
	\begin{proposition}\label{P:homological dual of progren Pi_s}
		With the above notation we have
		$D_h(\Pi_\fs)\simeq \Pi_{\fs^\vee}$.
	\end{proposition}
	\begin{proof}
		Use the previous lemma together with $D_h \bfi_{\wL,\wP}^\wG\simeq \bfi_{\wL,\wP^-}^\wG D_h$ (see \cref{P:D_h an ind res}) to get the following isomorphisms
		\begin{align*}
			D_h(\Pi_\fs) & = \oplus_{\wP} D_h(\bfi_{\wL,\wP}^\wG(\Pi_{[\rho]}))\\
			&\simeq \oplus_{\wP} \bfi_{\wL,\wP^-} D_h(\Pi_{[\rho]})\\
			& \simeq \oplus_{\wP}\bfi_{\wL,\wP^-}\Pi_{[\rho^\vee]}\\
			& = \Pi_{\fs^\vee}.\qedhere
		\end{align*}
	\end{proof}

	\begin{corollary}\label{C:D_h sends block to its contragredient}
		Homological duality restricts to a functor 
		\[D_h\colon \cD^b(\cM(\wG)_\fs)\to \cD^b(\cM(\wG)_{\fs^\vee})^\op\]
		which is moreover involutive when restricted to finitely generated modules.
	\end{corollary}
	\begin{proof}
		The functor $D_h$ sends the projective generator $\Pi_\fs$ of $\cM(\wG)_\fs$ to the projective generator $\Pi_{\fs^\vee}$ of $\cM(\wG)_{\fs^\vee}$.
		The duality statement for finitely generated follows from \cref{C:D_h square to id for fin.h.dim}.
	\end{proof}
	Thanks to the vanishing result for finite length modules (\cref{T:homological duality single degree}), we can deduce
	\begin{corollary}\label{C:subp main thm:involut} 
		The functor $D_h$ restricts to an involution 
		\[ \bD_h\colon \cM(\wG)_\fs^\finl \to (\cM(\wG)_{\fs^\vee}^\finl)^\op.\]
	\end{corollary}
	
	\begin{cor}\label{C:R_s anti-iso to R_svee}
		The homological duality induces an anti-isomorphism of algebras
		\[ \cR_\fs\simeq \cR_{\fs^\vee},\]
		or equivalently, an isomorphism of algebras $\cR_\fs^\op\simeq \cR_{\fs^\vee}$.
	\end{cor}
	\begin{proof}
		Since $D_h^2\simeq\Id$ (see \cref{C:D_h square to id for fin.h.dim}) we see that in particular $D_h$ is an equivalence of categories (contravariant).
		Hence it induces an anti-isomorphism of algebras:
		\[\cR_\fs = \Hom_\wG(\Pi_\fs,\Pi_\fs) \simeq \Hom_\wG(D_h\Pi_\fs,D_h\Pi_\fs) = \cR_{\fs^\vee}\]
		where in the last equality we used \cref{P:homological dual of progren Pi_s}.
	\end{proof}
	
	The above corollary means that we have an equivalence of categories $\cR_\fs\lmod\simeq \rmod\cR_{\fs^\vee}$.
	In other words, given a right module $V\in\rmod\cR_\fs$  its dual $V^*=\Hom_\bC(V,\bC)$ is naturally a left $\cR_\fs$-module, hence we can view it as a right $\cR_{\fs^\vee}$-module.
	Put shortly, taking the dual vector spaces gives us a functor
	\[ (-)^*\colon \rmod\cR_\fs\to \rmod\cR_{\fs^\vee}.\]
	
	We are now ready to describe the contragredient on the $\cR_\fs$-module side.
	The answer is not surprising at all:
	\begin{proposition}\label{P:contragredient on the R_s side}
		There is a commutative square of functors
		\[ 	\begin{tikzcd}
			\cM(\wG)_{\fs} \ar[r,"\sim"] \ar[d,"(-)^\vee"'] & \rmod \cR_{\fs} \arrow[d,"(-)^*"]\\
			(\cM(\wG)_{\fs^\vee})^\op \ar[r,"\sim"] & (\rmod \cR_{\fs^\vee})^\op.
		\end{tikzcd}\]
	\end{proposition}
	\begin{proof} 	
		The fact that the vertical right arrow makes sense was discussed above as a consequence of \cref{C:R_s anti-iso to R_svee}
		
		Note that $\Pi_{\fs}$ being projective and finitely generated, together with the calculation from \cref{P:homological dual of progren Pi_s}, implies that we have  natural isomorphisms of functors
		\[ \Hom_{\wG}(\Pi_{\fs},-)\simeq \Hom_\wG(\Pi_\fs,\cH(\wG))\otimes_{\cH(\wG)}- \simeq \Pi_{\fs^\vee}\otimes_{\cH(\wG)} - .\]		
		Let $V\in\cM(\wG)_\fs$ and use tensor-hom adjunction to obtain natural isomorphisms
		\begin{align*}
			\Hom_\wG(\Pi_{\fs}, V)^* & \simeq \Hom(\Pi_{\fs^\vee}\otimes_{\cH(\wG)} V, \bC)\\
			& \simeq \Hom_\wG(\Pi_{\fs^\vee},V^*)\\
			& \simeq \Hom_\wG(\Pi_{\fs^\vee},V^\vee)
		\end{align*}
		where the last equality follows because the image of a smooth module is smooth.
	\end{proof}
	
	As a final computation, we show that the homological duality on the $\cM(\wG)_\fs$ side goes to homological duality on the $\cR_\fs$-module side:
	\begin{proposition}\label{P:homolog duality on the R_s side}
		There is a commutative square of functors
		\[ 	\begin{tikzcd}
			\cD^b(\cM(\wG)_{\fs}) \ar[r,"\sim"] \ar[d,"D_h"'] & \cD^b(\rmod \cR_{\fs}) \arrow[d,"D_h"]\\
			\cD^b(\cM(\wG)_{\fs^\vee})^\op \ar[r,"\sim"] & \cD^b(\rmod \cR_{\fs^\vee})^\op
		\end{tikzcd}\]
		where the $D_h$ on the right is $\RHom_{\cR_\fs}(-,\cR_\fs)$ and we again use the algebra anti-isomorphism $\cR_\fs \simeq \cR_{\fs^\vee}$ from \cref{C:R_s anti-iso to R_svee}.
	\end{proposition}
	\begin{proof}
		We deal first with the left-bottom composition.
		Since $D_h$ is involutive (see \cref{C:D_h square to id for fin.h.dim}), it induces an isomorphism on Hom spaces. 
		In particular, using \cref{P:homological dual of progren Pi_s}, we get an isomorphism of left $\cR_\fs$-modules
		\[ \RHom_\wG(\Pi_{\fs^\vee}, D_h(V)) = \RHom_\wG(V,\Pi_\fs)\]
		for all $V\in\cD^b(\cM(\wG)_\fs)$.
		
		On the other hand, let us recall that the equivalence $\cM(\wG)_\fs\to \rmod\cR_\fs$ was given by
		the functor $\Hom_\wG(\Pi_\fs,-)$.
		The tensor-Hom adjunction gives its left adjoint as $-\otimes_{\cR_\fs} \Pi_\fs$.
		Being an equivalence implies the adjunction is a bi-adjunction, i.e.,  we also have that $\Hom_{\wG}(\Pi_\fs,-)$ is left adjoint to $-\otimes_{\cR_\fs} \Pi_\fs$.
		
		We can now compute the right-top composition and conclude the proof:
		\begin{align*}
			\RHom_{\cR_\fs}(\RHom_\wG(\Pi_\fs,V),\cR_{\fs})  &\simeq \RHom_\wG(V,\cR_{\fs}\otimes_{\cR_\fs} \Pi_\fs)\\
			&\simeq \RHom_\wG(V,\Pi_\fs),
		\end{align*}
		natural isomorphisms of left $\cR_\fs$-modules.
	\end{proof}
	
	\begin{remark}
		Using a similar argument and Morita theory (see for example \cite[4.11 Theorem 2, Corollary 2]{Pareigis-categories}) one shows that 
		under an equivalence of module categories $R\lmod \simeq
                S\lmod$, the $R$-bimodule $R$ is sent to the $S$-bimodule $S$, and as a consequence, the  homological duality for $R$ corresponds to the homological duality for $S$. 
	\end{remark}
	
	As a consequence of this section, in order to study homological duality on $\cM(\wG)$ we can study it on $\rmod \cR_\fs$ for all the cuspidal data $\fs$.
	This will play an important role in \cref{S:dualities on finite length} where we show that homological duality on finite length cuspidals is isomorphic to the contragredient (up to a shift).
	
	\section{The duality theorem of Schneider--Stuhler}\label{S:duality SchSt}
	
	We will use the duality theorem from \cref{S:abstract duality theorem}, applied to the Hecke algebra $\cH(\wG)$,
	to deduce a theorem of Schneider and Stuhler  \cite[Duality theorem]{SchStu}.
	In loc.cit., the result is stated and proved in the subcategory of modules with central character but this restriction was removed in \cite{NoriPras}. 
	Our proof is independent both of \cite{SchStu} and of \cite{NoriPras}, and it, moreover, provides a generalization of the main result of \cite{NoriPras} from irreducible modules to modules of finite length.
	The approach that we take was suggested in \cite[\S 3.4]{BBK}.

	Let $\fs=[\rho,\wL]\in\cB(\wG)$ be a cuspidal datum and put $d=d(\fs)$ the split rank of the center of $\wL$.
	We denote the contragredient cuspidal datum by $\fs^\vee = [\rho^\vee,\wL]$ for which we have $d(\fs^\vee)=d$.
	
	Recall that in \cref{T:homological duality single degree} it was proved that, for $\pi\in\cM(\wG)_\fs^\finl$, we have $H^i(D_h(\pi))\neq0$ only for $i=d$. 
	We put $\bD_h(\pi) = H^d(D_h(\pi))$.
	
	\begin{theorem}\label{T:main duality Schneider-Stuhler}
		Let $\pi\in\cM(\wG)_\fs^\finl$ be a finite length representation in the Bernstein block $\fs$ and let $\pi'\in\cM(\wG)$ be any smooth representation.
		Then the natural pairing
			\[ \Ext^{i}_\wG(\pi, \pi')  \times \Ext^{d-i}_\wG(\pi', \bD_h(\pi)^\vee)  \rightarrow 
			\Ext^{d}_\wG(\pi, D_h(\pi)^\vee)   \to \bC\] 
			provides an isomorphism
			\[ \Ext^{i}_\wG(\pi, \pi')^*\simeq  \Ext^{d-i}_\wG(\pi', \bD_h(\pi)^\vee).\]
			
	 If $\pi$ is, moreover, irreducible, then $\Ext^d_\wG(\pi,\bD_h(\pi)^\vee)\simeq \bC$.	
	\end{theorem}
	\begin{proof}
		The first part is just a reformulation of \cref{C:natural pairing RHom is perfect}.
		The second part comes from \cref{C:Hom M into DNak(M) is one dim} by noticing that $D_\Nak(\pi) = \bD_h(\pi)^\vee[d]$.
		
	\end{proof}

	\begin{remark}
		Note that $(\cdot)^\vee\circ \bD_h$ is an endofunctor of the category $\cM(\wG)_\fs^\finl$ of finite length modules.	
		We would like to think that $\bD_h$ commutes with the contragredient on finite length modules but we are not able to prove it.
		Another way to reformulate it is to say that $(-)^\vee\circ \bD_h$ is an involution on $\cM(\wG)_\fs^\finl$.
		\cref{C:D_h on finlen cusp is contrag} confirms this for finite length \emph{cuspidal} representations.
		More generally, $D_h$ should commute with Grothendieck--Serre duality (over the Bernstein center) for all finitely generated $\wG$-modules.\footnote{This was suggested to us by Roman Bezrukavnikov.} 
		This will be investigated in future work.
	\end{remark}

	\section{Dualities on finite length representations}\label{S:dualities on finite length}
	In this section we  prove that the homological duality restricted to finite length cuspidal
        representations is nothing else but contragredient duality up to a shift
        (see \cref{C:D_h on finlen cusp is contrag}). This is due to the fact that the condition (FsG) in Definition \ref{FsG}  is satisfied on a cuspidal block (see \cref{C:Frob sym cond for R_rho}).
	Along the way we will show the well-known fact that Grothendieck--Serre duality over the Bernstein center, restricted to finite length modules, is just the contragredient duality (see \cref{C:D_GS on admissibles is contrag}).
	We make use of the generalities developed in \cref{SS:general algebra}.
		
	Recall the Bernstein decomposition (\cref{T:Bernstein dec for tildeG}): 
	the category $\cM(\wG)$ of smooth representations of $\wG$ 
        decomposes into 
	blocks $\cM(\wG)\simeq \prod_\fs \cM(\wG)_\fs$, 
	and for each block,  we have $\cM(\wG)_\fs \simeq \rmod \cR_\fs$
        (see \cref{P:block equiv to modules over algebra}).
	
	Moreover, we have shown (see \cref{P:contragredient on the R_s side} and \cref{P:homolog duality on the R_s side}) that under the equivalence $\cM(\wG)_\fs \simeq \rmod \cR_\fs$,
        the contragredient on $\cM(\wG)_\fs$ corresponds to taking dual vector spaces on $\rmod\cR_\fs$ and that homological duality on $\cM(\wG)_\fs$ corresponds to homological duality on $\rmod\cR_\fs$.
	
	Recall from  \cref{SS:center of a block} that the center of the block $\cM(\wG)_\fs$ is $\cZ_\fs:=\cZ(\cR_\fs)$ and that $\cR_\fs$ is finite over $\cZ_\fs$.
	From the block decomposition, the (Bernstein) center of the category $\cM(\wG)$ satisfies
	\[ \cZ = \prod_\fs \cZ_\fs.\]
	Moreover, we know that each $\cZ_\fs$ is the ring of invariants of a finite group acting on a Laurent polynomial ring (see \cref{T:center of M(G)_s as invariants}), therefore it is a Cohen-Macaulay ring, and has a dualizing complex $\omega_{\cZ_\fs}^\circ$.
	The dualizing complex for the ring $\cZ$ is then 
	\[\omega_\cZ^\circ = \prod_\fs \omega_{\cZ_\fs}^\circ.\]

	The following is the relative version of \cref{D:D_GS for commutative k-algebras} and the boundedness of the derived categories is justified by the fact that $\omega_\cZ^\circ$ has finite injective dimension (see \cref{R:fin.inj.dim}):
	\begin{definition}
	The Grothendieck--Serre duality relative to the center is defined by
	\[ D_{GS/\cZ}
        =\RHom_\cZ(-,\omega^\circ_\cZ)^\sm\colon \cD^b_{fg}(\cM(\wG))\to \cD^b_{fg}(\cM(\wG))^\op \]
	where the superscript ``sm''  signifies taking smooth vectors.
	\end{definition}
	
	\begin{corollary} \label{C:D_GS on admissibles is contrag}
		The functor $D_{GS/\cZ}$ restricted to admissible modules is isomorphic to the contragredient functor.
		In particular, this holds for all finite length $\wG$-modules.
	\end{corollary}

\begin{proof}
		Let $V$ be an admissible module for $\wG$.
		For any open compact subgroup $\wK \subset \wG$,  $V^\wK$ is a finite dimensional module over the $\wK$-biinvariant Hecke algebra $\cH(\wG\git \wK)$  and it is also a module over the center $\cZ$.
		
		We have the following natural isomorphisms:
		\begin{align*}
			\RHom_\cZ(V,\omega_Z^\circ) & = \RHom_\cZ(\bigcup_\wK V^\wK, \omega_Z^\circ)&\\
			& = \varprojlim_\wK \RHom_\cZ(V^\wK,\omega_Z^\circ) &\\
			& = \varprojlim_\wK (V^\wK)^*& \text{ by \cref{C:D_GS/A on finite length R-mod}} \\
			& = V^*,&
		\end{align*}

and we conclude by taking smooth vectors.
	\end{proof}

	Grothendieck--Serre duality relative to the center behaves well under equivalences of categories.
	We record the following for completeness as an analogue of \cref{P:contragredient on the R_s side}	and \cref{P:homolog duality on the R_s side}.
\begin{proposition}\label{P:D_GS under equiv to R_s-mod}		
		Let $\fs=[\wL,\rho]\in\cB(\wG)$ be a cuspidal datum.
		Then, under the equivalence $\cM(\wG)_\fs\simeq \rmod \cR_\fs$, we have the following commutative square of functors
			\[ 	\begin{tikzcd}
			\cD^b(\cM(\wG)_{\fs}) \ar[r,"\sim"] \ar[d,"D_{GS}"'] & \cD^b(\rmod \cR_{\fs}) \arrow[d,"D_{GS}"]\\
			\cD^b(\cM(\wG)_{\fs^\vee})^\op \ar[r,"\sim"] & \cD^b(\rmod \cR_{\fs^\vee})^\op.
		\end{tikzcd}\]
\end{proposition}
\begin{proof}
	Under an equivalence of categories, the centers are isomorphic. 
	In this proof we put $Z:=\cZ_\fs$.
	So the two $D_{GS}$ functors in the diagram are both $\RHom_Z(-,\omega_{Z}^\circ)$.
	
	\cref{C:D_GS on admissibles is contrag} ensures that the target of the left vertical arrow is correct.
	Similarly, use \cref{C:D_GS/A on finite length R-mod} together with \cref{C:R_s anti-iso to R_svee} for the right vertical arrow.
	
	We have to show that, for $V\in\cM(\wG)$, there are natural isomorphisms 
	\begin{align}\label{Eq:D_GS commutes with equiv}
		\RHom_Z(\Hom_\wG(\Pi_\fs,V),\omega_Z^\circ) \simeq \Hom_\wG (\Pi_{\fs^\vee},\RHom_Z(V,\omega_Z^\circ)).
	\end{align}
	Notice that $\Pi_\fs$ is projective, so we don't need to derive the Hom space from it.
	
	Using \cref{L:Hom(M N)=Hom(M A)otimesN} and \cref{P:homological dual of progren Pi_s} we have $\Hom_\wG(\Pi_\fs,V)\simeq \Pi_{\fs^\vee}\otimes_{\cH(\wG)} V$.
	We can rewrite the left hand-side of \eqref{Eq:D_GS commutes with equiv} and use tensor-hom adjunction to conclude.
\end{proof}

	Let $\rho\in\cM(\wG)$ be an irreducible cuspidal representation and put as usual $d=\rank_\bZ(\wG/\wGo)$.
	It has already been proved that homological duality restricts to a contravariant functor
	\[ \bD_h:=[d]\circ D_h\colon \cM(\wG)_{[\rho]}^\finl \to (\cM(\wG)_{[\rho^\vee]}^\finl)^\op,\]
	see \cref{C:D_h sends block to its contragredient} and \cref{T:homological duality single degree}.
	We can see that it is not new:
	\begin{corollary}\label{C:D_h on finlen cusp is contrag}
		For $\rho\in\cM(\wG)$ an irreducible cuspidal representation of $\wG$, and  $d=\rank_\bZ(\wG/\wGo)$,
the above functor $\bD_h$ is isomorphic to the contragredient.
	\end{corollary}
	\begin{proof}
		The block $\cM(\wG)_{[\rho]}$ is equivalent to the category $\rmod\cR_{[\rho]}$ and this equivalence commutes with homological duality (see \cref{P:homolog duality on the R_s side}) as well as with contragredient (see \cref{P:contragredient on the R_s side}).
		Moreover, it sends finite length $\wG$-modules to finite length $\cR_{[\rho]}$-modules.
		In order to prove the claim it is therefore enough to do it for $\rmod\cR_{[\rho]}$.
		
		The center of $\cR_{[\rho]}$ is a Laurent polynomial algebra of dimension $d$ (see \cref{P:R_rho is Azumaya and Z Laur pol}) and moreover $\cR_{[\rho]}$ satisfies the condition \eqref{Cond: Hom(R,A)=R as bimodules} (a particular case of
                the condition (FsG) in Definition \ref{FsG}) by \cref{C:Frob sym cond for R_rho}.
		Applying \cref{C:D_h=D_h/A=D_GS=(-)^* on finite length if (FsG) satisfied} we conclude.
	\end{proof}
	The same argument gives the following corollary.
	\begin{corollary}\label{C:D_h=D_GS on all cuspidal block} Let $\rho\in\cM(\wG)$ be an irreducible cuspidal representation and $d=\rank_\bZ(\wG/\wGo)$.
		Then on the cuspidal block $\cD^b_{fg}(\cM(\wG)_{[\rho]})$, we have a natural isomorphism of functors
		\[ [d]\circ D_h\simeq D_{GS/\cZ}.\]
	\end{corollary}
	
	\begin{remark} 
		If $G$ is a torus  then all smooth representations of a finite central extension $\wG$ are cuspidal and therefore the above corollaries apply  to \emph{all} finite length smooth $\wG$-modules.
		Moreover, in this situation the Grothendieck--Serre duality and homological duality agree (up to shift) on all smooth modules.
	\end{remark}
	
	More generally, the following gives a necessary and sufficient condition for homological duality to be isomorphic to Grothendieck--Serre duality:	
	\begin{corollary}\label{C:D_h=D_GS if and only if}
		On a block $\cD^b(\cM(\wG)_\fs)$, we have that $[d(\fs)]\circ D_h\simeq  D_{GS/\cZ}$ if and only if $\cR_\fs$ satisfies condition (FsG) in  Definition \ref{FsG}.
	\end{corollary}
	\begin{proof}
	Given \cref{P:D_GS under equiv to R_s-mod}, this is simply a restatement of \cref{C:D_h=D_GS if and only if FsG}.
	\end{proof}
	
	\begin{example}
	Let $\fs=[\wL,\rho]$ be a cuspidal datum such that  the stabilizer of the inertia class $[\rho]$ in the relative Weyl group is trivial, i.e., there is no non-trivial $w\in W(\wL,\wG)$ such 	that ${}^w\rho \simeq \rho\chi$ for some unramified character $\chi$ of $L$.
	Then by a result of Roche \cite[Theorem 3.1]{Roche-parab}, parabolic induction induces an equivalence of blocks $\cM(L)_{[\rho]}\to\cM(G)_{\fs}$ and therefore a Morita equivalence of  $\cR_{[\rho]}$ and $\cR_{\fs}$.
	As $\cR_{[\rho]}$ is Azumaya by \cref{P:R_rho is Azumaya and Z Laur pol}, the same is true of $\cR_\fs$.
	By the general property of Azumaya algebras recalled in \cref{L:Azumaya are selfdual}, we find that the algebra $\cR_\fs$ satisfies condition  
	(FsG).
\end{example}

	\begin{example} \label{12.9} By \cref{C:D_GS on admissibles is contrag},
        the Grothendieck--Serre duality on irreducible representations of $\wG$ 
	is simply the contragredient. On the other hand, as it is well known, the homological duality 
	(same as Aubert-Zelevinsky duality for irreducible representations) interchanges the trivial representation and the Steinberg representation
	of $\SL_2(F)$. Therefore the Corollary above is not true for all representations of $\SL_2(F)$ which means that
the        condition (FsG) in Definition \ref{FsG} does not hold for the Iwahori block for $\SL_2(F)$.
	\end{example}
	
	Let us also note the particular case of a free abelian group $\Gamma\simeq \bZ^d$. (The
        abelian category of representations of $\Gamma$ is the Iwahori block of
        a split torus over any $p$-adic field, so in considering this example, we
        are not going outside the realm of $p$-adic groups!)
	The category of representations of $\Gamma$
        is the same as  the category of representations of the group algebra $\bC[\Gamma]\simeq \bC[X_1^\pm,\dots,X_d^\pm]$.
	From \cref{C:D_h=D_h/A=D_GS=(-)^* on finite length if (FsG) satisfied} we get:
	\begin{corollary}
		The Grothendieck--Serre duality, (shifted by $d$) homological duality and contragredient duality all agree on finite length $\Gamma$-modules:
		\[ D_{GS} = [d]\circ D_h = (-)^* \text{ as functors on } \cM(\Gamma)^\finl.\]
                Equivalently, for a finite length module $M$ over $A= \bC[X_1^\pm,\dots,X_d^\pm]$,
                \[ \Ext_A^d(M, A) \cong \Hom_\bC(M,\bC),\]
an isomorphism of $A$-modules,                an assertion which can also be made for  a finite length module $M$ over $A= \bC[X_1,\dots,X_d]$.
	\end{corollary}


	\vspace{.2cm}
		
	\noindent
	DF: IRMA Institut de Mathematique
	7 rue Ren\'e-Descartes, 67000 Strasbourg
	
	\noindent
	Email: {\tt fratila@math.unistra.fr}.
	\vspace{.2cm}
	
	\noindent DP:  Indian Institute of Technology Bombay, Mumbai. 
	
	\noindent Email: {\tt prasad.dipendra@gmail.com}
	
\end{document}